\documentclass[11pt]{article}
\usepackage{amsmath,latexsym,amssymb,amsfonts,amsbsy, amsthm}
\usepackage{graphicx}
\usepackage{empheq}
\usepackage{color}
\usepackage{ulem}
\def\inte#1{
\displaystyle\mathop{#1\kern0pt}^\circ }

\let\grad\nabla

\def\virgp{\raise 2pt\hbox{,}}
\def\cdotpv{\raise 2pt\hbox{;}}

\def\eqdefa{\buildrel\hbox{\footnotesize def}\over =}

\def\C{\mathop{\bf C\kern 0pt}\nolimits}
\def\DD{\mathop{\bf D\kern 0pt}\nolimits}
\def\K{\mathop{\bf K\kern 0pt}\nolimits}
\def\N{\mathop{\bf N\kern 0pt}\nolimits}
\def\Q{\mathop{\bf Q\kern 0pt}\nolimits}
\def\R{\mathop{\bf R\kern 0pt}\nolimits}
\def\SS{\mathop{\bf S\kern 0pt}\nolimits}
\def\ZZ{\mathop{\bf Z\kern 0pt}\nolimits}
\def\TT{\mathop{\bf T\kern 0pt}\nolimits}

\def\na{\nabla}

\newcommand{\beq}{\begin{equation}}
\newcommand{\eeq}{\end{equation}}
\newcommand{\ben}{\begin{eqnarray}}
\newcommand{\een}{\end{eqnarray}}
\newcommand{\beno}{\begin{eqnarray*}}
\newcommand{\eeno}{\end{eqnarray*}}
\newtheorem{defi}{Definition}[section]
\newtheorem{example}[defi]{Example}
\newtheorem{thm}[defi]{Theorem}
\newtheorem{lem}[defi]{Lemma}
\newtheorem{rmk}[defi]{Remark}
\newtheorem{col}[defi]{Corollary}
\newtheorem{prop}[defi]{Proposition}
\renewcommand{\theequation}{\thesection.\arabic{equation}}

\begin{document}

\title{ Linear instability of $Z$-pinch in plasma (I):  Inviscid case}
\author{Dongfen Bian \footnote{ School of Mathematics and Statistics,
 Beijing Institute of Technology, Beijing 100081, China; Division of Applied Mathematics, Brown University, Providence, Rhode Island 02912, USA. Email: {\tt dongfen\_bian@brown.edu/biandongfen@bit.edu.cn}.} \and
Yan Guo\footnote{Division of Applied Mathematics, Brown University, 
Providence, Rhode Island 02912, USA. Email: {\tt yan\_guo@brown.edu}.} \and
Ian Tice \footnote{Department of Mathematical Sciences, Carnegie Mellon University,  Pittsburgh, PA 15213, USA. Email: {\tt iantice@andrew.cmu.edu}.} }

\maketitle

\begin{abstract} 	
The z-pinch is a classical steady state for the MHD model, where a confined plasma
fluid is separated by vacuum, in the presence of a magnetic field which is generated
by a prescribed current along the z direction. We develop a variational framework to
study its stability in the absence of viscosity effect, and demonstrate for the first time that any such a z-pinch is always unstable. 
Moreover, we discover a sufficient condition such that the eigenvalues can be unbounded, which leads to ill-posedness of
the linearized MHD system.
\end{abstract}


\vskip 0.2cm


\renewcommand{\theequation}{\thesection.\arabic{equation}}
\setcounter{equation}{0}
\setcounter{tocdepth}{1}
\tableofcontents
\newpage
\section{Introduction}

\subsection{Formulation of the problem in Eulerian coordinates}

In the present paper, we are concerned with the plasma-vacuum MHD system, where the plasma is confined inside a rigid wall and isolated from it by a region of low enough density to be treated as a ``vacuum''. This model
describes confined plasmas in a closed vessel, but separated from the wall by a vacuum region. We consider the cylindrical domain $\big\{(x_1,x_2,z)\in \mathbb{R}^2\times 2\pi \mathbb{T}|x_1^2+x_2^2\leq r_w^2\big\}$, which is meant to model the container holding the plasma and the container divides into two disjoint pieces, $\Omega(t)=\big\{\rho(t)>0\big\}$ and $\Omega^v(t)=\big\{\rho(t)=0\big\}$, with the free boundary $\Sigma_{t, pv} \eqdefa \overline{\Omega(t)} \cap \overline{\Omega^v(t)}$  and the perfectly conducting wall $\Sigma_w$ on the outside $r_w$.
For smooth solutions, the compressible MHD system in the plasma region can be written in Eulerian coordinates as
\begin{equation}\label{mhd-plasma-1-viscosity}
\begin{cases}
&\partial_t \rho + (u\cdot \nabla) \rho + \rho \nabla\cdot u=0\quad \mbox{in} \quad \Omega(t),\\
&\rho(\partial_t u + (u\cdot \nabla) u) + \nabla (p+\frac{1}{2}|B|^2)=(B\cdot \nabla) B \quad \mbox{in} \quad \Omega(t),\\
&\partial_t B -\nabla \times (u \times B)=0 \quad \mbox{in} \quad \Omega(t),\\
&\nabla \cdot B=0\quad \mbox{in} \quad \Omega(t),
\end{cases}
\end{equation}
where the vector-field $u = (u_1, u_2,u_3)$ denotes the Eulerian plasma velocity field, $\rho$ denotes the density of the fluid, $B=(B_1, B_2,B_3)$ is magnetic field, and $p$ denotes the pressure function. The above system \eqref{mhd-plasma-1-viscosity} is called the inviscid compressible MHD equations which describe the motion of a perfectly conducting fluid interacting with a magnetic field.
Here, the open, bounded subset $\Omega(t) \subset \mathbb{R}^3$ denotes the changing volume occupied by
the plasma with  $\rho(t)>0$ in $\Omega(t)$. 
We have here considered the polytropic gases, the constitutive relation, which is also called the
equation of state, and is given by $p=A\,\rho^{\gamma}$, where $A$ is an entropy constant and $\gamma>1$ is the adiabatic gas exponent.

From the mass conservation equation in \eqref{mhd-plasma-1-viscosity} and pressure satisfying $\gamma $ law, one can get that 
\begin{equation}\label{pressure}
\partial_t p + u\cdot \nabla p + \gamma p \nabla\cdot u=0.
\end{equation}

In the vacuum domain $\Omega^v(t)$, we have the div-curl system
\begin{equation}\label{div-curl-1}
\begin{cases}
& \nabla \cdot \widehat{B}=0 \quad \mbox{in} \quad \Omega^v(t),\\
& \nabla \times \widehat{B}=0 \quad \mbox{in} \quad \Omega^v(t)
\end{cases}
\end{equation}
which describes the vacuum magnetic field $\widehat{B}$. Here, we consider so-called pre-Maxwell dynamics. That is, as usual in nonrelativistic
MHD, we neglect the displacement current $\frac{1}{c^2} \partial_t \widehat{E}$, where $c$ is the speed of the light
and $\widehat{E}$ is the electric field. In general, quantities with a hat $\widehat{\cdot}$ denote vacuum variables.

We assume that the plasma region $\Omega(t)$ with the fluid density $\rho(t)>0$ is isolated from the fixed perfectly conducting wall $\Sigma_w$ by a vacuum region $\Omega^v(t)$, which makes the plasma surface free to move.  Hence, this model is a free boundary problem of the combined plasma-vacuum system.
To solve this system, we need to prescribe appropriate boundary conditions.
On the perfectly conducting wall $\Sigma_w$, the normal component of the magnetic field must vanish:
\begin{equation}\label{wall-bdry-vacuum-1}
n\cdot \widehat{B}|_{\Sigma_w}=0,
\end{equation}
where $n$ is the outer unit normal to the boundary of $ \Sigma_w$.

We  prescribe the following jump conditions on the free boundary to connect the magnetic fields across the surface. These arise
from the Maxwell's equations and the continuum mechanics 
\begin{equation}\label{bdry-plasma-vacuum-1}
\begin{cases}
& n \cdot B=n \cdot \widehat{B} \quad \mbox{on} \quad \Sigma_{t, pv},\\
& \bigg[\bigg[p+\frac{1}{2}|B|^2\bigg]\bigg] =0 \quad \mbox{on} \quad \Sigma_{t, pv},
\end{cases}
\end{equation}
where  for any quantity $q$, $[[q]]_{\Sigma_{t, pv}}$ denotes $\widehat{q}-q$ on the free boundary $\Sigma_{t, pv}$, and $n $ is the outer normal to the free boundary of $ \Omega(t)$.

In conclusion, denote $\mathcal{V}(\Sigma_{t, pv})$ as the normal velocity of the free surface $\Sigma_{t, pv}$,  then the plasma-vacuum compressible MHD system can be written in Eulerian coordinates as
\begin{equation}\label{mhd-plasma-E-1-viscosity}
\begin{cases}
&\partial_t \rho + \nabla\cdot( \rho \, u)=0\quad \mbox{in} \quad \Omega(t),\\
&\rho(\partial_t u + (u\cdot \nabla) u) + \nabla (p+\frac{1}{2}|B|^2)=(B\cdot  \nabla) B \quad \mbox{in} \quad \Omega(t),\\
&\partial_t B - \nabla \times (u \times B)=0, \quad  \nabla \cdot B=0 \quad \mbox{in} \quad \Omega(t),\\
&\nabla \cdot B=0\quad \mbox{in} \quad \Omega(t),
\\
& \nabla \cdot \widehat{B}=0, \quad \nabla \, \times \, \widehat{B}=0 \quad \mbox{in} \quad \Omega^v(t),\\
&\mathcal{V}(\Sigma_{t, pv})=u\cdot n \quad \mbox{on} \quad \Sigma_{t, pv},\\
& n\cdot B=n\cdot \widehat{B}\quad \mbox{on} \quad \Sigma_{t, pv},\\
& p+\frac{1}{2}|B|^2-\frac{1}{2}|\widehat{B}|^2=0 \quad \mbox{on} \quad \Sigma_{t, pv},\\
&n\cdot \widehat{B}|_{\Sigma_w}=0,\, \rho|_{t=0}=\rho_0, \, u|_{t=0}=u_0, \, B|_{t=0}=B_0.
\end{cases}
\end{equation}

\subsection{Background}
 The  $z$-pinch instability  in plasma for the compressible MHD system \eqref{mhd-plasma-E-1-viscosity} with vacuum and free boundary is an interesting and long-time open problem since the pinch experiments of the 1960s and 1970s, see \cite{Mikhailovshii, Schmit} and the references therein.  There are many numerical simulations \cite{Freidberg,Goedbloed-poedts}. 
 Recently, Guo-Tice \cite{Guo-Tice-inviscid} and \cite{Guo-TIce-viscous} proved the 
 linear Rayleigh-Taylor instability for inviscid and viscous compressible fluids by introducing a new variational method.  Later on, using the variational framework, many authors considered the effects of magnetic field in the fluid equations. Jiang-Jiang \cite{jiang-jiang-ARMA} considered the magnetic inhibition theory
 in non-resistive incompressible MHD fluids. Jiang-Jiang \cite{jiang-jiang-CV} considered the nonlinear stability and instability in the Rayleigh-Taylor problem of compressible MHD equations without  vacuum and established the stability/instability criteria for the stratified compressible magnetic Rayleigh-Taylor problem in Lagrangian coordinates. Jiang-Jiang \cite{jiang-jiang-PD}  investigated the stability and instability of the Parker problem for the three-dimensional compressible isentropic viscous magnetohy-drodynamic system with zero resistivity in the presence of a modified gravitational force in a vertical strip domain in which the velocity of the fluid is non-slip on the boundary.   Wang-Xin \cite{wang-xin} proved the global well-posedness of the inviscid and resistive problem with surface tension around a non-horizontal uniform magnetic field for two-dimensional incompressible MHD equations. Wang \cite{wang} got sharp nonlinear stability criterion of viscous incompressible non-resistive MHD internal waves  in 3D. Gui \cite{Gui} considered the Cauchy problem of the two-dimensional incomplressible magnetohydrodynamics system with inhomogeneous density and electrical conductivity and has showed the global well-posedness for a generic family of the variations of the initial data and an inhomogeneous electrical conductivity.  All these results do not contain vacuum. For presenting vacuum, under the Taylor sign condition of the total pressure on the free surface, Gu-Wang \cite{Gu-Wang}  proved the local well-posedness of the ideal incompressible MHD equations in Sobolev spaces. 
 In this paper, we will rigirously prove the linear $z$-pinch instability for ideal compressible MHD system \eqref{mhd-plasma-E-1-viscosity}.  

\section{Steady state and Main results }

\subsection{Derivation of the MHD system  in Lagrangian coordinates}\label{susect-free-surface-1}
In this subsection, we mainly  introduce the Lagrangian coordinates in which the free boundary becomes fixed.

First, we assume the equilibrium domains are given by
\begin{equation*}
\overline{\Omega}=\{(r, \theta, z)| r<r_0, \, \theta\in[0,2\pi],z\in 2\pi \mathbb{T}\}\,\, \mbox{and}\,\,
\overline{\Omega}^v=\{(r,\theta, z)|r_0<r<r_w, \,  \theta\in[0,2\pi],z\in 2\pi \mathbb{T}\}.
\end{equation*} 
 Here, the constant $r_0$ is the interface boundary and the constant $r_w$ is the perfectly conducting wall position.
This is meant to be a simplified model of the toroidal geometry employed in tokamaks.

Now we introduce the Lagrangian coordinates.

{\bf 1. The flow map}

Let $h(t, \mathcal{X})$ be a position of the gas particle $\mathcal{X}$  in the equilibrium domain $\overline{\Omega}$ at time $t$ so that
\begin{equation}\label{def-flowmap-1}
\begin{cases}
&\frac{d}{dt}h(t, \mathcal{X})=u(t, h(t, \mathcal{X})), \quad t>0, \, \mathcal{X}\in \overline{\Omega},\\
&h|_{t=0}=\mathcal{X}+g_0(\mathcal{X}), \quad \mathcal{X}\in \overline{\Omega}.
\end{cases}
\end{equation}

 Then the displacement $g(t, \mathcal{X})\eqdefa h(t, \mathcal{X})-\mathcal{X}$ satisfies
\begin{equation}\label{def-flowmap-2}
\begin{cases}
&\frac{d}{dt}g(t, \mathcal{X})=u(t, \mathcal{X}+g(t,  \mathcal{X})),\\
&g|_{t=0}=g_0.
\end{cases}
\end{equation}
We define the Lagrangian quantities in the plasma as follows (where $\mathcal{X}=(x, y, z)\in \overline{\Omega}$):
\begin{equation*}
\begin{split}
&f(t, \mathcal{X})\eqdefa \rho(t, h(t, \mathcal{X})),\quad v(t, \mathcal{X})\eqdefa u(t, h(t, \mathcal{X})),\quad q(t, \mathcal{X})\eqdefa p(t, h(t, \mathcal{X})), \\
& b(t, \mathcal{X})\eqdefa B(t, h(t, \mathcal{X})), \quad \mathcal{A}\eqdefa (Dh)^{-1}, \quad J \eqdefa \mbox{det}(Dh).
\end{split}
\end{equation*}
According to definitions of the flow map $h$ and the displacement $g$,   for $(i,j,k)\in \{1,2,3\}$ one can get the following identities
\begin{equation}\label{flow-map-identity-1}
\mathcal{A}_{i}^k \partial_{k} h^j=\mathcal{A}_{k}^j \partial_{i} h^k=\delta_i^j, \quad \partial_k(J\mathcal{A}_{i}^k)=0,\quad\partial_{i} h^j=\delta_i^j+\partial_{i} g^j, \quad \mathcal{A}_{i}^j=\delta_i^j-\mathcal{A}_{i}^k \partial_kg^j,
\end{equation}
where the Einstein notation is used and will be used in the whole paper.
If the displacement $g$ is sufficiently small in an appropriate Sobolev space, then the flow mapping $h$ is a
diffeomorphism from $\Omega_0$ to $\Omega(t)$, which allows us to switch back and forth from Lagrangian to Eulerian coordinates.

{\bf 2. Derivatives of $J$ and $\mathcal{A}$ in Lagrangian coordinates}

We write the derivatives of $J$ and $\mathcal{A}$ in Lagrangian coordinates as follows:
\begin{equation}\label{identity-Lagrangian-1}
\begin{split}
&\partial_t J =J \mathcal{A}_{i}^j \partial_j v^i, \quad \partial_{\ell} J =J \mathcal{A}_{i}^j \partial_j\partial_{\ell} g^i, \quad\partial_t \mathcal{A}_{i}^j=-\mathcal{A}_{k}^j\mathcal{A}_{i}^{\ell} \partial_{\ell}v^k,\\
&\partial_{\ell} \mathcal{A}_{i}^j=-\mathcal{A}_{k}^j\mathcal{A}_{i}^{n} \partial_{n}\partial_{\ell}g^k, \quad \partial_{i} v^j=\partial_{i}h^k \mathcal{A}_{k}^\ell \partial_{\ell}v^j= \mathcal{A}_{i}^\ell \partial_{\ell}v^j+\partial_{i}g^k \mathcal{A}_{k}^\ell \partial_{\ell}v^j.
\end{split}
\end{equation}

{\bf 3. Plasma equations in Lagrangian coordinates}

Denote $(\nabla_{\mathcal{A}})_i=\mathcal{A}_i^j \partial_j$. Then we can write the plasma equations in Lagrangian coordinates as follows
\begin{equation}\label{mhd-fluid-2}
\begin{cases}
& \partial_t g=v\quad \mbox{in}\quad  \overline{\Omega}, \\
& f \partial_t v + \nabla_{\mathcal{A}}\Big(q+\frac{1}{2}|b|^2\Big)=(b\cdot \nabla_{\mathcal{A}}) b \quad \mbox{in}\quad  \overline{\Omega},\\
&\partial_t f+ f \nabla_{\mathcal{A}} \cdot  v=0\quad \mbox{in}\quad  \overline{\Omega},\\
&\partial_t b+ b \nabla_{\mathcal{A}}\cdot v=(b\cdot \nabla_{\mathcal{A}}) v\quad \mbox{in}\quad  \overline{\Omega},\\
&\nabla_{\mathcal{A}} \cdot b=0\quad \mbox{in}\quad  \overline{\Omega},\\
&n\cdot b=n\cdot \widehat{  b}\quad\mbox{on} \quad \Sigma_{0, pv},\\
& q+\frac{1}{2}|b|^2-\frac{1}{2}|\widehat{b}|^2=0 \quad\mbox{on} \quad \Sigma_{0, pv},
\end{cases}
\end{equation}
where the exterior magnetic field $\widehat{  b}$ satisfies the vacuum equations \eqref{mhd-vacuum-Lag-2} in lagrangian coordinates which can be recalled from Appendix \ref{appendix}.

Since $\partial_t J =J \mathcal{A}_{i}^j \partial_j v^i=J\,\nabla_{\mathcal{A}} \cdot  v$ and $J(0)=\det(Dh_0)=\det(I+Dg_0)$, with $I$ the identity matix, we find from the equation of $f$ in \eqref{mhd-fluid-2} that
$f\,J= \rho_0(h_0)\det(I+Dg_0)$,
where $\rho_0$ is given initial density function. Taking $\rho_0$ such that
$\rho_0(h_0)\det(I+Dg_0)=\overline{\rho}$,
we get
\begin{equation}\label{f-q-1}
\begin{split}
f= J^{-1}\,\overline{\rho},\quad q=AJ^{-\gamma}\,\overline{\rho}^{\gamma}.
\end{split}
\end{equation}

On the other hand, we multiply  the magnetic field equation of \eqref{mhd-fluid-2} by  $J \mathcal{A}^T$ to get
\begin{equation*}
\begin{split}
J \mathcal{A}_j^i  \partial_t b^j+ J\mathcal{A}_j^i  b^j \mathcal{A}_k^h \partial_hv^k=J \mathcal{A}_j^i b^h \mathcal{A}_h^k\partial_k v^j,
\end{split}
\end{equation*}
which along with \eqref{identity-Lagrangian-1} implies
\begin{equation*}
\begin{split}
&\partial_t(\mathcal{A}_j^i J b^j) =J \mathcal{A}_j^i  \partial_t b^j+\mathcal{A}_j^i  b^j\partial_t J +J b^j\partial_t \mathcal{A}_j^i =J \mathcal{A}_j^i  \partial_t b^j+ J\mathcal{A}_j^i  b^j\mathcal{A}_k^h \partial_hv^k-J b^j\mathcal{A}_{k}^i\mathcal{A}_{j}^{h} \partial_{h}v^k=0.
\end{split}
\end{equation*}
Therefore, we have
\begin{equation}\label{b-identity-1}
\begin{split}
J b^j \mathcal{A}_j^i=J(0) b^j_0 \mathcal{A}_j^i(0)=\det(I+Dg_0)B^j_0(h_0) \mathcal{A}_j^i(0),
\end{split}
\end{equation}
where $B_0$ is given initial magnetic field. Taking $B_0$ such that
$\det(I+Dg_0)B^j_0(h_0) \mathcal{A}_j^i(0)=\overline{B}^i$,
we obtain from \eqref{b-identity-1} that
\begin{equation*}\label{b-identity-2}
\begin{split}
b^k =J^{-1}\overline{B}^i \partial_ih^k=J^{-1}\overline{B}^k+J^{-1}\overline{B}^i \partial_ig^k.
\end{split}
\end{equation*}
\subsection{The equilibrium for the $z$-pinch plasma}

	In this paper, our goal is to study the linear $z$-pinch instability for the compressible MHD equations \eqref{mhd-plasma-1-viscosity}. Therefore, we look for the cylindrically symmetric steady solution $\overline{ u}=0$, $ \overline{ B}=(0,\overline{ B}_\theta(r),0)$, $\overline{ p}=\overline{p}(r)$, $\widehat{  \overline{B}}=(0,\widehat{\overline{B}}_\theta(r),0)$. For notational simplicity, in the following we abuse notation to denote steady state $z$-pinch solutions as $$p(r)=\overline{ p}(r), \quad B_\theta(r)=\overline{ B}_\theta(r),\quad \widehat{ B}_\theta(r)=\widehat{  \overline{ B}}_\theta(r),$$
	which imply that $B=\overline{ B}$ and $\widehat{ B}=\widehat{  \overline{ B}}$.

Then we can get the following lemma describing the steady solution.
\begin{lem}\label{steady-lem}
Assume that the function $p(r)$ satisfies $p(r)\geq 0$ and $p(r)=0$ if and only if $r=r_0$, and 
\begin{equation}\label{new-condition}
-\int_0^rs^2p'(s)ds \geq 0 \,\, \mbox{ for all} \quad 0\leq r\leq r_0, \quad p(r)\in C^{2,1}([0,r_0]).
\end{equation}	
Then the cylindrically symmetric steady solution $\overline{ u}=0$, $ B=B(r)$, $\mathbb{J}_z=\mathbb{J}_z(r)$, $\widehat{  B}=\widehat{B}(r)$ with a function $p(r)$ taking the form of 
	\begin{equation}\label{z-pinch-1-1}
	\begin{cases}
&B_r= 0, \quad B_z= 0, \quad  B_\theta(r)=\Big(-\frac{2}{r^2}\int_0^rs^2 p'(s)ds\Big)^{\frac{1}{2}},
\\
&\mathbb{J}_z(r)
=\frac{1}{2}\Big(-\frac{2}{r^2}\int_0^rs^2 p'(s)ds\Big)^{-\frac{1}{2}}\Big(\frac{4}{r^3}\int_0^rs^2p'(s)ds-2p'(r)\Big)\\
&\quad\quad\quad+\frac{1}{r}\Big(-\frac{2}{r^2}\int_0^rs^2p'(s)ds\Big)^{\frac{1}{2}}\quad \mbox{in} \quad \overline{\Omega},\\
& \widehat{B}_r= 0, \quad \widehat{B}_z= 0, \quad \widehat{B}_{\theta}(r) = B_{\theta}(r_0)\frac{r_0}{r} \quad \mbox{in} \quad \overline{\Omega}^v,
\end{cases}
	\end{equation}
	solves the  equilibrium equations in plasma domain,
	\begin{equation}\label{steady-equ-plasma}
	\nabla p=\mathbb{J}\times B, \qquad\nabla \cdot  B=0,\qquad
	\mathbb{J}=\nabla\times B,
	\end{equation}
	and the system \eqref{div-curl-1} in the vacuum region. 
	We can define the equilibrium density
	\begin{equation*}
	\rho(r) = \left(\frac{p(r)}{A} \right)^{1/\gamma}.
	\end{equation*}
Moreover, 
 we have 
 \begin{equation}\label{cond-jz}
 \mathbb{J}_z\in C^{1,1}([0,r_0]),\quad
 B_\theta\in C^{1,1}([0,r_0]).
 \end{equation}
\end{lem}
\begin{proof}
	In cylindrical $ r$, $\theta$, $z$-coordinates, 
 the equilibrium equations \eqref{steady-equ-plasma} 
	which are equivalent to the system
	\begin{equation*}
	\nabla (p+\frac{1}{2}|B|^2)=(B\cdot \nabla) B, \qquad\nabla \cdot  B=0,\quad \mathbb{J}=\nabla\times  B,
	\end{equation*}
are	reduced to 
	\begin{equation}\label{equilibrium-ode}
	\frac{d}{dr}\Big(p(r)+\frac{1}{2}| B_{\theta}(r)|^2\Big)= -\frac{B_{\theta}^2(r)}{r}, \quad \frac{1}{r}\frac{d}{dr}(r B_r(r))=0,\quad
	\frac{1}{r}\frac{d}{dr}(rB_{\theta}(r)) =\mathbb{J}_{z}(r).
	\end{equation}
	The first equation of \eqref{equilibrium-ode} is equivalent to $p'=-\mathbb{J}_zB_\theta$, which together with the third equation of \eqref{equilibrium-ode} implies that $-r^2p'=rB_\theta(rB_\theta)'$. Set $C(r)=rB_\theta(r)$ to reduce this to $(C^2)'=-2r^2p'$. Integrating and forcing $B_\theta(0)$ to be finite, which gives $C(0) = 0$, we find that $ B_\theta(r)=\Big(-\frac{2}{r^2}\int_0^rs^2p'(s)ds\Big)^{\frac{1}{2}}$. Thus, given the pressure $p$, we can compute $B_\theta$. Then we define $
	\mathbb{J}_z$ by the third equation of \eqref{equilibrium-ode} that 
	\begin{equation*}
\mathbb{J}_z(r)
=\frac{1}{2}\Big(-\frac{2}{r^2}\int_0^rs^2 p'(s)ds\Big)^{-\frac{1}{2}}\Big(\frac{4}{r^3}\int_0^rs^2p'(s)ds-2p'(r)\Big)+\frac{1}{r}\Big(-\frac{2}{r^2}\int_0^rs^2p'(s)ds\Big)^{\frac{1}{2}}.
	\end{equation*}
Therefore,	solving the system \eqref{div-curl-1} and \eqref{equilibrium-ode}, we get  the steady solution $z$-pinch $\big(\overline{ u}=0,\,\,  B=B(r)=B_\theta(r)e_\theta, \,\, \mathbb{J}_z=\mathbb{J}_z(r), \,\, \widehat{  B}=\widehat{B}(r)=\widehat{ B}_\theta(r)e_\theta\big)$ as follows
	\begin{equation*}
	\begin{cases}
	&B_r= 0, \quad B_z= 0, \quad  B_\theta(r)=\Big(-\frac{2}{r^2}\int_0^rs^2p'(s)ds\Big)^{\frac{1}{2}},
	 \\
	&\mathbb{J}_z(r)
	=\frac{1}{2}\Big(-\frac{2}{r^2}\int_0^rs^2p'(s)ds\Big)^{-\frac{1}{2}}\Big(\frac{4}{r^3}\int_0^rs^2p'(s)ds-2p'(r)\Big)\\
	&\quad\quad\quad+\frac{1}{r}\Big(-\frac{2}{r^2}\int_0^rs^2p'(s)ds\Big)^{\frac{1}{2}}\quad \mbox{in} \quad \overline{\Omega},\\
	& \widehat{B}_r= 0, \quad \widehat{B}_z= 0, \quad \widehat{B}_{\theta}(r) = B_{\theta}(r_0)\frac{r_0}{r} \quad \mbox{in} \quad \overline{\Omega}^v.
	\end{cases}
	\end{equation*}
Since $p\in C^{2,1}([0,r_0])$. 
	The equilibrium magnetic field $B_\theta$ is determined in terms of  $p$ by the equation
$$B_\theta(r)=\Big(-\frac{2}{r^2}\int_0^rs^2p'(s)ds\Big)^{\frac{1}{2}},$$
	which by forcing the value of $B_\theta(r) $ at $r=0$ to finite $B_\theta(0)$ and the value of $\mathbb{J}_z(r) $ at $r=0$ to finite $\mathbb{J}_z(0)$, gives that
	$B_\theta(0):
	=\lim_{r\rightarrow 0}B_\theta(r)=\lim_{r\rightarrow 0}\Big(-\frac{2}{r^2}\int_0^rs^2p'(s)ds\Big)^{\frac{1}{2}}=0$. 
	Since $p'=-\mathbb{J}_z(r)B_\theta(r)$, we have $p'(0)=0$, which gives that
	\begin{equation*}
	\begin{split}
	B'_\theta(0):=\lim_{r\rightarrow 0}\frac{B_\theta(r)}{r}=\lim_{r\rightarrow 0}\Big(-\frac{2}{r^4}\int_0^rs^2p'(s)ds\Big)^{\frac{1}{2}}=\lim_{r\rightarrow 0}\Big(-\frac{p'(r)}{2r}\Big)^{\frac{1}{2}}=\Big(-\frac{p''(0)}{2}\Big)^{\frac{1}{2}}.
	\end{split}
	\end{equation*}

		Hence, we can get $B_\theta\in C^{1,1}([0,r_0])$.	
By	
$	p'(r)=-B_\theta(r)\mathbb{J}_z(r),$
	we have
	$\mathbb{J}_z\in C^{1,1}([0,r_0])$.
\end{proof}
From  Lemma \ref{steady-lem}, we know that at the plasma-vacuum interface,  the steady solution $B(r)$  in  cylindrical $ r$, $\theta$, $z$-coordinates satisfies naturally 
\begin{equation*}
n_0\cdot B=n_0 \cdot \widehat{B}=0, \quad \mbox{on} \quad \Sigma_{0, pv},
\end{equation*} due to  $n_0=e_r$, $B=(0, B_\theta(r),0)$ and $\widehat{B}=(0, B_\theta(r_0)\frac{r}{r_0},0)$.

Now we introduce the admissibility of the pressure $p$, which will be used in the following sections.	
\begin{defi}\label{admissible}
	We say that $p$ is admissible if $p(r) \geq 0$ for all $r \in [0,r_0]$ and $p(r)=0$ if and only if $r=r_0$, $p'(r)\leq 0$ for $r$ near $r_0$, that is, $p'(s)\leq 0$ for $s\in (r_0-\epsilon, r_0]$ with small constant $\epsilon>0$,  and $p(r)$ satisfies  \eqref{new-condition} and
	\begin{equation}\label{minimizer-assumption}
	\quad \lim_{r\rightarrow r_0}\frac{p(r)}{p'(r)}=0.
	\end{equation}
\end{defi}

From Taylor expanding and  the steady state \eqref{z-pinch-1-1}, we have the following lemma.
\begin{lem}\label{expansion-zero}
	Assume the function $\mathbb{J}_z(r)$ satisfies  \eqref{cond-jz},	we have $\mathbb{J}_z(r)=\mathbb{J}_z(0)+r\mathbb{J}'_z(0)+O(r^2).$
	Moreover, we have
	\begin{equation}\label{taylor-B}
	B_\theta(r)=\frac{1}{2}\mathbb{J}_z(0)r
	+\frac{1}{3}\mathbb{J}'_z(0)r^2+O(r^3),
	\end{equation}
	\begin{equation}\label{taylor-p}
	\begin{split}
	p'(r)=-\frac{1}{2}\mathbb{J}^2_z(0)r
	-\frac{5}{6}\mathbb{J}'_z(0)\mathbb{J}_z(0)r^2+O(r^3).
	\end{split}
	\end{equation}
\end{lem}
\begin{proof}
	Since $\mathbb{J}_z(r)$ satisfies \eqref{cond-jz}, we can expand $\mathbb{J}_z(r)$ near $r=0$ as 
	\begin{equation}\label{expand-J-near-zero}
	\mathbb{J}_z(r)=\mathbb{J}_z(0)+r\mathbb{J}'_z(0)+O(r^2).
	\end{equation}
	From the steady state \eqref{z-pinch-1-1},  we know that
	\begin{equation*}
	B_\theta(r)=\frac{1}{r}\int_0^r\mathbb{J}_z(r)rdr,
	\end{equation*}
	which together with \eqref{expand-J-near-zero} gives \eqref{taylor-B}. Therefore, we have
	$$B'_\theta(r)=\frac{1}{2}\mathbb{J}_z(0)
	+\frac{2}{3}\mathbb{J}'_z(0)r+O(r^2),$$
	\begin{equation*}
	\begin{split}
	p'(r)&=-B_\theta(r)B'_{\theta}(r)-\frac{B_{\theta}^2(r)}{r}\\
	&=-\Big[\frac{1}{2}\mathbb{J}_z(0)r
	+\frac{1}{3}\mathbb{J}'_z(0)r^2+O(r^3)\Big]\Big[\frac{1}{2}\mathbb{J}_z(0)
	+\frac{2}{3}\mathbb{J}'_z(0)r+O(r^2)\Big]\\
	&\quad-\frac{\frac{1}{4}\mathbb{J}^2_z(0)r^2
		+\frac{1}{3}\mathbb{J}'_z(0)\mathbb{J}_z(0)r^3+O(r^4)}{r}\\
	&=-\frac{1}{2}\mathbb{J}^2_z(0)r
	-\frac{5}{6}\mathbb{J}'_z(0)\mathbb{J}_z(0)r^2+O(r^3).
	\end{split}
	\end{equation*}
\end{proof}
Now, we give the  properties about the steady solution.
\begin{lem}\label{instability-m=0}
There exists $r_* \in (0,r_0)$ such that
	\begin{equation}\label{cond}
	p'(r_*) + \frac{2 \gamma p(r_*) {B^2_\theta(r_*)} }{r_*(\gamma p(r_*) + {B^2_\theta(r_*)} )} <0.
	\end{equation}
\end{lem}
\begin{proof}
	First note that simple algebra reveals that \eqref{cond} is equivalent to the existence of $r_* \in (0,r_0)$ such that 
	\begin{equation*}
	p'(r_*) +  \frac{2 \gamma}{r_*}p(r_*) -  \frac{2 \gamma^2 {p^2(r_*)} }{r_*(\gamma p(r_*) + {B^2_\theta(r_*)})} <0.
	\end{equation*}
	Consider the function $q \in C^1([0,r_0])$ given by $q(r) = r^{2 \gamma} p(r)$.  We have that 
	$q(0) =  0 \text{ and } q(r_0) = r_0^{2\gamma} p(r_0) = 0$,
	so by the mean-value theorem there exists $r_* \in (0,r_0)$ such that 
	\begin{equation*}
	0 = q'(r_*) =  r_*^{2\gamma}p'(r_*) + 2\gamma {r_*}^{2\gamma -1} p(r_*) = r_*^{2\gamma} \left( p'(r_*) + \frac{2\gamma p(r_*)}{r_*} \right).
	\end{equation*}
	Consequently,
	\begin{equation}\label{zero-p-first}
	0 = p'(r_*) + \frac{2\gamma p(r_*)}{r_*}.
	\end{equation}
We know that $p(r_*) >0$, and so 
	\begin{equation}\label{negative-p-sec}
	-  \frac{2 \gamma^2 {p^2(r_*)} }{r_*(\gamma p(r_*) + {B^2_\theta(r_*)} )} < 0.
	\end{equation}
	Combining \eqref{zero-p-first} and \eqref{negative-p-sec}, we conclude that 
	\begin{equation*}
	p'(r_*) +  \frac{2 \gamma}{r}p(r_*) -  \frac{2 \gamma^2 {p^2(r_*)}}{r_*(\gamma p(r_*) + {B^2_\theta(r_*)} )} = -  \frac{2 \gamma^2 {p^2(r_*)} }{r_*(\gamma p(r_*) + {B^2_\theta(r_*)} )} < 0.
	\end{equation*}	This concludes the proof.
\end{proof}
We remark that Lemma \ref{instability-m=0} implies unconditional instability of any $z$-pinch equilibrium for $m=0$. 

Next, we give the following property of the steady solution when $|m|
\geq 2$ and $\mathbb{J}_z(r)$ is non-increasing and non-negative.
\begin{lem}\label{m-geq-2}
	For $|m|\geq 2$, suppose that $\mathbb{J}_z: [0,r_0] \to [0,\infty)$ is non-increasing, then we have 
	\begin{equation*}
	2 p'(r) + m^2 \frac{{B^2_\theta(r)}}{r} \ge 0 \,\, \, \,\mbox{for all }\,\, r\in [0,r_0].
	\end{equation*}
\end{lem}
\begin{proof}
Note that	\begin{equation*}\label{j-non-incm1}
	2 p'(r) + m^2 \frac{{B^2_\theta(r)}}{r} = B_\theta(r) \left(-2\mathbb{J}_z(r) + m^2 \frac{B_\theta(r)}{r}  \right).
	\end{equation*}
	Since $\mathbb{J}_z$ is non-increasing and non-negative, we have
	\begin{equation*}
	B_\theta(r) = \frac{1}{r} \int_0^r s \mathbb{J}_z(s) ds \ge \frac{\mathbb{J}_z(r)}{r} \int_0^r s ds = \frac{r \mathbb{J}_z(r)}{2}.
	\end{equation*}
	Hence,
	\begin{equation*}
	-2\mathbb{J}_z(r) + m^2 \frac{B_\theta(r)}{r} \ge -2\mathbb{J}_z(r) + \frac{m^2}{2} \mathbb{J}_z(r)  \ge \mathbb{J}_z(r) \frac{(m^2-4)}{2} \ge 0.
	\end{equation*}
Since $B_\theta(r) \ge 0$ as well which can be obtained from the third equation of the ODEs \eqref{equilibrium-ode}, we deduce that 
	\begin{equation*}
	2 p'(r) + m^2 \frac{{B^2_\theta(r)}}{r} \ge 0.
	\end{equation*}
\end{proof}
We remark that Lemma \ref{m-geq-2} implies absence of instability for $|m|\geq 2$ for a general class of $z$-pinch equilibria.

Now we will give an example when $|m|\geq2$,  the instability condition 
$2p'(r^*) + m^2 \frac{B^2_\theta(r^*)}{r^*} <0$ holds for some $r^* \in (0,r_0)$ and $\mathbb{J}_z$ vanishing at the origin $r=0$ in suitable order. 
\begin{lem}\label{unstable_2}
For $|m|\geq 2$, we define $\alpha = (m^2-2)/2 \ge 1$.  Suppose that $\beta > \alpha\ge 0$ and $\mathbb{J}_z$ vanishes to order $\beta$ at the origin $r=0$ in the sense that 
$ |\mathbb{J}_z(r)|\leq Cr^\beta$, and further suppose that  $B_\theta \neq 0$ in $(0,r_0]$, i.e. $B_\theta$ has a sign. Then
	there exists $r^* \in (0,r_0)$ such that  
	\begin{equation*}
	2p'(r^*) + m^2 \frac{B^2_\theta(r^*)}{r^*} <0.
	\end{equation*}
\end{lem}
\begin{proof}
We will only prove the result assuming that $B_\theta > 0$ in $(0,r_0]$, as the other case follows similarly.
	For $r \in [0,r_0]$ we compute that
	\begin{equation*}
	\begin{split}
	2p'(r) + m^2\frac{B^2_\theta(r)}{r} &= -2\frac{B^2_\theta(r)}{r} -2 B_\theta(r) B'_{\theta}(r) + m^2\frac{B^2_\theta(r)}{r} \\
	&= (m^2-2)\frac{B^2_\theta(r)}{r} -2 B_\theta(r) B'_{\theta}(r) \\
	&	= -2B_\theta(r) \left(B'_{\theta}(r) - \frac{\alpha}{r} B_{\theta}(r) \right) 
	\\ &= -2 r^{\alpha} B_{\theta}(r) (r^{-\alpha} B_{\theta}(r))'.
	\end{split}
	\end{equation*}
	Now, we may estimate from the third equation of \eqref{equilibrium-ode} and the assumption  that
	\begin{equation*}
	\frac{|B_\theta(r)|}{r^{\alpha+1}} \le \frac{1}{r^{2+\alpha}} \int_0^r s |\mathbb{J}_z(s)| ds \le \frac{C}{r^{2+\alpha}} \int_0^r s^{1+\beta} ds = \frac{C}{2+\beta} \frac{r^{2+\beta}}{r^{2+\alpha}} =  \frac{C}{2+\beta}  r^{\beta - \alpha}
	\end{equation*}
	and 
	\begin{equation*}
	\frac{|B'_\theta(r)|}{r^\alpha} \le \frac{\mathbb{J}_z(r)}{r^{\alpha}} + \frac{1}{r^{2+\alpha}} \int_0^r s |\mathbb{J}_z(s)| ds \le C r^{\beta-\alpha} + \frac{C}{2+\beta}  r^{\beta - \alpha}
	\end{equation*}
	in order to conclude that $[0,r_0] \ni r \mapsto B_\theta(r)/r^\alpha \in [0,\infty)$ is a continuous function. Hence if we define $q: [0,r_0] \to [0,\infty)$ via $q(r) = r^{-\alpha} B_\theta(r)$, then we find that $q \in C^1([0,r_0])$.  Note that 
	\begin{equation*}
	q(0) =0 \text{ and } q(r_0) = r_0^{-\alpha} B_\theta(r_0) >0
	\end{equation*}
	since $B_\theta(r_0) >0$.  By the mean-value theorem there exists $r^* \in (0,r_0)$ such that $q'(r^*) > 0$.  Therefore, we have
	\begin{equation*}
	2p'(r^*) + m^2 \frac{B^2_\theta(r^*)}{r^*} = -2 {r^*}^{\alpha} B_\theta(r^*) q'(r^*) <0.
	\end{equation*}	
\end{proof}
We remark that Lemma \ref{unstable_2} implies the instability for $|m|\geq 2$ for some class of $z$-pinch equilibria.
An interesting corollary is found if we suppose that $\mathbb{J}_z$ is non-negative and is compactly supported.

\begin{col}\label{insta-m=1-comp}
	If $\mathbb{J}_z \ge 0$ and $\mathbb{J}_z$ is compactly supported in $(0,r_0)$, then for each $m \in \mathbb{Z} \backslash \{0\}$ there exists $r^* \in (0,r_0)$ such that 
	\begin{equation*}
	2p'(r^*) + m^2 \frac{B^2_\theta(r^*)}{r^*} <0.
	\end{equation*}
\end{col}
We remark that Corollary \ref{insta-m=1-comp} implies the instability for any $m \in \mathbb{Z} \backslash \{0\}$ for some class of $z$-pinch equilibria.

\subsection{Perturbed ideal MHD system and main results }
We consider	the ideal compressible MHD equations in this paper. From the Appendix \ref{appendix},   we get the following linear perturbation equations 
	\begin{equation}\label{linear-perturbation-and-boundary}
	\begin{cases}
	&\partial_tg=v\quad \mbox{in}\quad \overline{\Omega},\\
	& \rho\partial_{tt} g =\nabla(g\cdot\nabla p+\gamma p\nabla\cdot g)+(\nabla \times B)\times [\nabla \times (g \times B)]\\
	&\quad+\{\nabla \times [\nabla \times (g \times B)]\}\times B,\quad \mbox{in}\quad \overline{\Omega},\\
	&  \nabla \cdot \widehat{Q}=0,\quad \mbox{in}\quad \overline{\Omega}^v,\\
	& \nabla \times \widehat{Q}=0,\quad \mbox{in}\quad \overline{\Omega}^v,\\
	& n \cdot \nabla \times (g\times \widehat{B})=n\cdot \widehat{Q}, \quad \mbox{on} \quad \Sigma_{0, pv},\\
	&-\gamma p\nabla\cdot g+B\cdot Q+g\cdot\nabla \big (\frac{1}{2}|B|^2\big)=\widehat{ B}\cdot\widehat{Q}+g\cdot\nabla\big(\frac{1}{2}|\widehat{ B}|^2\big), \quad \mbox{on} \quad \Sigma_{0, pv},\\
	&  n\cdot\widehat{Q}|_{\Sigma_w}=0,\\
	&g|_{t=0}=g_0,
	\end{cases}
	\end{equation}
	with $Q=\nabla\times (g\times B)$.
	
{\bf Notations: }
Define the energy
 \begin{equation}\label{linearized-functional}
E[g, \widehat{Q}]
=E^p[g]+E^s[g]+E^v[\widehat{Q}],
\end{equation}
where $E^p[{g}]$, $E^s[g]$ and  $E^v[\widehat{ Q}]$ are three real numbers, satisfying 
\begin{equation}\label{fluid-e}
E^p[g]=\frac{1}{2}\int_{\overline{\Omega}}[\gamma p |\nabla \cdot g|^2+Q^2+(g^*\cdot\nabla p)\nabla\cdot g+(\nabla \times B)\cdot (g^* \times Q)]dx,
\end{equation}
\begin{equation}\label{surface-e}
E^s[g]=\frac{1}{2}\int_{\Sigma_{0, pv}} |n\cdot g|^2n\cdot [[\nabla (p+\frac{1}{2}|B|^2)]]dx,
\end{equation}
\begin{equation*}\label{vacuum-e}
E^v[\widehat{Q}]=\frac{1}{2}\int_{\overline{\Omega}^v}|\widehat{Q}|^2dx.
\end{equation*}

In order to introduce the energy $E$, using Lemma \ref{vec-xi-grow-lem-first} in Appendix \ref{appendix}, 
	formally we can get that
 \begin{equation}
\begin{split}
 \frac{d}{dt}\|\sqrt{\rho}g_{t}\|^2_{L^2}
=-\frac{d}{dt}E[g, \widehat{Q}].
\end{split}
\end{equation}
We mainly study the normal mode solution for the linearized perturbation \eqref{linear-perturbation-and-boundary} and our goal is to prove that there exists normal mode solution $g$ and $\widehat{ Q}$ such that $\inf_{\int\rho g^2dV=1} E[g, \widehat{Q}]<0$, which is the key step for getting the instability and ill-posedness for the normal mode solution in cylindrical coordinates. 

The main purpose of this paper is to construct the growing mode solution of the form 
\begin{equation}\label{grow-mode}
\begin{split}
&g(r,\theta, z, t)=(g_{r,mk}(r),g_{\theta,mk}(r),g_{z,mk}(r))e^{\mu t+i(m\theta+kz)}, \\&
\widehat{Q}(r,\theta, z)=(i\widehat{ Q}_{r,mk}(r),\widehat{ Q}_{\theta,mk}(r),\widehat{Q}_{z,mk}(r))e^{\mu t+i(m\theta+kz)},
\end{split}
\end{equation} 
where $$\mu >0,$$
  $(r,\theta,z)$ are the cylindrical coordinates,  $m,k\in \mathbb{Z}$, 
and the subscripts $m$ and $k$ will be dropped for notational simplicity. Here, $g_\theta$ and $g_z$ are pure imaginary functions, $g_r$, $\widehat{Q}_r$, $\widehat{Q}_\theta$ and $\widehat{Q}_z$ are real-valued functions,
 then we define new  three real-valued funtions
\begin{equation}\label{defini-xi-eta-zeta}
\xi=e_r\cdot g=g_r,\quad \eta=-ie_z\cdot g=-ig_z,
\quad \zeta=ie_\theta\cdot g=ig_\theta.
\end{equation}
which together with \eqref{grow-mode}, gives that
	\begin{equation}\label{varable-in-cyl}
\begin{split}
&Q=\nabla\times (g\times B)=\frac{im}{r}B_{\theta}\xi e_r-[(B_{\theta}\xi)'-k B_{\theta}\eta]e_{\theta}-\frac{m}{r}\eta B_{\theta}e_z,
\\&g\cdot \nabla p+\gamma p\na\cdot g=p'\xi+\gamma p\na\cdot g,\quad \na\cdot g =\frac1r(r\xi)'-k\eta+\frac{m}{r}\zeta,
\end{split}
\end{equation}
where the factor $e^{\mu t+i(m\theta +kz)}$ is dropped for notational simplicity.

Denote the Fourier decomposition $$E=\sum_{m,k \in \mathbb{Z}}E_{m,k}.$$
 In terms of $\xi$, $\eta$ and $\zeta$, from the expressions in \eqref{varable-in-cyl},  the boundary conditions in \eqref{linear-perturbation-and-boundary} are transformed to 
 \begin{equation}\label{bound-q-wall}
 \widehat{ Q}_r=0,\quad \mbox{at} \quad r=r_w, 
 \end{equation}
 \begin{equation}\label{boundary-cyl-1}
 m\widehat{ B}_\theta\xi=r\widehat{ Q}_r, \quad \mbox{at} \quad r=r_0, 
 \end{equation}
 \begin{equation}\label{boundary-cyl-2}
 \begin{split}
 &B_{\theta}^2\xi-B_{\theta}^2\xi' r+kB_\theta^2\eta r -\widehat{ B}_{\theta}\widehat{Q}_{\theta}r=0, \quad \mbox{at} \quad r=r_0.
 \end{split}
 \end{equation}
 We impose the boundary conditions \eqref{bound-q-wall} and \eqref{boundary-cyl-1} as constraint for variational problem setup and the boundary condition \eqref{boundary-cyl-2} follows the minimizer solution. When $m=0$, we know that $\widehat{  Q}_r=0$ on the boundary $r=r_0$, which implies that $\widehat{  Q}_r$ is separated from interior variational problem. Obviously, it holds that $\widehat{  Q}=0$.
Therefore, for the case $m=0$ and any $k$, 
the energy functional \eqref{linearized-functional}
reduces to 
\begin{equation}\label{energy-m=0}
\begin{split}
E_{0,k}&=E_{0,k}(\xi,\eta)=2\pi^2\int_0^{r_0}\bigg\{\Big[\frac{2p'}{r}+\frac{4\gamma p B_{\theta}^2}{r^2(\gamma p+B_{\theta}^2)}\Big]\xi^2\\
&\quad+(\gamma p+B_{\theta}^2)\Big[k\eta-\frac{1}{r}\Big((r\xi)'-\frac{2B_{\theta}^2}{\gamma p +B_{\theta}^2}\xi\Big)\Big]^2\bigg\}rdr.
\end{split}
\end{equation} 

For the case $m\neq 0$ and any $k$, the solution $\widehat{ Q}_r$ and $\xi$ are related by the boundary conditions \eqref{boundary-cyl-1}, so we can not set $\widehat{  Q}_r=0$, therefore the energy funtional takes the form of 
\begin{equation}\label{sausage-in-v-t-out}
\begin{split}
E_{m,k}&=E_{m,k}(\xi,\eta,\zeta,\widehat{ Q}_r)\\&=2\pi^2\int_0^{r_0}\bigg\{(m^2+k^2r^2)\Big[\frac{B_\theta}{r}\eta+\frac{-kB_\theta(r\xi)'+2kB_{\theta}\xi}{m^2+k^2r^2}\Big]^2\\
&\quad+\gamma p\Big[\frac{1}{r}(r\xi)'-k\eta+\frac{m}{r}\zeta\Big]^2\bigg\}rdr+2\pi^2\int_0^{r_0}\frac{m^2B_\theta^2}{r(m^2+k^2r^2)}(\xi-r\xi')^2 \\
&\quad+2\pi^2 	\int_0^{r_0}\Big[2p'+\frac{m^2B_{\theta}^2}{r}\Big]\xi^2dr+2\pi^2\int_{r_0}^{r_w}\bigg[|\widehat{Q}_r|^2+\frac{1}{m^2+k^2r^2}|(r\widehat{Q}_r)'|^2\bigg]rdr.
\end{split}
\end{equation}

	Assume the steady solution $(\overline{ u}, B,p)$ in plasma satisfies  $\overline{u}=0$, $B=(0,B_{\theta}(r),0)$, $p=p(r)$ and the steady solution $\widehat{ B}$ in the vacuum region satisfies $\widehat{ B}=(0,\widehat{ B}_\theta(r),0) $ with $\widehat{ B}_\theta(r)=B_{\theta}(r_0)\frac{r_0}{r}$, which are stated in \eqref{z-pinch-1-1}.
	Our main results are as follows:
\begin{thm}\label{linear instability}
Assume \eqref{minimizer-assumption} and Definition \ref{admissible}/the admissibility of the pressure $p$ holds.
	Then we have
\item[1)] For the modes $m=0$ and any $k\in\mathbb{Z}$, $\lambda_{0,k}=\inf_{((\xi,\eta),\widehat{Q}_r)\in \mathcal{A}_1}E_{0,k}<0$, where the set  $\mathcal{A}_1$ is defined in \eqref{set-A-m=0}, 
there is always growing mode of $z$-pinch instability.
\item[2)] For the modes $m\neq 0$ and any $k\in \mathbb{Z}$, $2p'(r^*)+\frac{m^2B_\theta^2(r^*)}{r^*}<0$,
then $\mu_{m,k}=\sqrt{-\lambda_{m,k}}$ is the growing mode to the linearized PDE  \eqref{linear-perturbation-and-boundary},  see also \eqref{spectral-formulation} and \eqref{euler-l-q}.
\item[3)] Moreover, if the pressure satisfies 
$|p'|\leq C\rho $ for $r$ near $r_0$, then it holds that  \begin{equation}\label{sup-Lambda}
	0<\sup_{B}\mu<\infty,\,\,\mbox{where}\,\, B=\{(m,k)\in \mathbb{Z}\times \mathbb{Z}|\lambda_{m,k}<0\}.
\end{equation}
\end{thm}
\begin{thm}\label{ill-posedness}
Assume $p(r)=C(r_0-r)^\beta$ for $r$ near $r_0$ and $\beta\geq 1$. If $\gamma< \frac{\beta}{\beta-1} $, 
	then the growing mode has no lower bound, that is, $$\lim_{k\rightarrow\infty}\lambda_{k}=-\infty.$$
\end{thm}
\begin{rmk}
As a consequence of Theorem \ref{linear instability}, there exists growing mode solution of the form \eqref{grow-mode} and \eqref{defini-xi-eta-zeta} with growth rate $\mu=\mu_{m,k}>0$. This implies for any reasonable norm $\|\cdot\|$, the solution to the linearized PDE \eqref{linear-perturbation-and-boundary} grows at the rate of at least $e^{\mu t}$, with $\mu>0$.
\end{rmk}
\begin{rmk}
As a consequence of Theorem \ref{ill-posedness}, for general bounded initial data, for any $T>0$, the corresponding linear solution can become unbounded with $[0,T]$. This indicates an ill-posedness in the sense of Hadamard.
\end{rmk}
\begin{rmk}
	As shown in Example \ref{example}, Theorem \ref{linear instability} is valid for $p=C(r_0-r)^\beta$ near the boundary $r_0$ with $\beta\geq 1$ for $\gamma\geq \frac{\beta}{\beta-1}$ or $ p=Cexp\{-(r_0-r)^{-\beta}\}$ near the boundary $r_0$ with $\beta>0$ for $\gamma>1$. Hence the condition for ill-posedness in Theorem \ref{ill-posedness} is rather sharp.
\end{rmk}
\begin{rmk}
In plasma literature, instability with $|m|=0$ is called a Sausage instability, and instability with $|m|=1$ is called a Kink instability.
\end{rmk}

We establish for the first time that any z-pinch profile is unconditionally unstable. Moreover, sufficient conditions are obtained to determine the boundedness of the eigenvalues.

The key observation lies in Lemma \ref{instability-m=0}, which leads to a growing mode in so-called sausage instability. In order to construct a growing mode, we study the variational problem for the linearized functional \eqref{linearized-functional}. It should be noted that \eqref{fluid-e} only $\mbox{div}\, g$ is expected to be bounded, which fails to provide necessary compactness to find a minimizer or an eigenfunction. 

To overcome this seemingly lack of compactness, we study carefully the variational problem in cylindrical charts. It turns out that thanks to special symmetry of the $z$-pinch profile, the energy takes the form of \eqref{energy-m=0} for $m=0$ and any $k$, takes the form of \eqref{sausage-in-v-t-out} for $m\neq 0$ and any $k$. The only possible negative part which needs compactness in \eqref{energy-m=0} and \eqref{sausage-in-v-t-out} are given by $$\int_0^{r_0}2p'\xi^2dr.$$
 
We note crucially it depends only on $\xi$. Luckily, it is possible to control $\partial_r \xi$ so that the compactness is established away from $r=0$ and $r=r_0$.

The compactness at $r=r_0$ and $r=0$ is delicate, due to subtle vanishing order of the $z$-pinch. Near $r=r_0$, we make use of an integration by parts to derive  estimate \eqref{weight-p'} and to gain compactness as in Lemma \ref{embedding-p'} and Lemma \ref{embedding-p'-general-m}. Near $r=0$, however, we make use an expansion of exact $z$-pinch profile in Lemma \ref{expansion-zero},  to gain subtle but crucial higher vanishing power and positivity of lower power, more precisely, we observe
\begin{equation*}
\begin{split}
&	\frac{m^2B_\theta^2}{r(m^2+k^2r^2)}(\xi_n-r\xi_n')^2 +\Big[2p'+\frac{m^2B_{\theta}^2}{r}\Big]\xi_n^2\\&\quad=\Big[\frac{1}{4}\mathbb{J}^2_z(0)r+O(r^2)\Big]r^2\xi_n'^2+O(r^3)\xi_n^2,\quad\mbox{for}\quad |m|=1
\end{split}
\end{equation*}
and
\begin{equation*}
\begin{split}
\Big[2p'+\frac{m^2B_{\theta}^2}{r}\Big]\xi_n^2=\Big[\Big(\frac{m^2}{4}-1\Big)\mathbb{J}^2_z(0)r
+O(r^2)\Big]\xi_n^2,\quad\mbox{for}\quad |m|\geq 2
\end{split}
\end{equation*}
to ensure compactness. See Proposition \ref{infimum-A-out} for the details.

We establish the ill-posedness in Section 4. In the case $p(r)=C (r_0-r)^\beta$ for $r$ near $r_0$ and $\beta\geq 1$. If $\gamma< \frac{\beta}{\beta-1} $, 
 we use a careful scaling analysis to construct particular test function of the form $\xi_k=w(k^\alpha[r-r_0])$, without lower bound as $k\rightarrow \infty$.

\section{A family of variational problems }

\subsection{Growing mode ansatz and Cylindrical coordinates}
In this paper, we mainly study the normal mode solution for the linearized perturbation \eqref{linear-perturbation-and-boundary} in cylindrical  coordinates $e_r$, $e_\theta$ and $e_z$.	In order to write the energy in cylindrical  coordinates, we now record several computations in cylindrical  coordinates. In cylindrical  coordinates, 
	under  the normal mode 
\eqref{grow-mode}, the result of the gradient operator becomes algebraic multipliers
$\na =e_r\partial_r+ike_{z}+\frac{im}{r}e_\theta $ and we can get the following lemma.
\begin{lem}\label{lem-e}
 We decompose $E_{m,k}$ as follows
	\begin{equation}\label{E-mk}
E_{m,k}(\xi,\eta,\zeta,\widehat{ Q}_r)=E^p_{m,k}+E^s_{m,k}+E^v_{m,k},
	\end{equation}
 where the fluid energy takes the form of
	\begin{equation}\label{fluid-en-cyl-viscosity}
	\begin{split}
	E^p_{m,k}&=2\pi^2\int_0^{r_0}\bigg\{\frac{m^2B_\theta^2}{r^2(m^2+k^2r^2)}\big[(r\xi)'\big]^2+\beta_0(r\xi)^2\\
	&\qquad+(m^2+k^2r^2)\Big[\frac{B_\theta}{r}\eta+\frac{-kB_\theta(r\xi)'+2kB_{\theta}\xi}{m^2+k^2r^2}\Big]^2\\
	&\qquad+\gamma p\Big[\frac{1}{r}(r\xi)'-k\eta+\frac{m}{r}\zeta\Big]^2\bigg\}rdr-2\pi^2\Big[\frac{2m^2B_\theta^2}{m^2+k^2r^2}\xi^2\Big]_{r=r_0},
	\end{split}
	\end{equation}
	with 
	$$\beta_0=\frac{1}{r}\Big[\frac{m^2B^2_\theta}{r^3}+\frac{2m^2B_\theta (\frac{B_\theta}{r})'}{r(m^2+k^2r^2)}-\frac{4k^2m^2B^2_{\theta}}{r(m^2+k^2r^2)^2}+\frac{2k^2p'}{m^2+k^2r^2}\Big],$$
	the surface energy vanishes
	\begin{equation}\label{surface-en-cyl-viscosity}
	E^s_{m,k}=-2\pi^2[\widehat{B}_{\theta}^2-B_{\theta}^2]_{r=r_0}\xi^2(r_0)=0,
	\end{equation}
	and when $m\neq 0$ and any $k$,   the vacuum energy takes  the form of 
	\begin{equation}\label{vacuum-e-cyl-viscosity}
	\begin{split}
	E^v_{m,k}&=2\pi^2\int_{r_0}^{r_w}\Big[|\widehat{Q}_r|^2+\frac{1}{m^2+k^2r^2}|(r\widehat{Q}_r)'|^2\Big]rdr.
	\end{split}
	\end{equation}
\end{lem}
\begin{proof}
Recall $(\xi,\eta,
\zeta)$	 in \eqref{defini-xi-eta-zeta} and \eqref{grow-mode}.	We begin with the proof of \eqref{fluid-en-cyl-viscosity}.	Inserting the expressions of \eqref{varable-in-cyl} into \eqref{fluid-e} and using $g_\theta^*=i\zeta$, we can get  \eqref{fluid-en-cyl-viscosity}. 
	In fact,
	\begin{equation*}
	(g^*\cdot\nabla  p)(\nabla \cdot g )=\xi p'\Big(\frac1r(r\xi)'-k\eta+\frac{m}{r}\zeta\Big), 
	\end{equation*}
	\begin{equation*}
	\begin{split}
	(\nabla \times B)\cdot (g^* \times Q)&=\Big(B_\theta'+\frac{B_\theta}{r}\Big)\Big[-\xi\Big((B_{\theta}\xi)'-k B_{\theta}\eta\Big)-\frac{im}{r}B_{\theta}\xi g_\theta^*\Big]\\
	&=\Big(B_\theta'+\frac{B_\theta}{r}\Big)\Big[-\xi\Big((B_{\theta}\xi)'-k B_{\theta}\eta\Big)+\frac{m}{r}\zeta B_{\theta}\xi\Big],
	\end{split}
	\end{equation*}	
	which combining with \eqref{fluid-e} and  \eqref{varable-in-cyl} gives \eqref{fluid-en-cyl-viscosity}.
	
Now we turn to the proof of \eqref{surface-en-cyl-viscosity}.
	From \eqref{surface-e} and the equilibrium equations \eqref{equilibrium-ode}, we have 
	$$E^s_{m,k}=-2\pi^2[\widehat{B}_{\theta}^2-B_{\theta}^2]_{r=r_0}\xi^2(r_0),$$
	which together with $p+\frac{1}{2}B^2_\theta=\frac{1}{2}\widehat{  B}^2_\theta$ and $p=0$ on the interface boundary $r=r_0$, gives that
	$	E^s_{m,k}=0$.
	
Finally, we prove \eqref{vacuum-e-cyl-viscosity}.
	From the vacuum equation \eqref{linear-perturbation-and-boundary}$_3$  and $\widehat{ Q}$ in \eqref{grow-mode},  it follows that
	\begin{equation}\label{div-free-cly-viscosity}
	\frac{1}{r}(r\widehat{ Q}_r)'+\frac{m}{r}\widehat{ Q}_\theta+k\widehat{ Q}_z=0,
	\end{equation}
	that is, $\widehat{ Q}_z=-\frac{(r\widehat{ Q}_r)'+m\widehat{ Q}_\theta}{kr}.$
	Inserting the expression of $\widehat{ Q}_z$ into 
	$E^v_{m,k}[\widehat{Q}]=\frac{1}{2}\int_{\overline{\Omega}}|\widehat{Q}|^2dx$,
	implies that 
	\begin{equation}\label{vacuum-e-cyl-viscosity-2}
	\begin{split}
	E^v_{m,k}&=2\pi^2\int_{r_0}^{r_w}\bigg[|\widehat{Q}_r|^2+|\widehat{Q}_\theta|^2+|\widehat{ Q}_z|^2\bigg]rdr\\
	&=2\pi^2\int_{r_0}^{r_w}\bigg[|\widehat{Q}_r|^2+|\widehat{Q}_\theta|^2+\frac{1}{k^2r^2}\Big|(r\widehat{Q}_r)'+m\widehat{Q}_\theta\Big|^2\bigg]rdr.
	\end{split}
	\end{equation}	
	From the vacuum equations \eqref{linear-perturbation-and-boundary}$_2$ and $\widehat{ Q}$ in \eqref{grow-mode}, we have
	\begin{equation}\label{ode-curl-b}
	\begin{cases}
	&\frac{d}{dr}\widehat{ Q}_\theta+\frac{m}{r}\widehat{ Q}_r+\frac{1}{r}\widehat{ Q}_\theta=0,\\
	&k\widehat{ Q}_r+\frac{d}{dr}\widehat{ Q}_z=0,\\
	&\frac{m}{r}\widehat{ Q}_z-k\widehat{ Q}_\theta=0.
	\end{cases}
	\end{equation}  
	Using the third equation in \eqref{ode-curl-b}, from \eqref{div-free-cly-viscosity}, we can obtain the tangential components of $\widehat{ Q}$ in terms of radial component 
	\begin{equation}\label{tang-comp}
	\widehat{ Q}_{\theta}=-\frac{m}{m^2+k^2r^2}(r\widehat{ Q}_r)', \quad \widehat{ Q}_{z}=-\frac{kr}{m^2+k^2r^2}(r\widehat{ Q}_r)'.
	\end{equation}
From \eqref{tang-comp}, we know that the first equation and the second equation in \eqref{ode-curl-b} are equivalent.
	From \eqref{vacuum-e-cyl-viscosity-2} and \eqref{tang-comp}, the vacuum energy takes form of
	\begin{equation*}
	\begin{split}
	E^v_{m,k}&=2\pi^2\int_{r_0}^{r_w}\bigg[|\widehat{Q}_r|^2+|\widehat{Q}_\theta|^2+\frac{1}{k^2r^2}
	\Big|(r\widehat{Q}_r)'+m\widehat{Q}_\theta\Big|^2\bigg]rdr\\
	&=2\pi^2\int_{r_0}^{r_w}\bigg[|\widehat{Q}_r|^2+\frac{1}{m^2+k^2r^2}|(r\widehat{Q}_r)'|^2+\frac{m^2+k^2r^2}{k^2r^2}\Big|\widehat{Q}_\theta+\frac{m}{m^2+k^2r^2}(r\widehat{Q}_r)'\Big|^2\bigg]rdr\\
	&=2\pi^2\int_{r_0}^{r_w}\bigg[|\widehat{Q}_r|^2+\frac{1}{m^2+k^2r^2}|(r\widehat{Q}_r)'|^2\bigg]rdr.
	\end{split}
	\end{equation*}	
\end{proof}
We will use Lemma \ref{lem-e} to prove the following equivalent energy functionals. 
\begin{prop}\label{energy-prop}
The energy functionals \eqref{E-mk} to \eqref{vacuum-e-cyl-viscosity} take forms of  \eqref{energy-m=0} and \eqref{sausage-in-v-t-out} respectively. 
\end{prop}
\begin{proof} 
	From Lemma \ref{lem-e} and let $m=0$ in \eqref{fluid-en-cyl-viscosity}, we can get directly
	\begin{equation*}
	\begin{split}
	E_{0,k}(\xi,\eta)&=2\pi^2\int_0^{r_0}\bigg\{\frac{2p'\xi^2}{r}+B_{\theta}^2\Big[k\eta-\frac{1}{r}\big((r\xi)'-2\xi\big)\Big]^2+\gamma p\Big[\frac{1}{r}(r\xi)'-k\eta\Big]^2\bigg\}rdr.
	\end{split}
	\end{equation*}  
	Together with
the identity 
	\begin{equation*}
	\begin{split}
	&\frac{4\gamma p B_{\theta}^2\xi^2}{r(\gamma p+B_{\theta}^2)}+(\gamma p+B_{\theta}^2)\bigg[\frac{4B_\theta^4\xi^2}{r^2(\gamma p+B_\theta^2)^2}+\Big(k\eta-\frac{1}{r}(r\xi)'\Big)\frac{4B_{\theta}^2\xi}{r(\gamma p +B_{\theta}^2)}\bigg]r\\
	&\quad=\frac{4B_\theta^2\xi^2}{r}+4B_\theta^2\xi \Big(k\eta-\frac{1}{r}(r\xi)'\Big),
	\end{split}
	\end{equation*}
we establish \eqref{energy-m=0}.
	From Lemma \ref{lem-e}, it follows that
\begin{equation*}
\begin{split}
E_{m,k}(\xi,\eta,\zeta,\widehat{ Q}_r)&=2\pi^2\int_0^{r_0}\bigg\{\frac{m^2B_\theta^2}{r^2(m^2+k^2r^2)}\big[(r\xi)'\big]^2
+\beta_0(r\xi)^2\\
&\quad+(m^2+k^2r^2)\bigg[\frac{B_\theta}{r}\eta+\frac{-kB_\theta(r\xi)'+2kB_{\theta}\xi}{m^2+k^2r^2}\bigg]^2\\
&\quad+\gamma p\Big[\frac{1}{r}(r\xi)'-k\eta+\frac{m}{r}\zeta\Big]^2\bigg\}rdr-2\pi^2\bigg[\frac{2m^2B_\theta^2}{m^2+k^2r^2}\xi^2\bigg]_{r=r_0}\\
&\quad+2\pi^2\int_{r_0}^{r_w}\bigg[|\widehat{Q}_r|^2+\frac{1}{m^2+k^2r^2}|(r\widehat{Q}_r)'|^2\bigg]rdr,
\end{split}
\end{equation*}
with
$$\beta_0=\frac{1}{r}\bigg[\frac{m^2B_\theta^2}{r^3}+\frac{2m^2B_{\theta}(\frac{B_\theta}{r})'}{r(m^2+k^2r^2)}-\frac{4k^2m^2B^2_{\theta}}{r(m^2+k^2r^2)^2}+\frac{2k^2p'}{m^2+k^2r^2}\bigg].$$
By $p'=-B_\theta B'_\theta-\frac{B^2_\theta}{r}$, we have
\begin{equation*}
\begin{split}
&\int_0^{r_0}\Big(\frac{m^2B_\theta^2}{r^2(m^2+k^2r^2)}[(r\xi)']^2+\beta_0(r\xi)^2\Big)rdr-\bigg[\frac{2m^2B_{\theta}^2}{m^2+k^2r^2}\xi^2\bigg]_{r=r_0}\\
&=\int_0^{r_0}\bigg[\frac{m^2B^2_\theta}{r(m^2+k^2r^2)}(\xi-r\xi')^2+2p'\xi^2+\frac{m^2B^2_\theta\xi^2}{r}\bigg]dr,
\end{split}
\end{equation*}
which implies the energy \eqref{sausage-in-v-t-out}.
\end{proof}

Using  $g(r,\theta, z, t)=(g_{r}(r,t),g_{\theta}(r,t),g_{z}(r,t))e^{i(m\theta+kz)}$, we can prove that the second equation in \eqref{linear-perturbation-and-boundary} is reduced to the following system.
\begin{lem}\label{spectra-lem-without-factor}
Assume  $g(r,\theta, z, t)=(g_{r}(r,t),g_{\theta}(r,t),g_{z}(r,t))e^{i(m\theta+kz)}$ solves
the second equation in \eqref{linear-perturbation-and-boundary}, then 
	\begin{equation}\label{spectral-formulation-orig}
	\begin{split}
	&\left(
	\begin{array}{ccc}
	\frac{d}{dr}\frac{\gamma p+B_{\theta}^2}{r}\frac{d}{dr}r-\frac{m^2}{r^2}B_{\theta}^2
	-r(\frac{B_{\theta}^2}{r^2})'&-\frac{d}{dr}k(\gamma p+B_{\theta}^2)-\frac{2kB_{\theta}^2}{r}&\frac{d}{dr}\frac{m}{r}\gamma p\\
	\frac{k(\gamma p+B_{\theta}^2)}{r}\frac{d}{dr}r-\frac{2kB_{\theta}^2}{r}&
	-k^2(\gamma p+B_{\theta}^2)-\frac{m^2}{r^2}B_{\theta}^2&\frac{mk}{r}\gamma p\\
	-\frac{m\gamma p}{r^2}\frac{d}{dr}r&\frac{mk}{r}\gamma p&-\frac{m^2}{r^2}\gamma p
	\end{array}
	\right)
	\left(
	\begin{array}{lll}
	\xi   \\
	\eta \\
	\zeta\\
	\end{array}
	\right)\\
	&\quad=\rho \left(
	\begin{array}{lll}
	\xi_{tt}   \\
	\eta_{tt} \\
	\zeta_{tt}\\
	\end{array}
	\right).
	\end{split}
	\end{equation}
\end{lem}
\begin{proof}
	Inserting the expression \eqref{varable-in-cyl} into the second equation in \eqref{linear-perturbation-and-boundary}, by  $g(r,\theta, z, t)=(g_{r}(r,t),g_{\theta}(r,t),g_{z}(r,t))e^{i(m\theta+kz)}$ and the definitions of $\xi$, $\eta$ and $\zeta$ in \eqref{defini-xi-eta-zeta}, we can easily get that the second equation in \eqref{linear-perturbation-and-boundary} reduces to \eqref{spectral-formulation-orig}.
\end{proof}	
In order to study the stability to use variational methods, we use the following second-order ODE system.
\begin{lem}\label{spectral-lem-with-factor}
Assume $g(r,\theta, z, t)=(g_{r}(r,t),g_{\theta}(r,t),g_{z}(r,t))e^{\mu t+i(m\theta+kz)}$ solves
the second equation in \eqref{linear-perturbation-and-boundary}, then 
	\begin{equation}\label{spectral-formulation}
	\begin{split}
	&\left(
	\begin{array}{ccc}
	\frac{d}{dr}\frac{\gamma p+B_{\theta}^2}{r}\frac{d}{dr}r-\frac{m^2}{r^2}B_{\theta}^2
	-r(\frac{B_{\theta}^2}{r^2})'&-\frac{d}{dr}k(\gamma p+B_{\theta}^2)-\frac{2kB_{\theta}^2}{r}&\frac{d}{dr}\frac{m}{r}\gamma p\\
	\frac{k(\gamma p+B_{\theta}^2)}{r}\frac{d}{dr}r-\frac{2kB_{\theta}^2}{r}&
	-k^2(\gamma p+B_{\theta}^2)-\frac{m^2}{r^2}B_{\theta}^2&\frac{mk}{r}\gamma p\\
	-\frac{m\gamma p}{r^2}\frac{d}{dr}r&\frac{mk}{r}\gamma p&-\frac{m^2}{r^2}\gamma p
	\end{array}
	\right)
	\left(
	\begin{array}{lll}
	\xi   \\
	\eta \\
	\zeta\\
	\end{array}
	\right)\\
	&\quad
	=\rho \mu^2 \left(
	\begin{array}{lll}
	\xi   \\
	\eta \\
	\zeta\\
	\end{array}
	\right).
	\end{split}
	\end{equation}
\end{lem}
\begin{proof}	
	Inserting the expression \eqref{varable-in-cyl} into the second equation in \eqref{linear-perturbation-and-boundary}, by  $g(r,\theta, z, t)=(g_{r}(r,t),g_{\theta}(r,t),g_{z}(r,t))e^{\mu t+i(m\theta+kz)}$ and the definitions of $\xi$, $\eta$ and $\zeta$ in \eqref{defini-xi-eta-zeta}, we can easily get that the second equation in \eqref{linear-perturbation-and-boundary} reduces to \eqref{spectral-formulation}.
\end{proof}	
In order to study the stability to use variational methods in vacuum domain, we use the following second-order ODE about $\widehat{ Q}_r$ for $m\neq 0$ and any $k$.
\begin{lem}\label{ode-mk}
	The  vacuum equations \eqref{linear-perturbation-and-boundary}$_3$ and \eqref{linear-perturbation-and-boundary}$_4$  can be reduced to the second order differential equation
	\begin{equation}\label{euler-l-q}
	\bigg[\frac{r}{m^2+k^2r^2}(r\widehat{Q}_r)'\bigg]'-\widehat{Q}_r=0,
	\end{equation}
	with the other two components $\widehat{ Q}_{\theta}=-\frac{m}{m^2+k^2r^2}(r\widehat{ Q}_r)'$ and
	$\widehat{ Q}_{z}=-\frac{kr}{m^2+k^2r^2}(r\widehat{ Q}_r)'$. 
\end{lem}
\begin{proof}
	Inserting \eqref{tang-comp} into the first or second ODE in \eqref{ode-curl-b},  we can easily deduce that the radial component satisfies \eqref{euler-l-q}.
\end{proof}

\subsection{Variational problem when $m=0$}
In this subsection, we will introduce the definition of function space $X_k$, then  give the variational analysis for the case $m=0$ and any $k\in \mathbb{Z}$.

Let us first introduce the function space $X_k$ for any $k$ and its properties.
\begin{defi}\label{space-X}
The weighted Sobolev space $X_k$ is defined as the completion of $\{(\xi,\eta) \in C^{\infty}([0,r_0])\times C^{\infty}([0,r_0]) |\xi(0)=0\}$, with respect to the norm 
\begin{equation}\label{definition-Xk}
\begin{split}
\|(\xi,\eta)\|^2_{X_k}&=\int_0^{r_0}\bigg\{p\Big|\frac{1}{r}\partial_r(r\xi(r))\Big|^2+B_{\theta}^2\Big[k\eta-\frac{1}{r}\big((r\xi)'-2\xi\big)\Big]^2\bigg\}rdr
\\&\quad+\int_0^{r_0} \rho (|\xi|^2+|\eta|^2)rdr.
\end{split}
\end{equation}
\end{defi}

Define function $g(r)=\sup_{r\leq s\leq r_0}\frac{p(s)}{-p'(s)}$, then we can get the following lemma.
\begin{lem}\label{weight-p'-lem}
Assume $s_1$ near $r=r_0$,  then it holds that
\begin{equation}\label{weight-p'}
\begin{split}
\int_{s_1}^{r_0}-p'(r)\xi^2dr
\leq   2p(s_1)\xi^2(s_1)+4g(s_1)
\int_{s_1}^{r_0}p\xi'^2dr,
\end{split}
\end{equation}
with $g(s_1)\rightarrow 0$ as $s_1\rightarrow r_0$.
\end{lem}
\begin{proof}
	From Definition \ref{admissible}/admissibility of the pressure $p$, we have $\frac{p}{p'}\rightarrow 0$ for $r \rightarrow r_0$, which together with the definition of function $g$, gives that $p(s)\leq-g(s_1) p'(s)$ for $s_1\leq s\leq r_0$, with $g(s_1)\rightarrow 0$ as $s_1\rightarrow r_0$. 
	Since $p> 0$ for all $r\in(0,r_0)$ and $p'(r)\leq  0$ for $r$ near $r_0$, 
	the H$\ddot{o}$lder inequality provides 
\begin{equation*}
\begin{split}
\int_{s_1}^{r_0}-p'(r)\xi^2dr&=p(s_1)\xi^2(s_1)+2\int_{s_1}^{r_0}p\xi\xi'dr\\
&\leq p(s_1)\xi^2(s_1)+2\Big(\int_{s_1}^{r_0}p\xi^2dr\Big)^{\frac{1}{2}}
\Big(\int_{s_1}^{r_0}p\xi'^2dr\Big)^{\frac{1}{2}}\\
&\leq  p(s_1)\xi^2(s_1)+2g^{\frac{1}{2}}(s_1)\Big(\int_{s_1}^{r_0}-p'\xi^2dr\Big)^{\frac{1}{2}}
\Big(\int_{s_1}^{r_0}p\xi'^2dr\Big)^{\frac{1}{2}},
\end{split}
\end{equation*}
which by Cauchy inequality ensures \eqref{weight-p'}.
\end{proof}
From the Definition \ref{space-X}, we can show the following compactness results.
\begin{lem}\label{embedding-p'}
Assume $s_1$ is near $r=r_0$ and $p$ is admissible. Let $\pi_1$ denote the projection operator onto the first factor. Then $\pi_1: X_k\rightarrow Z$ is a bounded, linear, compact map, 
with the norm 
\begin{equation}\label{defi-z}
\|\xi\|_Z^2=\int_{0}^{s_1}\xi^2dr+\int_{s_1}^{r_0}-p'\xi^2dr,
\end{equation} and we denote it by 
\begin{equation}\label{defi-b-z}
X_k\subset\subset Z.
\end{equation}
\end{lem}
\begin{proof}
For any $(\xi,\eta)\in X_k$, we have for $r\in (0,\frac{r_0}{2})$ that
\begin{equation*}
\begin{split}
|r\xi(r)p(r)|&= \Big|\int_0^r\partial_{s}(s\xi(s)p(s))ds\Big|\leq \Big|\int_0^r\partial_{s}(s\xi(s))p(s)ds\Big|+\Big|\int_0^rs\xi(s)p'(s)ds\Big|\\
&\leq \Big(\int_0^rp(s)\Big|\frac{1}{s}\partial_{s}(s\xi(s))\Big|^2sds\Big)^\frac{1}{2}\Big(\int_0^rp(s)sds\Big)^\frac{1}{2}+\Big|\int_0^r\mathbb{J}_zB_\theta\xi(s)sds\Big|\\
&\leq \frac{r}{\sqrt{2}}\|p\|_{L^{\infty}(0,r)}\Big(\int_0^rp(s)\Big|\frac{1}{s}\partial_{s}(s\xi(s))\Big|^2sds\Big)^\frac{1}{2}\\
&\quad+\frac{r}{\sqrt{2}}\|\mathbb{J}_z\|_{L^{\infty}(0,r)}\|B_\theta\|_{L^{\infty}(0,r)}\Big(\int_0^r\xi^2(s)sds\Big)^\frac{1}{2},
\end{split}
\end{equation*}
which gives that 
\begin{equation}\label{sup-xi-b}
\begin{split}
|\xi(r)p(r)|&\leq \frac{1}{\sqrt{2}}\|p\|_{L^{\infty}(0,r)}\Big(\int_0^rp(s)\Big|\frac{1}{s}\partial_{s}(s\xi(s))\Big|^2sds\Big)^\frac{1}{2}\\
&\quad+\frac{1}{\sqrt{2}}\|\mathbb{J}_z\|_{L^{\infty}(0,r)}\|B_\theta\|_{L^{\infty}(0,r)}\Big(\int_0^r\xi^2(s)sds\Big)^\frac{1}{2}.
\end{split}
\end{equation}
Here  $\|p\|_{L^{\infty}}$, $\|\mathbb{J}_z\|_{L^{\infty}}$ and $\|B_\theta\|_{L^{\infty}}$ are bounded from Lemma \ref{steady-lem} and \eqref{cond-jz}. 

	Assume that $\|(\xi_n,\eta_n)\|_{X_k}\leq C$, for $n\in \mathbb{N}$. Fix any $\kappa>0$.
	We claim that there exists a subsequence $\{\xi_{n_i}\}$ so that \begin{equation}\label{claim}
	\sup_{i,j}\|\xi_{n_i}-\xi_{n_j}\|_{Z}\leq \kappa.
	\end{equation}

To prove the claim, let $s_0$ with $0<s_0<s_1<r_0$ and $s_0$ be chosen small enough, so that 
\begin{equation}\label{first-inter-new}
3s_0C^2\bigg(\|p\|_{L^{\infty}}^2+\frac{\|\mathbb{J}_z\|^2_{L^{\infty}}\|B_\theta\|_{L^{\infty}}^2}{\inf_{0<r\leq s_0}\rho}\bigg)\frac{1}{\inf_{0<r\leq s_0}{p}^2}\leq \kappa.
\end{equation}	
From Definition \ref{admissible}/admissibility of $p$, we have $\frac{p}{p'}\rightarrow 0$ for $r \rightarrow r_0$, which together with the definition of $g$, gives that  $g(s_1)\rightarrow 0$ as $s_1\rightarrow r_0$. Choose $s_1$ close enough to $r_0$, such that 
\begin{equation}\label{fix-s_1-m=0}
g(s_1)C\leq \frac{\kappa}{6},\quad\frac{Cp(s_1)}{3(s_1-s_0)}\leq 
\frac{1}{6}.
\end{equation}
 Since the subinterval $(s_0,s_1)$ avoids the singularity of $\frac{1}{r}$ and degenerate of $p$ at the boundary $r=r_0$, the function $\xi_n$ is uniformly bounded in $H^1(s_0,s_1)$. By the compact embedding $H^1(s_0,s_1)\subset\subset C^0(s_0,s_1)$, one can extract a subsequence $\{\xi_{n_i}\}$ that converges in $L^\infty(s_0,s_1)$. So for $i, j$ large enough, it holds that 
	$$\sup_{i,j}\|\xi_{n_i}-\xi_{n_j}\|^2_{L^\infty(s_0,s_1)}\leq \frac{\kappa}{3(s_1-s_0)}.$$
Since  $p'(r)\leq 0$ near $r=r_0$, by Lemma \ref{weight-p'-lem} and \eqref{fix-s_1-m=0},  we deduce for $i$ and $j$ large enough that
\begin{equation*}
\begin{split}
\int_{s_1}^{r_0}-p'(r)(\xi_{n_i}-\xi_{n_j})^2dr
&\leq 2p(s_1)(\xi_{n_i}-\xi_{n_j})^2(s_1)+4g(s_1)
\int_{s_1}^{r_0}p(\xi'_{n_i}-\xi'_{n_j})^2dr\\
&\leq \frac{Cp(s_1)\kappa}{3(s_1-s_0)}+g(s_1)C\leq \frac{\kappa}{3},
\end{split}
\end{equation*}
where we have used the facts  $\|(\xi_n,\eta_n)\|_{X_k}\leq C$ and
\begin{equation*}
\begin{split}
\int_{s_1}^{r_0}p{\xi'_n}^2dr&\lesssim \int_{s_1}^{r_0}p\Big|\frac{1}{r}\partial_r(r\xi_n(r))\Big|^2rdr+\int_{s_1}^{r_0}p\xi_n^2rdr
\\
&\lesssim \int_{s_1}^{r_0}p\Big|\frac{1}{r}\partial_r(r\xi_n(r))\Big|^2rdr+\|\rho\|^{\gamma-1}_{L^\infty(s_1,r_0)}\int_{s_1}^{r_0}\rho\xi_n^2rdr\\
&\lesssim \int_{0}^{r_0}p\Big|\frac{1}{r}\partial_r(r\xi_n(r))\Big|^2rdr+\|\rho\|^{\gamma-1}_{L^\infty(s_1,r_0)}\int_{0}^{r_0}\rho\xi_n^2rdr.
\end{split}
\end{equation*}

Then along the above subsequence one can get 
from \eqref{sup-xi-b} and \eqref{first-inter-new} that 
\begin{equation}
\begin{split}
\|\xi_{n_i}-\xi_{n_j}\|^2_{Z}&=\int_0^{s_1}|\xi_{n_i}-\xi_{n_j}|^2dr+\int_{s_1}^{r_0}-p'|\xi_{n_i}-\xi_{n_j}|^2dr\\
&=\Big(\int_0^{s_0}+\int_{s_0}^{s_1}\Big)|\xi_{n_i}-\xi_{n_j}|^2dr+\int_{s_1}^{r_0}-p'|\xi_{n_i}-\xi_{n_j}|^2dr
\\
&\leq s_0C^2\bigg(\|p\|_{L^{\infty}}^2+\frac{\|\mathbb{J}_z\|^2_{L^{\infty}}\|B_\theta\|_{L^{\infty}}^2}{\inf_{0<r\leq s_0}\rho}\bigg)\frac{1}{\inf_{0<r\leq s_0}{p}^2}\\
&\quad+(s_1-s_0)\sup_{i,j}\|\xi_{n_i}-\xi_{n_j}\|^2_{L^\infty(s_0,s_1)}
\\
&\quad+\frac{Cp(s_1)\kappa}{3(s_1-s_0)}+g(s_1)C
\leq \kappa,
\end{split}
\end{equation}
which implies the claim \eqref{claim} and the compactness result \eqref{defi-b-z}.
\end{proof}
Now, we consider the case $m=0$ and any $k\in\mathbb{Z}$.
In order to understand $\mu$,
we consider  the following energy
\begin{equation}\label{sausage-in-v}
\begin{split}
E_{0,k}(\xi,\eta)&=2\pi^2\int_0^{r_0}\bigg\{\frac{2p'\xi^2}{r}+B_{\theta}^2\Big[k\eta-\frac{1}{r}((r\xi)'-2\xi)\Big]^2+\gamma p\Big[\frac{1}{r}(r\xi)'-k\eta\Big]^2\bigg\}rdr,
\end{split}
\end{equation} which is equivalent to \eqref{energy-m=0} from Proposition \ref{energy-prop}.
We denote
\begin{equation*}
\mathcal{J}(\xi,\eta)=2\pi^2\int_0^{r_0} \rho (|\xi|^2+|\eta|^2)rdr.
\end{equation*}
From the Definition \ref{space-X} and Lemma \ref{weight-p'-lem}, it follows that $E_{0,k}$ and $\mathcal{J}$ are both well defined on the space $X_k$. 
\begin{lem}\label{well-define}
	$E_{0,k}(\xi,\eta)$ and $\mathcal{J}(\xi,\eta)$ are both well defined on the space $X_k$. 
\end{lem}
\begin{proof}
		From \eqref{taylor-p}, we have for every $s_1<r_0$ that
		\begin{equation}\label{lower-bound}
		\Big|2\pi^2\int_0^{s_1}2p'\xi^2dr\Big|=\Big|4\pi^2\int_0^{s_1} \Big[-\frac{1}{2}\mathbb{J}^2_z(0)r
		-\frac{5}{6}\mathbb{J}'_z(0)\mathbb{J}_z(0)r^2+O(r^3)\Big] \xi^2dr\Big|\leq C\mathcal{J}(\xi,\eta).
		\end{equation}
			Using  Lemma \ref{weight-p'-lem}, for $s_1$ near $r_0$, we get
	\begin{equation}
		\begin{split}
		\Big|\int_{s_1}^{r_0}2p'(r)\xi^2dr	\Big|=	\int_{s_1}^{r_0}-2p'\xi^2dr
		&\leq 4p(s_1)\xi^2(s_1)+8g(s_1)
		\int_{s_1}^{r_0}p\xi'^2dr\\
		&\leq C\mathcal{J}(\xi,\eta)+C\|\big(\xi,\eta\big)\|^2_{X_k},
		\end{split}
		\end{equation}	
		where we have used the facts 
		 $$\xi^2(s_1)\leq \|\xi\|^2_{H^1(s_0,s_1)}\leq C\mathcal{J}(\xi,\eta)+C\|\big(\xi,\eta\big)\|^2_{X_k}\quad\mbox{for}\,\,\, 0<s_0<s_1<r_0,$$ and 
		\begin{equation*}
		\begin{split}
			\int_{s_1}^{r_0}p{\xi'}^2dr&\lesssim \int_{s_1}^{r_0}p\Big|\frac{1}{r}\partial_r(r\xi(r))\Big|^2rdr+\int_{s_1}^{r_0}p\xi^2rdr
			\\
			&\lesssim \int_{0}^{r_0}p\Big|\frac{1}{r}\partial_r(r\xi(r))\Big|^2rdr+\|\rho\|^{\gamma-1}_{L^\infty(s_1,r_0)}\int_{0}^{r_0}\rho\xi^2rdr\\
			&\leq C\|\big(\xi,\eta\big)\|^2_{X_k}.
		\end{split}
	\end{equation*}
		Hence, we get
		\begin{equation*}
	\Big|	\int_0^{r_0}2p'\xi^2dr\Big|\leq  C\mathcal{J}(\xi,\eta)+C\|\big(\xi,\eta\big)\|^2_{X_k},
		\end{equation*}
		which implies that
		 \begin{equation*}
		\begin{split}
		|E_{0,k}(\xi,\eta)|
		&\leq C\mathcal{J}(\xi,\eta)+C	\|(\xi,\eta)\|^2_{X_k}+C\int_0^{r_0}\Big\{B_{\theta}^2\Big[k\eta
		-\frac{1}{r}\big((r\xi)'-2\xi\big)\Big]^2\Big\}rdr\\
		&\quad+C\int_0^{r_0} p\big|\frac{1}{r}(r\xi)'\big|^2rdr+ C\|\rho\|^{\gamma-1}_{L^\infty}\int_0^{r_0} \rho |\eta|^2rdr
	\leq C\|\big(\xi,\eta\big)\|^2_{X_{k}}.	
		\end{split}
		\end{equation*}
		Therefore, $E_{0,k}(\xi,\eta)$ and $\mathcal{J}(\xi,\eta)$ are well defined on the space $X_k$.
\end{proof}

Now we define \begin{equation*}
\lambda=\inf_{(\xi,\eta)\in X_k}\frac{E_{0,k}(\xi,\eta)}{\mathcal{J}(\xi,\eta)}.
\end{equation*}
Consider the set 
\begin{equation}\label{set-A-m=0}
\mathcal{A}_1=\Big\{(\xi,\eta)\in X_k| \mathcal{J}(\xi,\eta)=1\Big\}.
\end{equation}
We want to show that the infimum of $E_{0,k}(\xi,\eta)$ over the set $\mathcal{A}_1$ is achieved and is negative and that the minimizer solves \eqref{spectral-formulation} 
	 with $m=0$ and any $k\in \mathbb{Z}$  and the corresponding boundary conditions.
 
First, we prove that the energy $E_{0,k}$ has a lower bound on the set $ \mathcal{A}_1$.
\begin{lem}\label{lower-bound-lem-m=0}
	The energy $E_{0,k}(\xi,\eta)$ has a lower bound on the set $ \mathcal{A}_1$.
\end{lem}
\begin{proof}
		We can directly get  from the energy \eqref{sausage-in-v} for $0<s_0<s_1<r_0$ that
		\begin{equation*}
		\begin{split}
		E_{0,k}(\xi,\eta)&\geq2\pi^2\int_0^{r_0}\gamma p\Big[\frac{1}{r}(r\xi)'-k\eta\Big]^2rdr+2\pi^2 \int_0^{r_0}2p'\xi^2dr\\
		&\geq 2\pi^2\int_{s_0}^{r_0}\gamma p\Big[\frac{1}{r}(r\xi)'-k\eta\Big]^2rdr\\
		&\quad+2\pi^2\Big(\int_0^{s_1}+\int_{s_1}^{r_0}\Big)2p'\xi^2dr,\, \forall\, (\xi,\eta,\zeta)\in \mathcal{A}_1.
		\end{split}
		\end{equation*}
		Recall  \eqref{lower-bound}, for every $s_1<r_0$, we have
		$\Big|2\pi^2\int_0^{s_1}2p'\xi^2dr\Big|\leq C\mathcal{J}(\xi,\eta).$
		Hence, we get
		\begin{equation*}
		\begin{split}
		E_{0,k}(\xi,\eta)\geq
		2\pi^2\int_{s_0}^{r_0}\gamma p\Big|\frac{1}{r}(r\xi)'\Big|^2rdr+2\pi^2\int_{s_1}^{r_0}2p'\xi^2dr-C\mathcal{J}(\xi,\eta).
		\end{split}
		\end{equation*}
		The key is to control $\int_{s_1}^{r_0}2p'\xi^2dr$. 
			Since in the interval $(s_1,r_0)$, 
				using  Lemma \ref{weight-p'-lem}, we know that
	\begin{equation}\label{estimates-p'-2}
			\begin{split}
			\Big|\int_{s_1}^{r_0}2p'(r)\xi^2dr	\Big|&=	\int_{s_1}^{r_0}-2p'\xi^2dr
			\leq 4p(s_1)\xi^2(s_1)+8g(s_1)
			\int_{s_1}^{r_0}p\xi'^2dr\\
			&\leq C(\sigma)\mathcal{J}(\xi,\eta)p(s_1)+\Big(Cp(s_1)\sigma+Cg(s_1)\Big)\int_{s_0}^{r_0}p\xi'^2dr,
			\end{split}
			\end{equation}	
			where we have used the facts  for $0<s_0<s_1<r_0$ and $\sigma$ small enough
	\begin{equation}
	\begin{split}
			\xi^2(s_1)&\leq \|\xi\|^2_{L^{\infty}} \leq \|\xi\|_{L^2(s_0,s_1)}\|\xi\|_{H^1(s_0,s_1)}
			\leq C(\sigma) \|\xi\|^2_{L^2(s_0,s_1)}+\sigma\|\xi\|^2_{H^1(s_0,s_1)}\\
			&\leq C(\sigma)\mathcal{J}(\xi,\eta)+C\sigma\int_{s_0}^{s_1}p\xi'^2dr
			\end{split}
			\end{equation}
			and $$\int_{s_0}^{s_1}p\xi'^2dr\leq \int_{s_0}^{r_0}p\xi'^2dr,\quad \int_{s_1}^{r_0}p\xi'^2dr\leq \int_{s_0}^{r_0}p\xi'^2dr.$$
		Choosing $s_1$ close enough to $r_0$ and $\sigma$ small enough such that  $Cp(s_1)\sigma+Cg(s_1)\leq \frac{\gamma}{2}$, yields that
		\begin{equation*}
		\begin{split}
		E_{0,k}(\xi,\eta)\geq\pi^2\int_{s_0}^{r_0}\gamma p\Big|\frac{1}{r}(r\xi)'\Big|^2rdr-3C\mathcal{J}(\xi,\eta)\geq -3C\mathcal{J}(\xi,\eta)= -3C,
		\end{split}
		\end{equation*}
		which implies that the energy $E_{0,k}(\xi,\eta)$ has a lower bound on the set $ \mathcal{A}_1$. 		
\end{proof}

Using the fact that	$\mathcal{J}(\xi,\eta)=1$ and $E_{0,k}$ has a lower bound on the set $\mathcal{A}_1$, we can choose a minimizing sequence such that along the minimizing sequence, 
we have $M\leq E_{0,k}(\xi_n,\eta_n)<M+1$, and for the minimizing sequence,
we can show coercivity estimate: 
	\begin{equation}\label{coercivity}
	\begin{split}
	\|(\xi_n,\eta_n)\|^2_{X_k}\leq \mathcal{J}+C(M+1)+C\int_0^{r_0}\Big(2p'\xi_n^2+\gamma pk^2\eta_n^2r\Big)dr\leq C.
	\end{split}
	\end{equation}

We now show that the infimum of $E_{0,k}$ over the set $\mathcal{A}_1$ is negative.
\begin{prop}\label{infimum-A-3}
It holds that $\lambda=-\mu^2=\inf E_{0,k}<0$. 
\end{prop}
\begin{proof}
	Since both $E_{0,k}$ and $\mathcal{J}$ are homogeneous degree $2$, it suffices to show that 
	\begin{equation*}
	\inf_{(\xi,\eta)\in X_k}\frac{E_{0,k}(\xi,\eta)}{\mathcal{J}(\xi,\eta)}<0.
	\end{equation*}
	But since $\mathcal{J}$ is positive definite, one may reduce to constructing any  $(\xi,\eta)\in X_k$ (see \eqref{definition-Xk})  such that $E_{0,k}(\xi,\eta)<0$. 
Using Lemma \ref{instability-m=0},  we can choose a smooth function $\xi^*\in C_c^{\infty}(0,r_0)$ such that $$2\pi^2\int_0^{r_0}\Big[\frac{2p'}{r}+\frac{4\gamma p B_{\theta}^2}{r^2(\gamma p+B_{\theta}^2)}\Big]{\xi^*}^2rdr< 0.$$
Then, 
we can assume that $k\eta^*=\frac{1}{r}\Big((r\xi^*)'-\frac{2B_{\theta}^2}{\gamma p +B_{\theta}^2}\xi^*\Big)$, such that the second term of $E_{0,k}(\xi^*,\eta^*)$ in \eqref{energy-m=0}  vanishes. Here, $\xi^*$ and $\eta^*$ are smooth functions and belong to the space $X_k$.

 Therefore, the energy \eqref{energy-m=0}  becomes
	\begin{equation*}
	\widetilde{E}(\xi^*)
	=2\pi^2\int_0^{r_0}\Big[\frac{2p'}{r}+\frac{4\gamma p B_{\theta}^2}{r^2(\gamma p+B_{\theta}^2)}\Big]{\xi^*}^2rdr<0,
	\end{equation*}
	which implies the result. 
\end{proof}

With Proposition \ref{infimum-A-3} in hand, we apply the direct methods to deduce  the existence of a minimizer of $E_{0,k}$ on the set $\mathcal{A}_1$.
\begin{prop}\label{infimum-A}
	$E_{0,k}$ achieves its infimum on the set $\mathcal{A}_1$.	
\end{prop}
\begin{proof}
	First note  that $E_{0,k}$ is bounded below on the set $\mathcal{A}_1$. Let $(\xi_n,\eta_n) \in \mathcal{A}_1$ be a minimizing sequence. Then, we know that $(\xi_n,\eta_n)$ is bounded in $X_k$ (see \eqref{definition-Xk}), so up to the extraction of a subsequence $\psi_n=|B_{\theta}|\Big[k\eta_n-\frac{1}{r}\big((r\xi_n)'-2\xi_n\big)\Big]r^{\frac12}\rightharpoonup \psi=|B_{\theta}|\Big[k\eta-\frac{1}{r}\big((r\xi)'-2\xi\big)\Big]r^{\frac12}$ weakly in $L^2$, and $\xi_n\rightarrow \xi$ strongly in $Z$ from the compact embedding in Lemma \ref{embedding-p'}.
	By weak lower semi-continuity, since $\psi_n\rightharpoonup \psi$ in the space $L^2(0,r_0)$,
	we have 
	$$\int_0^{r_0}B_{\theta}^2\Big[k\eta-\frac{1}{r}\big((r\xi)'-2\xi\big)\Big]^2rdr\leq\liminf_{n\to\infty}\int_0^{r_0}B_{\theta}^2\Big[k\eta_n-\frac{1}{r}\big((r\xi_n)'-2\xi_n\big)\Big]^2rdr.
	$$
	
	Because of the quadratic structure of all the terms in the integrals defining $E_{0,k}$, similarly by weak lower semicontinuity and $\xi_n\rightarrow \xi$ strongly in $Z$, we get that 
	\begin{equation*}
	E_{0,k}(\xi,\eta)\leq \liminf_{n\to\infty}E_{0,k}(\xi_n,\eta_n)=\inf_{\mathcal{A}_1}E_{0,k}.
	\end{equation*}
	
	All that remains is to show that $(\xi,\eta)\in \mathcal{A}_1 $.
	
	Again by lower semi-continuity, we know that $\mathcal{J}(\xi,\eta)\leq 1$. Suppose by way
	of contradiction that $\mathcal{J}(\xi,\eta)<1$. By the homogeneity of $\mathcal{J}$, we may find $\alpha>1$ so that $\mathcal{J}(\alpha\xi,\alpha\eta)=1$, i.e., we may scale up $(\xi,\eta)$ so that $(\alpha \xi,\alpha\eta )\in \mathcal{A}_1$. By Proposition \ref{infimum-A-3}, we know that $\inf E_{0,k}< 0$, and from this we deduce that $$E_{0,k}(\alpha\xi,\alpha\eta)=\alpha^2E_{0,k}(\xi,\eta)=\alpha^2 \inf E_{0,k}< \inf E_{0,k},$$
	which is a contradiction since $(\alpha \xi,\alpha\eta )\in \mathcal{A}_1$. Hence $\mathcal{J}(\xi, \eta)=1$, so that we show that $( \xi,\eta )\in \mathcal{A}_1$.
\end{proof}
We now prove that the minimizer constructed in the previous result satisfies Euler-Lagrange equations equivalent to \eqref{spectral-formulation} with $m=0$ and any $k\neq 0$.

\begin{prop}\label{infimum-A-2}
	Let $( \xi,\eta )\in \mathcal{A}_1$ be the minimizers of $E_{0,k}$ constructed in Proposition \ref{infimum-A}. Then $(\xi,\eta)$ are smooth when restricted to $(0,r_0)$ and satisfy 
	\begin{equation} \label{spectal formulation-2} 
	\begin{split}
	&\left(
	\begin{array}{ccc}
	\frac{d}{dr}\frac{\gamma p+B_{\theta}^2}{r}\frac{d}{dr}r
	-r(\frac{B_{\theta}^2}{r^2})'&-\frac{d}{dr}k(\gamma p+B_{\theta}^2)-\frac{2kB_{\theta}^2}{r}\\
	\frac{k(\gamma p+B_{\theta}^2)}{r}\frac{d}{dr}r-\frac{2kB_{\theta}^2}{r}&
	-k^2(\gamma p+B_{\theta}^2)
	\end{array}
	\right)
	\left(
	\begin{array}{lll}
	\xi   \\
	\eta \\
	\end{array}
	\right) 
	=-\rho \lambda \left(
	\begin{array}{lll}
	\xi   \\
	\eta \\
	\end{array}
	\right),
	\end{split}
	\end{equation}
along with the interface boundary condition	
	\begin{equation}\label{inter-two-com}
	B^2_\theta\big[k\eta r-\xi'r+\xi \big]\Big|_{r=r_0}=0.
	\end{equation}
\end{prop}
\begin{proof}
	Since we want to use the structure of the energy and properties of functional space, we first change the spectral formula \eqref{spectal formulation-2} into the following equations by a simple computation
	\begin{equation}\label{new-equations}
	\begin{cases}
	&-\frac{d}{dr}\Big\{\gamma p\Big[k\eta-\frac{1}{r}(r\xi)'\Big] \Big\}-\frac{d}{dr}\Big\{B^2_{\theta}\Big[k\eta-\frac{1}{r}\big((r\xi)'-2\xi\big)\Big]\Big\}
	\\
	&\quad-\frac{2B^2_{\theta}}{r}\Big[k\eta-\frac{1}{r}\big((r\xi)'-2\xi\big)\Big]-\frac{2p'\xi}{r}=-\rho \lambda\xi,\\
	&-k(\gamma p+B^2_{\theta})\Big[k\eta-\frac{1}{r}(r\xi)'+\frac{2B^2_\theta}{r(\gamma p+B^2_{\theta})}\xi\Big]=-\rho \lambda\eta.
	\end{cases}
	\end{equation}
	Next, we prove that the minimization $\xi$ and $\eta$ satisfy the equations \eqref{new-equations} in weak sense on $(0,r_0)$. 
	
	Fix $(\xi_0,\eta_0)\in X_{k} $ (see \eqref{definition-Xk}). Define 
	$$j(t,\tau(t))=\mathcal{J}(\xi+t\xi_0+\tau(t)\xi,\eta+t\eta_0+\tau(t)\eta)$$
	and note that $j(0,0)=1$. Moreover, $j$ is smooth, 
	\begin{equation*}
	\begin{split}
	&\frac{\partial j}{\partial t}(0,0)=2\pi^2\int_{0}^{r_0}2\rho(\xi_0\xi +\eta_0\eta)rdr, \\
	&\frac{\partial j}{\partial \tau}(0,0)=2\pi^2\int_{0}^{r_0}2\rho(\xi^2 +\eta^2)rdr=2.
	\end{split}
	\end{equation*}
	So, by the inverse function theorem, we can solve for $\tau=\tau(t)$ in a neighborhood of $0$ as a $C^1$ function of $t$ so that $\tau(0)=0$ and $j(t,\tau(t))=1$. We may differentiate the last equation to find 
	\begin{equation*}
	\frac{\partial j}{\partial t}(0,0)+\frac{\partial j}{\partial \tau}(0,0)\tau'(0)=0,
	\end{equation*}
	hence that 
	\begin{equation*}
	\tau'(0)=-\frac{1}{2}\frac{\partial j}{\partial t}(0,0)=-2\pi^2\int_{0}^{r_0}\rho(\xi_0\xi +\eta_0\eta)rdr.
	\end{equation*}
	Since $(\xi, \eta)$ are minimizers over $\mathcal{A}_1$, we may make variations with respect to $(\xi_0,\eta_0)$ to find that 
	\begin{equation*}
	0=\frac{d}{dt}\bigg|_{t=0}E_{0,k}(\xi+t\xi_0+\tau(t)\xi, \eta+t\eta_0+\tau(t)\eta),
	\end{equation*}
	which implies that
	\begin{equation*}
	\begin{split}
	0&=4\pi^2\int_{0}^{r_0}2p'\xi(\xi_0+\tau'(0)\xi)dr +4\pi^2\int_{0}^{r_0}B^2_{\theta}\Big[k\eta-\frac{1}{r}\big((r\xi)'-2\xi\big)\Big]\\
	&\quad\times \bigg\{k\eta_0+\tau'(0)k\eta-\frac{1}{r}\Big[\Big(r\big(\xi_0+\tau'(0)\xi\big)\Big)'-2\big(\xi_0+\tau'(0)\xi\big)\Big]\bigg\}rdr\\
	&\quad +4\pi^2\int_{0}^{r_0}\gamma p\Big[\frac{1}{r}(r\xi)'-k\eta\Big]\Big[\frac{1}{r}\Big(r\big(\xi_0+\tau'(0)\xi\big)\Big)'-k\big(\eta_0+\tau'(0)\eta\big)\Big]rdr\\
	&=4\pi^2\int_{0}^{r_0}2p'\xi\xi_0dr +4\pi^2\int_{0}^{r_0}B^2_{\theta}\Big[k\eta-\frac{1}{r}\big((r\xi)'-2\xi\big)\Big] \Big[k\eta_0-\frac{1}{r}\Big(\big(r\xi_0\big)'-2\xi_0\Big)\Big]rdr\\
	&\quad+4\pi^2\int_{0}^{r_0}\gamma p\Big[\frac{1}{r}(r\xi)'-k\eta\Big]\Big[\frac{1}{r}\big(r\xi_0\big)'-k\eta_0\Big]rdr\\
	&\quad+4\tau'(0)\pi^2\int_0^{r_0}\Big\{2p'\xi^2+B_{\theta}^2\Big[k\eta-\frac{1}{r}((r\xi)'-2\xi)\Big]^2r+\gamma p\Big[\frac{1}{r}(r\xi)'-k\eta\Big]^2r\Big\}dr.
	\end{split}
	\end{equation*}
	Hence, we have
	\begin{equation*}
	\begin{split}
	-\tau'(0)\lambda&=2\pi^2\int_{0}^{r_0}\rho\lambda(\xi_0\xi +\eta_0\eta)rdr\\
	&=2\pi^2\int_{0}^{r_0}2p'\xi\xi_0dr +2\pi^2\int_{0}^{r_0}B^2_{\theta}\Big[k\eta-\frac{1}{r}\big((r\xi)'-2\xi\big)\Big]
	\Big[-\frac{1}{r}\Big(\big(r\xi_0\big)'-2\xi_0\Big)\Big]rdr\\
	&\quad+2\pi^2\int_{0}^{r_0}\gamma p\Big[\frac{1}{r}(r\xi)'-k\eta\Big](r\xi_0)'dr+2\pi^2\int_{0}^{r_0}kB^2_{\theta}\Big[k\eta-\frac{1}{r}\big((r\xi)'-2\xi\big)\Big]
	\eta_0rdr\\
	&\quad+2\pi^2\int_{0}^{r_0}\gamma p\Big[\frac{1}{r}(r\xi)'-k\eta\Big](-k\eta_0)rdr.
	\end{split}
	\end{equation*}
	
	Since $\xi_0$ and $\eta_0$ are independent,  we deduce that
	\begin{equation}\label{weak-form}
	\begin{split}
	&\int_{0}^{r_0}2p'\xi\xi_0dr -\int_{0}^{r_0}B^2_{\theta}\Big[k\eta-\frac{1}{r}\big((r\xi)'-2\xi\big)\Big]
	\Big[\big(r\xi_0\big)'-2\xi_0\Big]dr\\
	&\quad+\int_{0}^{r_0}\gamma p\Big[\frac{1}{r}(r\xi)'-k\eta\Big](r\xi_0)'dr=\int_{0}^{r_0}\rho\lambda\xi_0\xi rdr,
	\end{split}
	\end{equation}	
	\begin{equation*}
	\begin{split}
	&\int_{0}^{r_0}kB^2_{\theta}\Big[k\eta-\frac{1}{r}\big((r\xi)'-2\xi\big)\Big]
	\eta_0rdr+\int_{0}^{r_0}\gamma pk\Big[k\eta-\frac{1}{r}(r\xi)'\Big]\eta_0rdr\\
	&\quad= \int_{0}^{r_0}\rho\lambda\eta_0\eta rdr.
	\end{split}
	\end{equation*}
	Therefore, $\xi$ and $\eta$ satisfy \eqref{new-equations}  in weak sense on $(0,r_0)$.
	Now we prove that the interface boundary condition \eqref{inter-two-com} is satisfied. 
	From the first equation of \eqref{new-equations}, we get 
	\begin{equation}\label{p-xi}
	\begin{split}
	&-\frac{d}{dr}\Big\{(\gamma p+B^2_{\theta})\Big[k\eta-\frac{1}{r}\big((r\xi)'-2\xi\big)\Big]\Big\}+\frac{d}{dr}\Big(\gamma p \frac{2\xi}{r}\Big)\\
	&\quad-\frac{2B^2_{\theta}}{r}\Big[k\eta-\frac{1}{r}\big((r\xi)'-2\xi\big)\Big]-\frac{2p'\xi}{r}=-\rho \lambda\xi,
	\end{split}
	\end{equation}
	that is, 
	\begin{equation}
	\begin{split}
	&-\frac{d}{dr}\Big\{(\gamma p+B^2_{\theta})\Big[k\eta-\frac{1}{r}\big((r\xi)'-2\xi\big)\Big]\Big\}+ \frac{2\gamma p\xi'}{r}+ \frac{2\gamma p'\xi}{r}-\frac{2\gamma p \xi }{r^2} \\
	&\quad-\frac{2B^2_{\theta}}{r}\Big[k\eta-\frac{1}{r}\big((r\xi)'-2\xi\big)\Big]-\frac{2p'\xi}{r}=-\rho \lambda\xi,
	\end{split}
	\end{equation}
	which together with $(\xi,\eta)\in X_k$ (see \eqref{definition-Xk}), $\xi \in Z$ (see \eqref{defi-z}) and the claim $\sqrt{p}\xi' \in L^2(\frac{r_0}{2}, r_0)$ gives that 
	\begin{equation}
	\frac{d}{dr}\Big\{(\gamma p+B^2_{\theta})\Big[k\eta-\frac{1}{r}\big((r\xi)'-2\xi\big)\Big]\Big\}\in L^2(\frac{r_0}{2}, r_0).
	\end{equation}	
	Now we prove the claim  $\sqrt{p}\xi' \in L^2(\frac{r_0}{2}, r_0)$. In fact,  from $(\xi,\eta)\in X_k$, it follows that
	\begin{equation}
	\begin{split}
	\|(\xi,\eta)\|^2_{X_k(\frac{r_0}{2},r_0)}&=\int_{\frac{r_0}{2}}^{r_0}\bigg\{p\Big|\frac{1}{r}\partial_r(r\xi(r))\Big|^2+B_{\theta}^2\Big[k\eta-\frac{1}{r}\big((r\xi)'-2\xi\big)\Big]^2\bigg\}rdr\\
	&\quad	+\int_{\frac{r_0}{2}}^{r_0} \rho (|\xi|^2+|\eta|^2)rdr\leq C.
	\end{split}
	\end{equation}
Together with $\frac{1}{r}\partial_r(r\xi(r))=\frac{\xi}{r}+\xi'$, we deduce that  $\sqrt{p}\xi' \in L^2(\frac{r_0}{2}, r_0)$.	
	So $(\gamma p+B^2_{\theta})\Big[k\eta-\frac{1}{r}\big((r\xi)'-2\xi\big)\Big]$ is well-defined at the endpoint $r=r_0$. 
	Make variations with respect to $\xi_0\in C_c^{\infty}((0,r_0])$. Integrating the terms in \eqref{weak-form}  by parts and using that $\xi$ solves the first equation of \eqref{new-equations} on $(0,r_0)$, we obtain
	\begin{equation}\label{new-boundary}
	\begin{split}
	-B^2_\theta\Big[k\eta-\frac{1}{r}\big((r\xi)'-2\xi\big)\Big](r\xi_0)\Big|_{r=r_0}+\gamma p\Big[\frac{1}{r}(r\xi)'-k\eta \Big](r\xi_0)\Big|_{r=r_0}=0,
	\end{split}
	\end{equation}
	which implies by $p=0$ on the boundary $r=r_0$ that 
	\begin{equation}\label{new-boundary-2}
	\begin{split}
	-B^2_\theta\Big[k\eta-\frac{1}{r}\big((r\xi)'-2\xi\big)\Big](r\xi_0)\Big|_{r=r_0}=0.
	\end{split}
	\end{equation}
	Since $\xi_0$ may be chosen arbitrarily, we get the interface boundary condition
	\begin{equation}\label{new-boiundary-3}
	\begin{split}
	B^2_\theta\big[k\eta r-\xi'r+\xi \big]\Big|_{r=r_0}=0.
	\end{split}
	\end{equation}
	This completes the proof.
\end{proof}
\subsection{Variational problem  when $m\neq 0$}
In this subsection, we prove the Kink ($|m|=1$) instability, in fact, we give the analysis for any  $m\neq0$ and any $k\in \mathbb{Z}$. Let's first introduce the definition of the space $Y_{m,k}$.
\begin{defi} \label{space-y}
		The weighted Sobolev space $Y_{m,k}$ is defined as the completion of $\{(\xi,\eta,\zeta) \in C^{\infty}([0,r_0])\times C^{\infty}([0,r_0])\times C^{\infty}([0,r_0])|\xi(0)=0)\}$,  with respect to the norm
\begin{equation}\label{definition-ymk}
\begin{split}
\|\big(\xi,\eta,\zeta\big)\|^2_{Y_{m,k}}&
=\int_0^{r_0}p\Big[\frac{1}{r}(r\xi)'+\frac{m\zeta}{r}\Big]^2rdr+\int_0^{r_0}\frac{B^2_\theta}{r}\xi^2dr\\
&\quad+\int_0^{r_0}B_{\theta}^2\Big[\frac{\eta}{r}-\frac{k}{m^2+k^2r^2}\big((r\xi)'-2\xi\big)\Big]^2rdr\\
&\quad+\int_0^{r_0}\frac{B^2_\theta}{r}(\xi-r\xi')^2dr+\int_0^{r_0} \rho \big(|\xi|^2+|\eta|^2+|\zeta|^2\big)rdr.
\end{split}
\end{equation} 
\end{defi}

From Definition \ref{space-y}, we can get the following compactness result.
\begin{lem}\label{embedding-p'-general-m}
Assume $s_1$ is near $r=r_0$ and $p$ is admissible.  Let $\Pi_1$ denote the projection operator onto the first factor. Then $\Pi_1: Y_{m,k}\rightarrow V$ is a bounded, linear, compact map, with the norm
 \begin{equation}\label{defi-v}
\|\xi\|_V^2=\int_{0}^{s_1}B^2_\theta\xi^2dr+\int_{s_1}^{r_0}-p'\xi^2dr,
\end{equation} and we denote  it by 
	\begin{equation}\label{defi-b-v}
	Y_{m,k}\subset\subset V.
	\end{equation}
\end{lem}
\begin{proof}
For any $(\xi,\eta,\zeta)\in Y_{m,k}$, we have for $r\in (0,\frac{r_0}{2})$ that
\begin{equation*}
\begin{split}
|r\xi(r)B_\theta(r)|&= \Big|\int_0^r\partial_{s}\big(s\xi(s)B_\theta(s)\big)ds\Big|\\
&\leq \Big|\int_0^rs\xi'(s)B_\theta(s)ds\Big|+\Big|\int_0^rs\xi(s)B'_\theta(s)ds\Big|+\Big|\int_0^r\xi(s)B_\theta(s)ds\Big|\\
&\leq \Big(\int_0^rB^2_\theta(s)(\xi'(s))^2sds\Big)^\frac{1}{2}\Big(\int_0^rsds\Big)^\frac{1}{2}+\Big(\int_0^r\big(B'_\theta(s))^2sds\Big)^\frac{1}{2}\Big(\int_0^r\xi^2(s)sds\Big)^\frac{1}{2}\\
&\quad+\Big(\int_0^r\frac{B^2_\theta(s)}{s}\xi^2(s)ds\Big)^\frac{1}{2}\Big(\int_0^rsds\Big)^\frac{1}{2}\\
&\leq \frac{r}{\sqrt{2}}\Big(\int_0^rB^2_\theta(s)(\xi'(s))^2sds\Big)^\frac{1}{2}+\frac{r}{\sqrt{2}}\|B'_\theta\|_{L^{\infty}}\Big(\int_0^r\xi^2(s)sds\Big)^{\frac{1}{2}}\\
&\quad+\frac{r}{\sqrt{2}}\Big(\int_0^r\frac{B^2_\theta(s)}{s}\xi^2(s)ds\Big)^\frac{1}{2},
\end{split}
\end{equation*}
which gives that 
\begin{equation}\label{sup-xi}
\begin{split}
|\xi(r)B_\theta(r)|&\leq \frac{1}{\sqrt{2}}\Big(\int_0^rB^2_\theta(s)(\xi'(s))^2sds\Big)^\frac{1}{2}+\frac{1}{\sqrt{2}}\|B'_\theta\|_{L^{\infty}}\Big(\int_0^r\xi^2(s)sds\Big)^{\frac{1}{2}}\\
&\quad+\frac{1}{\sqrt{2}}\Big(\int_0^r\frac{B^2_\theta(s)}{s}\xi^2(s)ds\Big)^\frac{1}{2}.
\end{split}
\end{equation}
Assume that $\|(\xi_n,\eta_n,\zeta_n)\|_{Y_{m,k}}\leq C$, for $n\in \mathbb{N}$,
then we have 
\begin{equation*}
\begin{split}
&\|\big(\xi_n(r),\eta_n(r),\zeta_n(r)\big)\|^2_{Y(0,\frac{r_0}{2})}\\&
=\int_0^{\frac{r_0}{2}}p\Big[\frac{1}{r}(r\xi_n)'+\frac{m\zeta_n}{r}\Big]^2rdr+\int_0^{\frac{r_0}{2}}\frac{B^2_\theta}{r}\xi_n^2dr\\
&\quad+\int_0^{\frac{r_0}{2}}B_{\theta}^2\Big[\frac{\eta_n}{r}-\frac{k}{m^2+k^2r^2}\big((r\xi_n)'-2\xi_n\big)\Big]^2rdr\\
&\quad+\int_0^{\frac{r_0}{2}}\frac{B^2_\theta}{r}(\xi_n-r\xi'_n)^2dr+\int_0^{\frac{r_0}{2}} \rho (|\xi_n|^2+|\eta_n|^2+|\zeta_n|^2)rdr\leq C^2,
\end{split}
\end{equation*} 
which implies that for $r\in (0,\frac{r_0}{2})$,
\begin{equation}\label{sup-xi-2}
\begin{cases}
&\int_0^rB^2_\theta(s)(\xi'_n(s))^2sds\leq C^2,\\
& \int_0^r\xi_n^2(s)sds\leq \frac{1}{\min_{(0,r)}{\rho}}C^2,\\
& \quad\int_0^r\frac{B^2_\theta(s)}{s}\xi_n^2(s)ds\leq C^2.
\end{cases}
\end{equation}
Fix any $\kappa>0$.
We claim that there exists a subsequence $\{\xi_{n_i}\}$ so that 
\begin{equation}\label{claim-2}
\sup_{i,j}\|\xi_{n_i}-\xi_{n_j}\|_{V}\leq \kappa.
\end{equation}
To prove the claim, from \eqref{sup-xi} and \eqref{sup-xi-2}, let $s_0\in (0,\frac{r_0}{2})$ be chosen small enough so that 
\begin{equation}\label{sup-xi-b-g}
3s_0\Big(2C^2+\frac{1}{\min_{(0,s_0)}{\rho}}\|B'_\theta\|_{L^{\infty}}^2C^2\Big)\leq \kappa.
\end{equation}

	From Definition \ref{admissible}/admissibility of $p$, we have $\frac{p}{p'}\rightarrow 0$ for $r \rightarrow r_0$, which together with the definition of $g$, gives that  $g(s_1)\rightarrow 0$ as $s_1\rightarrow r_0$. Choose $s_1$ close enough to $r_0$, such that 
	\begin{equation}\label{fix-s_1}
g(s_1)C\leq \frac{\kappa}{6},\quad\frac{Cp(s_1)}{3(s_1-s_0)\|B_\theta\|^2_{L^{\infty}(0,r_0)}}\leq 
\frac{1}{6}.
	\end{equation}
	Since the subinterval $(s_0,s_1)$ avoids the singularity of $\frac{1}{r}$ and the degenerate of pressure $p$ on the boundary $r=r_0$,
the function $\xi_n$ is uniformly bounded in $H^1(s_0,s_1)$. By the compact embedding $H^1(s_0,s_1)\subset\subset C^0(s_0,s_1)$, one can extract a subsequence $\{\xi_{n_i}\}$ that converges in $L^\infty(s_0,s_1)$. So for $i, j$ large enough, it holds that 
\begin{equation*}
\sup_{i,j}\|\xi_{n_i}-\xi_{n_j}\|^2_{L^\infty(s_0,s_1)}\leq \frac{\kappa}{3(s_1-s_0)\|B_\theta\|^2_{L^{\infty}(0,r_0)}}.
\end{equation*}
Since  $p'(r)\leq 0$ near $r=r_0$, by Lemma \ref{weight-p'-lem} and \eqref{fix-s_1},  we deduce for $i$ and $j$ large enough
\begin{equation}
\begin{split}
\int_{s_1}^{r_0}-p'(r)(\xi_{n_i}-\xi_{n_j})^2dr
&\leq 2p(s_1)(\xi_{n_i}-\xi_{n_j})^2(s_1)+4g(s_1)
\int_{s_1}^{r_0}p(\xi'_{n_i}-\xi'_{n_j})^2dr\\
&\leq \frac{Cp(s_1)\kappa}{3(s_1-s_0)\|B_\theta\|^2_{L^{\infty}(0,r_0)}}+g(s_1)C\leq \frac{\kappa}{3}.
\end{split}
\end{equation}

Then along the above subsequence we can get from  \eqref{sup-xi-b-g} that 
\begin{equation*}
\begin{split}
\|\xi_{n_i}-\xi_{n_j}\|^2_{V}&=\int_0^{s_1}B^2_\theta|\xi_{n_i}-\xi_{n_j}|^2dr
+\int_{s_1}^{r_0}-p'|\xi_{n_i}-\xi_{n_j}|^2dr\\
&=\bigg(\int_0^{s_0}+\int_{s_0}^{s_1}\bigg) B^2_\theta|\xi_{n_i}-\xi_{n_j}|^2dr+\int_{s_1}^{r_0}-p'|\xi_{n_i}-\xi_{n_j}|^2dr\\
&\leq s_0\Big(2C^2+\frac{1}{\min_{(0,s_0)}{\rho}}\|B'_\theta\|_{L^{\infty}}^2C^2\Big)\\
&\quad+(s_1-s_0)\sup_{i,j}\|\xi_{n_i}-\xi_{n_j}\|^2_{L^\infty(s_0,r_0)}\|B_\theta\|^2_{L^\infty(0,r_0)}\\
&\quad+\frac{Cp(s_1)\kappa}{3(s_1-s_0)\|B_\theta\|^2_{L^{\infty}(0,r_0)}}+g(s_1)C\leq \kappa,
\end{split}
\end{equation*}
which implies the claim \eqref{claim-2} and the compactness result \eqref{defi-b-v}. 
\end{proof}
Now, consider the case  any  $m\neq0$ and any $k\in \mathbb{Z}$. 
We need to consider the energy \eqref{sausage-in-v-t-out} in Proposition \ref{energy-prop} 
and
\begin{equation}\label{constra}
\mathcal{J}(\xi,\eta,\zeta)=2\pi^2\int_0^{r_0} \rho (|\xi|^2+|\eta|^2+|\zeta|^2)rdr.
\end{equation}
Then from the Definition \ref{space-y}, we can get the following lemma. 
\begin{lem}
	$E_{m,k}(\xi,\eta,\zeta,\widehat{ Q}_r)$ defined in \eqref{sausage-in-v-t-out}   and $\mathcal{J}(\xi,\eta,\zeta)$ are both well defined on the space $Y_{m,k}\times H^1(r_0,r_w)$. 
\end{lem}
\begin{proof}
From \eqref{taylor-p} and Lemma \ref{weight-p'-lem}, 
similar to the proof of Lemma \ref{well-define}, we can get
	\begin{equation*}
\int_0^{r_0}2p'\xi^2dr\leq  C\mathcal{J}(\xi,\eta,\zeta)+C\|\big(\xi,\eta,\zeta\big)\|^2_{Y_{m,k}},
	\end{equation*}
	which implies that
	\begin{equation*}
	\begin{split}
	|E_{m,k}(\xi,\eta,\zeta,\widehat{ Q}_r)|
	&\leq C\mathcal{J}(\xi,\eta,\zeta)+C\|(\xi,\eta,\zeta)\|_{Y_{m,k}}^2+C\int_0^{r_0}B_{\theta}^2\Big[\frac{\eta}{r}-\frac{k}{m^2+k^2r^2}\big((r\xi)'-2\xi\big)\Big]^2rdr\\&\quad+C\int_0^{r_0}p\Big[\frac{1}{r}(r\xi)'+\frac{m\zeta}{r}\Big]^2rdr+C\int_0^{r_0}\frac{B^2_\theta}{r}\xi^2dr+C\int_0^{r_0}\frac{B^2_\theta}{r}(\xi-r\xi')^2dr\\
	&\quad+ C\|\rho\|^{\gamma-1}_{L^\infty}\int_0^{r_0} \rho |\eta|^2rdr+C\int_{r_0}^{r_w}\bigg[|\widehat{Q}_r|^2+|\widehat{Q}_r'|^2\bigg]rdr\\
	&\leq C\|\big(\xi,\eta,\zeta\big)\|^2_{Y_{m,k}}+\|\widehat{ Q}_r\|^2_{H^1}.
	\end{split}
	\end{equation*}
	Therefore, $E_{m,k}(\xi,\eta,\zeta)$ and $\mathcal{J}(\xi,\eta,\zeta)$ are well defined on the space $Y_{m,k}\times H^1(r_0,r_w)$.
\end{proof}

Now, we define \begin{equation}\label{e-3}
\lambda=\inf_{\big((\xi,\eta,\zeta),\widehat{ Q}_r\big
	)\in Y_{m,k}\times H^1(r_0,r_w)}\frac{E_{m,k}(\xi,\eta,\zeta,\widehat{ Q}_r)}{\mathcal{J}(\xi,\eta,\zeta)}.
\end{equation}
Consider the set 
\begin{equation}\label{set-A-m}
\mathcal{A}_2=\left\{
\begin{aligned}
\big((\xi,\eta,\zeta),\widehat{ Q}_r\big
)\in Y_{m,k}\times H^1(r_0,r_w)| \mathcal{J}(\xi,\eta,\zeta)=1,\,\\ m\widehat{  B}_\theta \xi=r\widehat{ Q}_r\, \,\mbox{at}\, \,r=r_0\,\, \mbox{and}\,\, \widehat{ Q}_r=0\, \,\mbox{at}\, \,r=r_w
\end{aligned}
\right\},
\end{equation}
where the functions $\xi$, $\eta$ and $\zeta$ are restricted to $(0,r_0)$, and the function $\widehat{ Q}_r$ is restricted to $(r_0,r_w)$.
We want to show that the infimum of $E_{m,k}(\xi,\eta,\zeta,\widehat{ Q}_r)$ over the set $\mathcal{A}_2$ is achieved and is negative and that the minimizer solves \eqref{spectral-formulation} and \eqref{euler-l-q} with the corresponding boundary conditions. 
First, we study the lower bound of the energy $E_{m,k}(\xi,\eta,\zeta,\widehat{ Q}_r)$ on the set $ \mathcal{A}_2$.
\begin{lem}\label{lower-bound-lem}
The energy $E_{m,k}(\xi,\eta,\zeta,\widehat{ Q}_r)$ has a lower bound on the set $ \mathcal{A}_2$.
\end{lem}
\begin{proof}
We can directly get  from the energy \eqref{sausage-in-v-t-out} that for $0<s_0<s_1<r_0$, 
\begin{equation*}
\begin{split}
E_{m,k}(\xi,\eta,\zeta,\widehat{ Q}_r)&\geq2\pi^2\int_0^{r_0}\gamma p\Big[\frac{1}{r}(r\xi)'-k\eta+\frac{m}{r}\zeta\Big]^2rdr+2\pi^2 \int_0^{r_0}2p'\xi^2dr\\
&\geq 2\pi^2\int_{s_0}^{r_0}\gamma p\Big[\frac{1}{r}(r\xi)'-k\eta+\frac{m}{r}\zeta\Big]^2rdr\\
&\quad+2\pi^2\Big(\int_0^{s_1}+\int_{s_1}^{r_0}\Big)2p'\xi^2dr,\,\, \forall\, (\xi,\eta,\zeta)\in \mathcal{A}_2.
\end{split}
\end{equation*}
Recalling  \eqref{lower-bound}, for  every $s_1<r_0$, we have
$\Big|2\pi^2\int_0^{s_1}2p'\xi^2dr\Big|\leq C\mathcal{J}.$
Hence, we get
\begin{equation*}
\begin{split}
E_{m,k}(\xi,\eta,\zeta,\widehat{ Q}_r)\geq
 2\pi^2\int_{s_0}^{r_0}\gamma p\Big|\frac{1}{r}(r\xi)'\Big|^2rdr+2\pi^2\int_{s_1}^{r_0}2p'\xi^2dr-C\mathcal{J}.
\end{split}
\end{equation*}
	The key is to control $\int_{s_1}^{r_0}2p'\xi^2dr$. 
		Since in the interval $(s_1,r_0)$, 
	using  Lemma \ref{weight-p'-lem}, similarly as \eqref{estimates-p'-2}, we know that
	\begin{equation}
	\begin{split}
	\Big|\int_{s_1}^{r_0}2p'(r)\xi^2dr	\Big|&=	\int_{s_1}^{r_0}-2p'\xi^2dr
	\leq Cp(s_1)\xi^2(s_1)+Cg(s_1)
	\int_{s_1}^{r_0}p\xi'^2dr\\
	&\leq C(\sigma)\mathcal{J}p(s_1)+\Big(Cp(s_1)\sigma+Cg(s_1)\Big)\int_{s_0}^{r_0}p\xi'^2dr.
	\end{split}
	\end{equation}	
	Choosing $s_1$ close enough to $r_0$ and $\sigma$ small enough such that  $Cp(s_1)\sigma+Cg(s_1)\leq \frac{\gamma}{2}$, yields that
\begin{equation*}
\begin{split}
E_{m,k}(\xi,\eta,\zeta,\widehat{ Q}_r)\geq\pi^2\int_{s_0}^{r_0}\gamma p\Big|\frac{1}{r}(r\xi)'\Big|^2rdr-3C\mathcal{J}\geq -3C\mathcal{J}=-3C,
\end{split}
\end{equation*}
 which implies that the energy $E_{m,k}(\xi,\eta,\zeta,\widehat{ Q}_r)$ has a lower bound on the set $ \mathcal{A}_2$. 
\end{proof}
Now we prove the corecivity estimate.
	Using the fact that	$\mathcal{J}(\xi,\eta, \zeta
	)=1$  and $E_{m,k}$ has a lower bound on the set $\mathcal{A}_2$, we can choose a minimizing sequence such that along the minimizing sequence, 
	we know that $M\leq E_{m,k}(\xi_{n},\eta_n,\zeta_n,\widehat{ Q}_{rn})<M+1$, and therefore 
	we have coercivity estimate: 
		\begin{equation}\label{coercivity-mneq0}
	\begin{split}
	\|\big(\xi_n,\eta_n,\zeta_n\big)\|^2_{Y_{m,k}}\leq \mathcal{J}+C(M+1)+C\int_0^{r_0}\Big(2p'\xi_n^2+\gamma pk^2\eta_n^2r\Big)dr\leq C.
	\end{split}
	\end{equation}

Next, we prove that the infimum of $E_{m,k}(\xi,\eta,\zeta,\widehat{ Q}_r)$ over the set $\mathcal{A}_2$ is negative.
\begin{prop}\label{infimum-A-out-3}
If there exists $r^*\in (0,r_0)$ such that $2p'(r^*)+\frac{m^2B_\theta^2(r^*)}{r^*}<0$, then
it holds that $\lambda=\inf E_{m,k}<0$. 
\end{prop}
\begin{proof}
	Since both $E_{m,k}$ and $\mathcal{J}$ are homogeneous degree $2$, it suffices to show that 
	\begin{equation*}
	\inf_{((\xi,\eta,\zeta),\widehat{ Q}_r))\in Y_{m,k} \times H^1(r_0,r_w)}\frac{E_{m,k}(\xi,\eta,\zeta,\widehat{ Q}_r)}{\mathcal{J}(\xi,\eta,\zeta)}<0.
	\end{equation*}
	But since $\mathcal{J}$ is positive definite, one may reduce to constructing any $((\xi,\eta,\zeta),\widehat{ Q}_r))\in Y_{m,k} \times H^1(r_0,r_w)$ such that $$E_{m,k}(\xi,\eta,\zeta,\widehat{ Q}_r)<0.$$
If there exists $r^*\in (0,r_0)$ such that $2p'+\frac{m^2B_\theta^2}{r}<0$, then
we can choose a smooth function $\xi^*\in C_c^{\infty}(0,r_0)$ such that $$2\pi^2\int_0^{r_0}\Big[2p'+\frac{m^2B_{\theta}^2}{r}\Big]{\xi^*}^2dr< 0.$$
	Then, 
	we can assume that $\eta^*=\frac{rk(r\xi^*)'-2rk\xi^*}{m^2+k^2r^2}$ and $\zeta^*=\frac{r}{m}(k\eta^*-\frac{1}{r}(r\xi^*)')$, such that the first and second terms of $E_{m,k}(\xi^*,\eta^*,\zeta^*,\widehat{ Q}_r^*)$ in \eqref{sausage-in-v-t-out} vanish,
	that is,
	\begin{equation*}
	\begin{split}
	&	2\pi^2\int_0^{r_0}(m^2+k^2r^2)\Big[\frac{B_\theta}{r}\eta^*+\frac{-kB_\theta(r\xi^*)'+2kB_{\theta}\xi^*}{m^2+k^2r^2}\Big]^2rdr=0,\\
	&
	2\pi^2\int_0^{r_0}\gamma p\Big[\frac{1}{r}(r\xi^*)'-k\eta^*+\frac{m\zeta^*}{r}\Big]^2rdr=0.
	\end{split}
	\end{equation*} Here, $\xi^*$, $\eta^*$ and $\zeta^*$ are smooth functions and belong to the space $Y_{m,k}$.
	
	Minimizing energy $E_{m,k}(\xi,\eta,\zeta,\widehat{ Q}_r)$ with respect to $\eta^* $ and $\zeta^*$, 	so we consider the limit $k\rightarrow \infty$, then we can get that
	\begin{equation*}
	\begin{split}
	\widetilde{E}(\xi^*,\widehat{ Q}_r^*)&=\lim_{k\rightarrow \infty}\bigg(2\pi^2\int_0^{r_0}\frac{m^2B_\theta^2}{r(m^2+k^2r^2)}(\xi^*-r{\xi^*}')^2 +2\pi^2\int_0^{r_0}\Big[2p'+\frac{m^2B_{\theta}^2}{r}\Big]{\xi^*}^2dr\\
	&\quad-2\pi^2\Big[\frac{r}{m^2+k^2r^2}(r\widehat{ Q}_r^*)'r\widehat{ Q}_r^*\Big]_{r=r_0}\bigg)\\
	&=2\pi^2\int_0^{r_0}\Big[2p'+\frac{m^2B_{\theta}^2}{r}\Big]{\xi^*}^2dr<0,
	\end{split}
	\end{equation*} 
	which implies the result.	
\end{proof}
Using Proposition \ref{infimum-A-out-3}, we can achieve the minimizer of the energy $E_{m,k}(\xi,\eta,\zeta,\widehat{ Q}_r)$.
\begin{prop}\label{infimum-A-out}
If $\lambda=\inf E_{m,k}<0$, then $E_{m,k}(\xi,\eta,\zeta,\widehat{ Q}_r)$ achieves its infimum on the set  $\mathcal{A}_2$.	
\end{prop}
\begin{proof}
	First from Lemma \ref{lower-bound-lem}, we have that $E_{m,k}(\xi,\eta,\zeta,\widehat{ Q}_r)$ is bounded below on the set $\mathcal{A}_2$. Assume that  $((\xi_n,\eta_n,\zeta_n),\widehat{ Q}_{rn}) \in \mathcal{A}_2$ be a minimizing sequence. Then $(\xi_n,\eta_n,\zeta_n)$ is bounded in $Y_{m,k}$ (see \eqref{definition-ymk}),  and $ \widehat{ Q}_{rn}$ is bounded in $ H^1(r_0,r_w)$, so up to the extraction of a subsequence $\psi_n=\sqrt{m^2+k^2r^2}|B_{\theta}|\Big[\frac{\eta_n}{r}-\frac{k}{m^2+k^2r^2}((r\xi_n)'-2\xi_n)\Big]r^{\frac12}\rightharpoonup \psi=\sqrt{m^2+k^2r^2}|B_{\theta}|\Big[\frac{\eta}{r}-\frac{k}{m^2+k^2r^2}((r\xi)'-2\xi)\Big]r^{\frac12}$ weakly in $L^2$, and $(\xi_n,\widehat{ Q}_{rn})\rightarrow (\xi,\widehat{ Q}_{r})$ strongly in $V\times L^2(r_0,r_w)$ from the compact embedding in Lemma \ref{embedding-p'-general-m} and  the compact embedding $H^1(r_0,r_w)\subset\subset L^2(r_0,r_w)$.
	By weak lower semi-continuity, since $\psi_n\rightharpoonup \psi$ in the space $L^2(0,r_0)$,
		we have 
		\begin{equation*}
		\begin{split}
		&\int_0^{r_0}(m^2+k^2r^2)B^2_{\theta}\Big[\frac{\eta}{r}-\frac{k}{m^2+k^2r^2}((r\xi)'-2\xi)\Big]^2rdr\\
		&\quad\leq\liminf_{n\to\infty}\int_0^{r_0}(m^2+k^2r^2)B^2_{\theta}\Big[\frac{\eta_n}{r}-\frac{k}{m^2+k^2r^2}((r\xi_n)'-2\xi_n)\Big]^2rdr.
		\end{split}
		\end{equation*}

	Because of the quadratic structure of all the terms in the integrals defining $E_{m,k}$, similarly by weak lower semi-continuity,
	we have 
		\begin{equation*}
	\begin{split}
	\int_0^{r_0}\gamma p\Big[\frac{1}{r}(r\xi)'-k\eta+\frac{m}{r}\zeta\Big]^2rdr\leq\liminf_{n\to\infty}\int_0^{r_0}\gamma p\Big[\frac{1}{r}(r\xi_n)'-k\eta_n+\frac{m}{r}\zeta_n\Big]^2rdr.
	\end{split}
	\end{equation*}	
	Now let us deal with $\int_0^{r_0}\frac{m^2B_\theta^2}{r(m^2+k^2r^2)}(\xi_n-r\xi_n')^2 +\int_0^{r_0}\Big[2p'+\frac{m^2B_{\theta}^2}{r}\Big]\xi_n^2dr$ by the following two cases.

{\bf Case I: }	When $|m|=1$, we can write for every $s_0>0$, 
	\begin{equation}
	\begin{split}
&	\int_0^{r_0}\frac{m^2B_\theta^2}{r(m^2+k^2r^2)}(\xi_n-r\xi_n')^2dr +\int_0^{r_0}\Big[2p'+\frac{m^2B_{\theta}^2}{r}\Big]\xi_n^2dr\\
&=\Big(\int_0^{s_0} +\int_{s_0}^{r_0}\Big)\frac{m^2B_\theta^2}{r(m^2+k^2r^2)}(\xi_n-r\xi_n')^2dr +\Big(\int_0^{s_0} +\int_{s_0}^{r_0}\Big)\Big[2p'+\frac{m^2B_{\theta}^2}{r}\Big]\xi_n^2dr.
	\end{split}
	\end{equation}
Applying \eqref{taylor-B} and \eqref{taylor-p} in Lemma \ref{expansion-zero}, for fixed $k$ we have
\begin{equation*}
\begin{split}
&	\int_0^{s_0}\frac{m^2B_\theta^2}{r(m^2+k^2r^2)}(\xi_n-r\xi_n')^2dr +\int_0^{s_0}\Big[2p'+\frac{m^2B_{\theta}^2}{r}\Big]\xi_n^2dr\\
&=	\int_0^{s_0}\Big[\frac{1}{4}\mathbb{J}^2_z(0)r
+\frac{1}{3}\mathbb{J}'_z(0)\mathbb{J}_z(0)r^2+O(r^3)\Big](\xi_n-r\xi_n')^2dr\\
&\quad+\int_0^{s_0}\Big[-\frac{3}{4}\mathbb{J}^2_z(0)r
-\frac{4}{3}\mathbb{J}'_z(0)\mathbb{J}_z(0)r^2+O(r^3)\Big]\xi_n^2dr\\
&=\int_0^{s_0}\Big[\frac{1}{4}\mathbb{J}^2_z(0)r+\frac{1}{3}\mathbb{J}'_z(0)\mathbb{J}_z(0)r^2+O(r^3)\Big]\Big(r^2\xi_n'^2-2r\xi_n\xi_n'\Big)dr\\
&\quad+\int_0^{s_0}\Big[-\frac{1}{2}\mathbb{J}^2_z(0)r
-\mathbb{J}'_z(0)\mathbb{J}_z(0)r^2+O(r^3)\Big]\xi_n^2dr\\
&=\int_0^{s_0}\Big[\frac{1}{4}\mathbb{J}^2_z(0)r+\frac{1}{3}\mathbb{J}'_z(0)\mathbb{J}_z(0)r^2+O(r^3)\Big]r^2\xi_n'^2dr\\
&\quad+\int_0^{s_0}\Big[\frac{1}{2}\mathbb{J}^2_z(0)r+\mathbb{J}'_z(0)\mathbb{J}_z(0)r^2\Big]\xi_n^2dr\\
&\quad+\int_0^{s_0}\Big[-\frac{1}{2}\mathbb{J}^2_z(0)r
-\mathbb{J}'_z(0)\mathbb{J}_z(0)r^2+O(r^3)\Big]\xi_n^2dr\\
&=\int_0^{s_0}\Big[\frac{1}{4}\mathbb{J}^2_z(0)r+O(r^2)\Big]r^2\xi_n'^2dr+\int_0^{s_0}O(r^3)\xi_n^2dr.
\end{split}
\end{equation*}

	On the other hand,
	it follows from $B_\theta=\frac{1}{r}\int_0^r s\mathbb{J}_z(s)ds$ that
	\begin{equation*}
	\sup_{0\leq r\leq r_0}\frac{|B_\theta|}{r}\leq\frac{\|\mathbb{J}_z\|_{L^\infty}}{2},
	\end{equation*}
	which gives that  $B_\theta(0)=0$.
	Assume $(\xi_{n},\eta_n,\zeta_n)\in Y_{m,k}$ and $\mathcal{J}(\xi_n,\eta_n,\zeta_n)=1$, by \eqref{sup-xi-2} and $B_\theta(0)=0$, we choose  $s_0$ small enough, such that $Cs_0 +Cs^2_0\leq \kappa$ with $\kappa=\frac{1}{l}$, $l\in \mathbb{N}$, then  we have for $n$ large enough such that
\begin{equation}
\begin{split}
&\int_0^{s_0}O(r^2)r^2|\xi'_{n}-\xi'|^2dr+\int_0^{s_0}O(r^3)|\xi_{n}-\xi|^2dr
\\
&\quad\leq Cs_0\int_0^{s_0}r^3|\xi'_{n}-\xi'|^2dr+Cs^2_0\int_0^{s_0}r|\xi_{n}-\xi|^2dr\\
&\quad\leq Cs_0+Cs_0^2\leq \kappa,
\end{split}
\end{equation}
which together with weak lower semi-continuity,
gives that 
\begin{equation}
\begin{split}
&\int_0^{s_0}\Big[\frac{1}{4}\mathbb{J}^2_z(0)r+O(r^2)\Big]r^2\xi'^2dr+\int_0^{s_0}O(r^3)\xi^2dr\\
&\leq \liminf_{n\to\infty}\Big\{\int_0^{s_0}\Big[\frac{1}{4}\mathbb{J}^2_z(0)r+O(r^2)\Big]r^2\xi_n'^2dr+\int_0^{s_0}O(r^3)\xi_n^2dr\Big\}.
\end{split}
\end{equation}
Hence, we have
\begin{equation}
\begin{split}
&	\int_0^{s_0}\frac{m^2B_\theta^2}{r(m^2+k^2r^2)}(\xi-r\xi')^2dr +\int_0^{s_0}\Big[2p'+\frac{m^2B_{\theta}^2}{r}\Big]\xi^2dr\\
&=\int_0^{s_0}\Big[\frac{1}{4}\mathbb{J}^2_z(0)r+O(r^2)\Big]r^2\xi'^2dr+\int_0^{s_0}O(r^3)\xi^2dr\\
&\leq\liminf_{n\to\infty}\Big\{\int_0^{s_0}\Big[\frac{1}{4}\mathbb{J}^2_z(0)r+O(r^2)\Big]r^2\xi_n'^2dr+\int_0^{s_0}O(r^3)\xi_n^2dr\Big\}\\
&=\liminf_{n\to\infty}\Big\{\int_0^{s_0}\frac{m^2B_\theta^2}{r(m^2+k^2r^2)}(\xi_n-r\xi_n')^2dr +\int_0^{s_0}\Big[2p'+\frac{m^2B_{\theta}^2}{r}\Big]\xi_n^2dr\Big\}.
\end{split}
\end{equation}
On the other hand, by weak lower semi-continuity, we can show that
	\begin{equation}
	\begin{split}
	&	\int_{s_0}^{r_0}\frac{m^2B_\theta^2}{r(m^2+k^2r^2)}(\xi-r\xi')^2dr +\int_{s_0}^{r_0}\frac{m^2B_{\theta}^2}{r}\xi^2dr\\
	&\leq \liminf_{n\to\infty} \int_{s_0}^{r_0}\frac{m^2B_\theta^2}{r(m^2+k^2r^2)}(\xi_n-r\xi_n')^2dr +\liminf_{n\to\infty} \int_{s_0}^{r_0}\frac{m^2B_{\theta}^2}{r}\xi_n^2dr.
	\end{split}
	\end{equation}
	From Lemma \ref{embedding-p'-general-m},  for $s_1$ near $r=r_0$, it holds that as $n\to\infty$, 
\begin{equation*}
\begin{split}
\int_{s_0}^{s_1}p'|\xi_n-\xi|^2dr&\leq\Big( \int_{s_0}^{s_1}B^2_\theta|\xi_n-\xi|^2dr\Big)^{\frac{1}{2}}
\Big( \int_{s_0}^{s_1}\mathbb{J}^2_z|\xi_n-\xi|^2dr\Big)^{\frac{1}{2}}\\
&\leq \Big( \int_{s_0}^{s_1}B^2_\theta|\xi_n-\xi|^2dr\Big)^{\frac{1}{2}}
\Big( \int_{s_0}^{s_1} |\xi_n-\xi|^2dr\Big)^{\frac{1}{2}}
\|\mathbb{J}_z\|_{L^\infty}\\
&\leq C\Big( \int_{s_0}^{s_1}B^2_\theta|\xi_n-\xi|^2dr\Big)^{\frac{1}{2}}\quad \rightarrow 0,
\end{split}
\end{equation*}
and
\begin{equation*}
\begin{split}
\int_{s_1}^{r_0}-p'|\xi_n-\xi|^2dr\quad \rightarrow 0,
\end{split}
\end{equation*}
where we have used $\xi_n$ is uniformly bounded in $H^1(s_0,s_1)$ and $\mathbb{J}_z\in L^{\infty}([0,r_0])$. 

{\bf Case II: } When $|m|\geq 2$, 	the positive term $\int_0^{r_0}\frac{m^2B_\theta^2}{r(m^2+k^2r^2)}(\xi_n-r\xi_n')^2$ is dealt  with by weak lower semi-continuity, which implies that
\begin{equation*}
\begin{split}
\int_0^{r_0}\frac{m^2B_\theta^2}{r(m^2+k^2r^2)}(\xi-r\xi')^2\leq\liminf_{n\to\infty}\int_0^{r_0}\frac{m^2B_\theta^2}{r(m^2+k^2r^2)}(\xi_n-r\xi_n')^2.
\end{split}
\end{equation*}
For the term $\int_0^{r_0}\Big[2p'+\frac{m^2B_{\theta}^2}{r}\Big]\xi_n^2dr$,
we have for every $s_0>0$, 
\begin{equation}
\begin{split}
	\int_0^{r_0}\Big[2p'+\frac{m^2B_{\theta}^2}{r}\Big]\xi_n^2dr= \Big(\int_0^{s_0} +\int_{s_0}^{r_0}\Big)\Big[2p'+\frac{m^2B_{\theta}^2}{r}\Big]\xi_n^2dr.
\end{split}
\end{equation}
Applying \eqref{taylor-B} and \eqref{taylor-p} in Lemma \ref{expansion-zero},  we deduce
	\begin{equation}
	\begin{split}
	&\int_0^{s_0}\Big[2p'+\frac{m^2B_{\theta}^2}{r}\Big]\xi_n^2dr\\
	&=\int_0^{s_0}\Big[-\mathbb{J}^2_z(0)r-\frac{5}{3}\mathbb{J}_z'(0)\mathbb{J}_z(0)r^2+m^2\Big(\frac{1}{4}\mathbb{J}^2_z(0)r+\frac{1}{3}\mathbb{J}_z'(0)\mathbb{J}_z(0)r^2\Big)
	+O(r^3)\Big]\xi_n^2dr\\
	&=\int_0^{s_0}\Big[\Big(\frac{m^2}{4}-1\Big)\mathbb{J}^2_z(0)r
	+O(r^2)\Big]\xi_n^2dr.
	\end{split}
	\end{equation}
Assume $\mathcal{J}(\xi_n,\eta_n,\zeta_n)=1$, we choose  $s_0$ small enough, such that $Cs_0 \leq \kappa,$ with $\kappa=\frac{1}{l}$, $l\in \mathbb{N}$, then  we have for $n$ large enough such that
	\begin{equation}
	\int_0^{s_0}O(r^2)|\xi_{n}-\xi|^2dr
	\leq Cs_0\int_0^{s_0}r|\xi_{n}-\xi|^2dr\leq Cs_0\leq \kappa,
	\end{equation}
	which together with weak lower semicontinuity,
	gives that 
	\begin{equation}
	\begin{split}
	\int_0^{s_0}\Big[2p'+\frac{m^2B_{\theta}^2}{r}\Big]\xi^2dr&=\int_0^{s_0}\Big[\Big(\frac{m^2}{4}-1\Big)\mathbb{J}^2_z(0)r
	+O(r^2)\Big]\xi^2dr\\
	&\leq \liminf_{n\to\infty}\int_0^{s_0}\Big[\Big(\frac{m^2}{4}-1\Big)\mathbb{J}^2_z(0)r
	+O(r^2)\Big]\xi_n^2dr\\
	&=\liminf_{n\to\infty}\int_0^{s_0}\Big[2p'+\frac{m^2B_{\theta}^2}{r}\Big]\xi_n^2dr.
	\end{split}
	\end{equation}
	On the other hand, by weak lower semi-continuity, we can show that
	\begin{equation}
	\begin{split}
		\int_{s_0}^{r_0}\frac{m^2B_{\theta}^2}{r}\xi^2dr\leq\liminf_{n\to\infty} \int_{s_0}^{r_0}\frac{m^2B_{\theta}^2}{r}\xi_n^2dr.
	\end{split}
	\end{equation}
	From  Lemma \ref{embedding-p'-general-m},  for $s_1$ near $r=r_0$, similarly  it holds that as $n\to\infty$, 
	\begin{equation*}
	\begin{split}
	\int_{s_0}^{s_1}p'|\xi_n-\xi|^2dr
	\leq C\Big( \int_{s_0}^{s_1}B^2_\theta|\xi_n-\xi|^2dr\Big)^{\frac{1}{2}}\quad \rightarrow 0,
	\end{split}
	\end{equation*}
\begin{equation*}
\begin{split}
\int_{s_1}^{r_0}-p'|\xi_n-\xi|^2dr\quad \rightarrow 0.
\end{split}
\end{equation*}
		Therefore, we get that for any fixed $k$ and $m\neq 0$,
	\begin{equation*}
	E_{m,k}(\xi,\eta,\zeta,\widehat{ Q}_{r})\leq \liminf_{n\to\infty}E_{m,k}(\xi_n,\eta_n,\zeta_n,\widehat{ Q}_{rn})=\inf_{\mathcal{A}_2}E_{m,k}.
	\end{equation*}
	All that remains is to show that $((\xi,\eta,\zeta),\widehat{ Q}_r)\in \mathcal{A}_2. $
	
	Again by lower semi-continuity, we know that $\mathcal{J}(\xi,\eta,\zeta)\leq 1$. Suppose by way
	of contradiction that $\mathcal{J}(\xi,\eta,\zeta)<1$. By the homogeneity of $\mathcal{J}$, we may find $\alpha>1$ so that $\mathcal{J}(\alpha\xi,\alpha\eta,\alpha\zeta)=1$, i.e., we may scale up $((\xi,\eta,\zeta),\widehat{ Q}_r)$ so that $((\alpha \xi,\alpha\eta,\alpha\zeta),\alpha\widehat{ Q}_r)\in \mathcal{A}_2$. By Proposition \ref{infimum-A-out-3}, we know that $\inf E_{m,k}< 0$, and from this we deduce that $$E_{m,k}(\alpha\xi,\alpha\eta,\alpha\zeta,\alpha\widehat{ Q}_r)=\alpha^2E_{m,k}(\xi,\eta,\zeta,\widehat{ Q}_r)=\alpha^2 \inf E_{m,k}< \inf E_{m,k},$$ 
	which is a contradiction since $((\alpha \xi,\alpha\eta,\alpha\zeta),\alpha\widehat{ Q}_r )\in \mathcal{A}_2$. Hence $\mathcal{J}(\xi, \eta,\zeta)=1$, so that $((\xi,\eta,\zeta),\widehat{ Q}_r)\in \mathcal{A}_2$. 
\end{proof}
We now prove that the minimizer constructed in the previous result satisfies Euler-Lagrange equations equivalent to \eqref{spectral-formulation} and \eqref{euler-l-q} with suitable boundary conditions.
\begin{prop}\label{infimum-A-out-2}
	Let $((\xi,\eta,\zeta),\widehat{ Q}_{r})\in\mathcal{A}_2$ be the minimizers of $E_{m,k}$ constructed in Proposition \ref{infimum-A-out}. Then $(\xi,\eta,\zeta)$ are smooth in $(0,r_0)$ and satisfy 
	\begin{equation} \label{spectal formulation-kink-4} 
	\begin{split}
	&\left(
	\begin{array}{ccc}
	\frac{d}{dr}\frac{\gamma p+B_{\theta}^2}{r}\frac{d}{dr}r-\frac{m^2}{r^2}B_{\theta}^2
	-r(\frac{B_{\theta}^2}{r^2})'&-\frac{d}{dr}k(\gamma p+B_{\theta}^2)-\frac{2kB_{\theta}^2}{r}&\frac{d}{dr}\frac{m}{r}\gamma p\\
	\frac{k(\gamma p+B_{\theta}^2)}{r}\frac{d}{dr}r-\frac{2kB_{\theta}^2}{r}&
	-k^2(\gamma p+B_{\theta}^2)-\frac{m^2}{r^2}B_{\theta}^2&\frac{mk}{r}\gamma p\\
	-\frac{m\gamma p}{r^2}\frac{d}{dr}r&\frac{mk}{r}\gamma p&-\frac{m^2}{r^2}\gamma p
	\end{array}
	\right)
	\left(
	\begin{array}{lll}
	\xi   \\
	\eta \\
	\zeta\\
	\end{array}
	\right) \\
	&\quad=-\rho \lambda \left(
	\begin{array}{lll}
	\xi   \\
	\eta \\
	\zeta\\
	\end{array}
	\right),
	\end{split}
	\end{equation}
	the solution $\widehat{ Q}_{r}$ is smooth on $(r_0,r_w)$ and satisfies
	\begin{equation}\label{ode-curl-div}
	\bigg[\frac{r}{m^2+k^2r^2}(r\widehat{Q}_r)'\bigg]'-\widehat{Q}_r=0, 
	\end{equation}  
	along with the interface boundary condition	
	\begin{equation}\label{inter-four-2}
	B_{\theta}^2\xi-B_{\theta}^2\xi' r+kB_\theta^2\eta r=\widehat{ B}_{\theta}\widehat{Q}_{\theta}r, \quad \mbox{at} \quad r=r_0,
	\end{equation}
	where the other two components of $\widehat{ Q}$ are denoted by $\widehat{ Q}_{\theta}=-\frac{m}{m^2+k^2r^2}(r\widehat{ Q}_r)'$, $\widehat{ Q}_{z}=-\frac{kr}{m^2+k^2r^2}(r\widehat{ Q}_r)'$.
\end{prop}
\begin{proof}
	Fix $((\xi_0,\eta_0,\zeta_0), q_r)\in Y_{m,k} \times H^1(r_0,r_w)$ and assume they satisfy $ m\widehat{  B}_\theta \xi_0=rq_r$ on the bounary $r=r_0$ and $q_r=0$ on the boundary $r=r_w$. Define 
	$$j(t,\tau)=\mathcal{J}(\xi+t\xi_0+\tau\xi,\eta+t\eta_0+\tau\eta,\zeta+t\zeta_0+\tau\zeta).$$
 Note that $j(0,0)=1$. Moreover, $j$ is smooth, 
	\begin{equation*}
	\begin{split}
	&\frac{\partial j}{\partial t}(0,0)=2\pi^2\int_{0}^{r_0}2\rho(\xi_0\xi +\eta_0\eta+\zeta_0\zeta)rdr, \\
	&\frac{\partial j}{\partial \tau}(0,0)=2\pi^2\int_{0}^{r_0}2\rho(\xi^2 +\eta^2+\zeta^2)rdr=2.
	\end{split}
	\end{equation*}
	So, by the inverse function theorem, we can solve for $\tau=\tau(t)$ in a neighborhood of $0$ as a $C^1$ function of $t$ so that $\tau(0)=0$ and $j(t,\tau(t))=1$. We may differentiate the last equation to find 
	\begin{equation*}
	\frac{\partial j}{\partial t}(0,0)+\frac{\partial j}{\partial \tau}(0,0)\tau'(0)=0,
	\end{equation*}
	which gives that 
	\begin{equation*}
	\tau'(0)=-\frac{1}{2}\frac{\partial j}{\partial t}(0,0)=-2\pi^2\int_{0}^{r_0}\rho(\xi_0\xi +\eta_0\eta+\zeta_0\zeta)rdr.
	\end{equation*}
	Since $((\xi, \eta,\zeta),\widehat{ Q}_r)$ are minimizers over the set $\mathcal{A}_2$, we may make variations with respect to $((\xi_0,\eta_0,\zeta_0),q_r)$ to find that 
	\begin{equation*}
	0=\frac{d}{dt}\bigg|_{t=0}E(\xi+t\xi_0+\tau(t)\xi, \eta+t\eta_0+\tau(t)\eta, \zeta+t\zeta_0+\tau(t)\zeta,\widehat{ Q}_r+tq_r+\tau \widehat{ Q}_r),
	\end{equation*}
	which implies that
	\begin{equation*}
	\begin{split}
	0&=4\pi^2 \int _0^{r_0}\bigg\{\frac{m^2B_{\theta}^2}{r^2(m^2+k^2r^2)}(r\xi)'(r\xi_0)'+\beta_0(r\xi) \cdot (r\xi_0)+\frac{(m^2+k^2r^2)B_{\theta}^2}{r^2}\eta\eta_0\\
	&\quad-\frac{kB^2_{\theta}(r\xi_0)'\eta}{r}+\frac{2kB_{\theta}^2}{r}\xi_0\eta+\frac{k^2B_\theta^2}{m^2+k^2r^2}(r\xi)'(r\xi_0)'-\frac{2k^2B^2_{\theta}}{m^2+k^2r^2}\xi(r\xi_0)'\\
	&\quad-\frac{kB_\theta^2}{r}(r\xi)'\eta_0+\frac{2kB^2_\theta}{r}\xi\eta_0-\frac{2k^2B^2_{\theta}}{m^2+k^2r^2}\xi_0(r\xi)'+\frac{4k^2B_{\theta}^2}{m^2+k^2r^2}\xi\xi_0
	\\
	&\quad+\gamma p \bigg[\frac{1}{r^2}(r\xi)'(r\xi_0)'+\frac{-k(r\xi)'\eta_0+\frac{m}{r}(r\xi)'\zeta_0}{r}+\frac{-k(r\xi_0)'\eta+\frac{m}{r}(r\xi_0)'\zeta}{r}\\
	&\quad+\Big(-k\eta+\frac{m}{r}\zeta\Big)\Big(-k\eta_0+\frac{m}{r}\zeta_0\Big)\bigg]\bigg\}rdr-4\pi^2\Big[\frac{2m^2B_{\theta}^2}{m^2+k^2r^2}\xi\xi_0\Big]_{r=r_0}\\
	&\quad+4\pi^2\tau'(0)\int_0^{r_0}\bigg\{\frac{m^2B_{\theta}^2}{r^2(m^2+k^2r^2)}|(r\xi)'|^2+\beta_0(r\xi)^2+(m^2+k^2r^2)\Big[\frac{B_\theta}{r}\eta\\
	&\quad+\frac{-kB_\theta(r\xi)'+2kB_{\theta}\xi}{m^2+k^2r^2}\Big]^2+\gamma p\Big[\frac{1}{r}(r\xi)'-k\eta+\frac{m}{r}\zeta\Big]^2\bigg\}rdr\\
	&\quad-4\pi^2\tau'(0)\Big[\frac{2m^2B_{\theta}^2}{m^2+k^2r^2}\xi^2\Big]_{r=r_0}+4\pi^2\int_{r_0}^{r_w} \Big[\widehat{Q}_r q_r+\frac{1}{m^2+k^2r^2}(r\widehat{ Q}_r)'(rq_r)'\Big]rdr\\
	&\quad+4\pi^2\tau'(0)\int_{r_0}^{r_w}\Big[|\widehat{Q}_r|^2+\frac{1}{m^2+k^2r^2}|(r\widehat{Q}_r)'|^2
	\Big]rdr,
	\end{split}
	\end{equation*}
	that is,
	\begin{equation*}
	\begin{split}
0&=4\pi^2 \int _0^{r_0}\bigg\{\frac{\gamma p+B_{\theta}^2}{r}(r\xi)'(r\xi_0)'+r\beta_0(r\xi) \cdot (r\xi_0)-kB_{\theta}^2(r\xi_0)'\eta+\frac{2kB_{\theta}^2}{r}(r\xi_0)\cdot\eta\\
&\quad-\frac{2k^2B^2_{\theta}}{m^2+k^2r^2}r\xi(r\xi_0)'-\frac{2k^2B^2_{\theta}}{m^2+k^2r^2}r\xi_0(r\xi)'+\frac{4k^2B_{\theta}^2}{m^2+k^2r^2}\xi(r\xi_0)
-k\gamma p(r\xi_0)'\eta\\
&\quad+\frac{m}{r}\gamma p (r\xi_0)'\zeta+\frac{(m^2+k^2r^2)B_{\theta}^2}{r^2}\eta\cdot(r\eta_0)-\frac{kB_{\theta}^2}{r}(r\xi)'(r\eta_0)+\frac{2kB_{\theta}^2}{r}\xi\cdot (r\eta_0)\\
&\quad-\frac{k\gamma p}{r}(r\xi)'\cdot(r\eta_0)+k^2\gamma p \eta\cdot (r\eta_0)-\frac{mk}{r}\gamma p \zeta\cdot (r\eta_0)+\frac{m\gamma p}{r^2}(r\xi)'(r\zeta_0)\\
&\quad-\frac{mk}{r}\gamma p \eta(r\zeta_0)+\frac{m^2}{r^2}\gamma p \zeta\cdot(r\zeta_0)\bigg\}dr-4\pi^2\Big[\frac{2m^2B_{\theta}^2}{m^2+k^2r^2}\xi\xi_0\Big]_{r=r_0}\\
&\quad+ 4\pi^2\int_{r_0}^{r_w} \Big[\widehat{Q}_r q_r+\frac{1}{m^2+k^2r^2}(r\widehat{ Q}_r)'(rq_r)'\Big]rdr+2\tau'(0)\lambda.
	\end{split}
	\end{equation*}
	Therefore, we can prove that 
	\begin{equation}\label{kink-va}
	\begin{split}
	-\tau'(0)&\lambda=2\pi^2\int_{0}^{r_0}\rho\lambda(\xi_0\xi +\eta_0\eta+\zeta\zeta_0)rdr
	\\&=2\pi^2 \int _0^{r_0}\bigg\{\frac{\gamma p+B_{\theta}^2}{r}(r\xi)'(r\xi_0)'+r\beta_0(r\xi) \cdot (r\xi_0)-kB_{\theta}^2(r\xi_0)'\eta\\
	&\quad+\frac{2kB_{\theta}^2}{r}(r\xi_0)\cdot\eta-\frac{2k^2B^2_{\theta}}{m^2+k^2r^2}r\xi(r\xi_0)'-\frac{2k^2B^2_{\theta}}{m^2+k^2r^2}r\xi_0(r\xi)'\\
	&\quad+\frac{4k^2B_{\theta}^2}{m^2+k^2r^2}\xi(r\xi_0)
	-k\gamma p(r\xi_0)'\eta+\frac{m}{r}\gamma p (r\xi_0)'\zeta+(\frac{m^2}{r^2}+k^2)B_{\theta}^2\eta\cdot(r\eta_0)\\
	&\quad-\frac{kB_{\theta}^2}{r}(r\xi)'(r\eta_0)+\frac{2kB_{\theta}^2}{r}\xi\cdot (r\eta_0)-\frac{k\gamma p}{r}(r\xi)'\cdot(r\eta_0)+k^2\gamma p \eta\cdot (r\eta_0)\\
	&\quad-\frac{mk}{r}\gamma p \zeta\cdot (r\eta_0)+\frac{m\gamma p}{r^2}(r\xi)'(r\zeta_0)-\frac{mk}{r}\gamma p \eta(r\zeta_0)+\frac{m^2}{r^2}\gamma p \zeta\cdot(r\zeta_0)\bigg\}dr\\
	&\quad-2\pi^2\Big[\frac{2m^2B_{\theta}^2}{m^2+k^2r^2}\xi\xi_0\Big]_{r=r_0}+ 2\pi^2\int_{r_0}^{r_w} \Big[\widehat{Q}_r q_r+\frac{1}{m^2+k^2r^2}(r\widehat{ Q}_r)'(rq_r)'\Big]rdr.
	\end{split}
	\end{equation}
Noting that 	$$\beta_0=\frac{1}{r}\Big[\frac{m^2B_{\theta}^2}{r^3}+\frac{2m^2B_\theta (\frac{B_\theta}{r})'}{r(m^2+k^2r^2)}-\frac{4k^2m^2B^2_{\theta}}{r(m^2+k^2r^2)^2}+\frac{2k^2p'}{m^2+k^2r^2}\Big],$$
$$ p'=-\frac{1}{r}(rB_{\theta})'B_{\theta},$$
	we can  get that
	\begin{equation*}
	\begin{split}
	& \int _0^{r_0}\bigg\{r\beta_0(r\xi) \cdot (r\xi_0)-\frac{2k^2B^2_{\theta}}{m^2+k^2r^2}r\xi(r\xi_0)'-\frac{2k^2B^2_{\theta}}{m^2+k^2r^2}r\xi_0(r\xi)'+\frac{4k^2B_{\theta}^2}{m^2+k^2r^2}\xi(r\xi_0)\bigg\}dr\\&
	\quad= \int _0^{r_0}\bigg\{\frac{m^2}{r^2} B_{\theta}^2\xi \cdot (r\xi_0)+\Big(\frac{B_{\theta}^2}{r^2}\Big)'(r\xi) \cdot (r\xi_0)\bigg\}dr-\Big[\frac{2k^2r^2B_{\theta}^2}{m^2+k^2r^2}\xi\xi_0\Big]_{r=r_0}.
	\end{split}
	\end{equation*}
	
	Since $\xi_0$, $\eta_0$, $\zeta_0$ and $q_r$ are independent, and using $$-\Big[\frac{2k^2r^2B_{\theta}^2}{m^2+k^2r^2}\xi\xi_0\Big]_{r=r_0}-\Big[\frac{2m^2B_{\theta}^2}{m^2+k^2r^2}\xi\xi_0\Big]_{r=r_0}=-2 \Big[B_{\theta}^2\xi\xi_0\Big]_{r=r_0}, $$ inserting the above identity into \eqref{kink-va}, one has the triplet of equations
	\begin{equation}\label{weak-four}
	\begin{split}
	&-\int_{0}^{r_0}(\gamma p+B_{\theta}^2)k\eta (r\xi_0)'dr+\int_{0}^{r_0}2B_{\theta}^2k\eta \xi_0dr+\int_{0}^{r_0}\Big(\frac{B_{\theta}^2}{r^2}\Big)'(r\xi)\cdot (r\xi_0)dr \\
	&\quad +\int_{0}^{r_0}\frac{\gamma p+B_{\theta}^2}{r}(r\xi)' (r\xi_0)'dr+\int_{0}^{r_0}\frac{m^2}{r^2}B_{\theta}^2\xi \cdot (r\xi_0)dr+\int_{0}^{r_0}\frac{m}{r}\gamma p (r\xi_0)'\zeta dr\\
	&\quad -2 \big[B_{\theta}^2\xi\xi_0\big]_{r=r_0}
	+\int_{r_0}^{r_w} \Big[\widehat{Q}_r q_r+\frac{1}{m^2+k^2r^2}(r\widehat{ Q}_r)'(rq_r)'\Big]rdr=\int_{0}^{r_0}\rho\lambda\xi_0\xi rdr,
	\end{split}
	\end{equation}	
	\begin{equation*}
	\begin{split}
	&\int_{0}^{r_0}\Big(\frac{m^2}{r^2}+k^2\Big)B_{\theta}^2\eta\cdot(r\eta_0)dr-\int_{0}^{r_0}\frac{k(\gamma p+B_{\theta}^2)}{r}(r\xi)'(r\eta_0)dr+\int_{0}^{r_0}\frac{2kB_{\theta}^2}{r}\xi\cdot (r\eta_0)dr\\
	&\quad+\int_{0}^{r_0}k^2\gamma p \eta\cdot (r\eta_0)dr-\int_{0}^{r_0}\frac{mk}{r}\gamma p \zeta\cdot (r\eta_0)dr=\int_{0}^{r_0}\rho\lambda\eta_0\eta rdr,
	\end{split}
	\end{equation*}
	\begin{equation*}
	\int_{0}^{r_0}\frac{m\gamma p}{r^2}(r\xi)'(r\zeta_0)dr-\int_{0}^{r_0}\frac{mk}{r}\gamma p \eta(r\zeta_0)dr+\int_{0}^{r_0}\frac{m^2}{r^2}\gamma p \zeta\cdot(r\zeta_0)dr=\int_{0}^{r_0}\rho\lambda\zeta_0\zeta rdr.
	\end{equation*}
	By making variation with $\xi_0$ compactly supported in $(0,r_0)$, and make variations $q_r$ compactly supported in $(r_0, r_w)$, one gets that $(\xi, \eta, \zeta)$ satisfy \eqref{spectal formulation-kink-4} in a weak sense in $(0,r_0)$ and $\widehat{ Q}_{r}$ satisfies \eqref{ode-curl-div} in a weak sense in $(r_0,r_w)$. 
	
	Now we show that the interface boundary condition
	 \eqref{inter-four-2} is satisfied.
	From the first equation \eqref{spectal formulation-kink-4}, we know that 
	\begin{equation*}
	\begin{split}
	&\frac{d}{dr}\Big[\gamma p\Big(\frac{1}{r}(r\xi)'-k\eta+\frac{m}{r}\zeta\Big)\Big]+\frac{d}{dr}\Big[B^2_\theta\Big(\frac{1}{r}(r\xi)'-k\eta\Big)\Big]-\frac{2B_\theta B'_{\theta}\xi}{r}\\
	&\quad
	-\frac{B^2_\theta \xi'}{r}+\frac{B^2_\theta \xi}{r^2}+\frac{B^2_{\theta}}{r}\Big(\frac{1}{r}(r\xi)'-2k\eta-\frac{m^2}{r}\xi\Big)=-\rho \lambda\xi,
	\end{split}
	\end{equation*}
	which together with $(\xi,\eta,\zeta)\in Y_{m,k}$ (see \eqref{definition-ymk}) and $\xi \in V$ (see \eqref{defi-v}),  gives that 
	\begin{equation*}
	\frac{d}{dr}\Big[\gamma p\Big(\frac{1}{r}(r\xi)'-k\eta+\frac{m}{r}\zeta\Big)+B^2_\theta\Big(\frac{1}{r}(r\xi)'-k\eta\Big)\Big]\in L^2(\frac{r_0}{2}, r_0).
	\end{equation*}After a similar argument, we have 
	\begin{equation*}
	\frac{d}{dr}\Big[\frac{r}{m^2+k^2r^2}(r\widehat{Q}_r)'\Big]\in L^2(r_0,r_w), 
	\end{equation*}  
	so $\gamma p\Big(\frac{1}{r}(r\xi)'-k\eta+\frac{m}{r}\zeta\Big)+B^2_\theta\Big(\frac{1}{r}(r\xi)'-k\eta\Big)$ and $\frac{r}{m^2+k^2r^2}(r\widehat{Q}_r)'$ are well-defined at the endpoint $r=r_0$.
	Make variations with respect to $\xi_0\in C_c^{\infty}(0,r_0]$, $q_r\in C_c^{\infty}[r_0,r_w)$. Integrating the terms in \eqref{weak-four} with derivatives of $\xi_0$ and $q_r$ by parts, using $\xi$ solves the first equation of \eqref{spectal formulation-kink-4} on $(0,r_0)$ and $\widehat{ Q}_r$ solves \eqref{ode-curl-div} on $(r_0,r_w)$, we get that 
	\begin{equation*}
	\Big [(\gamma p+B_{\theta}^2)(r\xi)'\xi_0-k(\gamma p+B_\theta^2)\eta \xi_0r-2 B_{\theta}^2\xi\xi_0+m\gamma p \zeta\xi_0 \Big]_{r=r_0}-\Big[\frac{r}{m^2+k^2r^2}(r\widehat{ Q}_r)'rq_r\Big]_{r=r_0}=0.
	\end{equation*}
	Since $\xi_0$ and $q_r$ may be chosen arbitrarily, and $q_r$ satisfies $m\widehat{  B}_\theta \xi_0=rq_r$ on the bounary $r=r_0$,  using $p=0$ on the boundary $r=r_0$ and $\widehat{ Q}_{\theta}=-\frac{m}{m^2+k^2r^2}(r\widehat{ Q}_r)'$, we deduce the interface boundary condition
	$$\Big [B_{\theta}^2\xi-B_{\theta}^2\xi' r+kB_\theta^2\eta r-\widehat{  B}_\theta\widehat{ Q}_{\theta}r\Big]\Big|_{r=r_0}=0.$$	
\end{proof}
\section{Analysis about the growing mode as a function of $m$ and $k$}\label{section-bound-growing-mode}
In this section, we first prove the growing mode is bounded for any $(m,k) \in \mathbb{Z}\times \mathbb{Z}$, if the  pressure satisfies $|p'|\leq C\rho $ for $r$ near $r_0$, and then show that the growing mode has no lower bound under suitable condition of the pressure $p$. First, we prove the growing mode is bounded for any $ m$ and $k$, under the condition $|p'|\leq C\rho $ for $r$ near $r_0$.
\begin{prop}\label{bound-growing-mode}
	 If $|p'|\leq C\rho $ for $r$ near $r_0$, then the growing mode is bounded 
	 for any $m$ and $k$. 
\end{prop}
\begin{proof}
	We can directly get from the energy $E_{0,k}$ and $E_{m,k}$ in \eqref{energy-m=0} and \eqref{sausage-in-v-t-out} 
	that for any $m$ and $k$
	\begin{equation*}
	\begin{split}
	E_{m,k}\geq 2\pi^2 \int_0^{r_0}2p'\xi^2dr=2\pi^2 \int_0^{s_1}2p'\xi^2dr+2\pi^2 \int_{s_1}^{r_0}2p'\xi^2dr.
	\end{split}
	\end{equation*}
	From $p'=-B_\theta B'_\theta-\frac{B^2_\theta}{r}=-B_\theta\mathbb{J}_z$, it follows that for $s_1$ near $r_0$
	\begin{equation*}
	\Big|\int_0^{s_1}p'\xi^2dr\Big|=\Big| \int_0^{s_1}\frac{p'}{ r}\xi^2rdr\Big|=\Big|-\int_0^{s_1}\frac{B_\theta \mathbb{J}_z}{r} \xi^2rdr\Big|\leq \sup_{0\leq r\leq r_0}\Big|\frac{B_\theta}{r}\Big|\sup_{0\leq r\leq r_0}\Big|\mathbb{J}_z\Big|\mathcal{J}
	\end{equation*}
	and 
	\begin{equation*}
	\Big|\int_{s_1}^{r_0}p'\xi^2dr\Big|=\Big|\int_{s_1}^{r_0}\frac{p'}{\rho }\rho \xi^2dr\Big|\leq \sup_{0\leq r\leq r_0}\Big|\frac{p'}{\rho}\Big|\mathcal{J}.
	\end{equation*}
	On the other hand,
	it follows from $B_\theta=\frac{1}{r}\int_0^r s\mathbb{J}_z(s)ds$ that
	\begin{equation*}
	\sup_{0\leq r\leq r_0}\frac{|B_\theta|}{r}\leq\frac{\|\mathbb{J}_z\|_{L^\infty}}{2}. 
	\end{equation*}
	Therefore, we get from $|p'|\leq C\rho $ for $r$ near $r_0$ that
	$$\Big|\int_0^{r_0}2p'\xi^2dr\Big|\leq C(\mathcal{J}),$$ 
	which ensures that the energy $E_{0,k}$ and $E_{m,k}$ have a uniform lower bound.
	Therefore, the growing mode is bounded.
\end{proof}
Now we introduce the following examples which insures the condition in Proposition \ref{bound-growing-mode}.
\begin{example}\label{example}
	(I) Assume $p(r)=C (r_0-r)^\beta$  for $r$ near $r_0$ and $\beta\geq 1$.  If $\gamma \geq \frac{\beta}{\beta-1} $,  then $|p'|\leq C\rho$ for $r$ near $r_0$ and  the growing mode is bounded for any $m$ and $k$. 
	
	(II) Assume $p=C exp\{-(r_0-r)^{-\beta}\}$ for $r$ near $r_0$ and $\beta>0$.  If $\gamma >1$, then $|p'|\leq C\rho$ for $r$ near $r_0$ and  the growing mode is bounded for any $m$ and $k$. 
\end{example}
\begin{proof}
(I) Since  $p=C (r_0-r)^{\beta}$ for $r$ near $r_0$, we deduce that $p'(r)\sim -(r_0-r)^{\beta-1}$ for $r$ near $r_0$.
By $p=A\rho ^\gamma$,  we have $\rho \sim (r_0-r)^{\frac{\beta}{\gamma}}$ for $r$ near $r_0$. Hence, if $\gamma \geq \frac{\beta}{\beta-1} $, then $|p'|\leq C\rho$ for $r$ near $r_0$. By Proposition \ref{bound-growing-mode}, we get that the growing mode is bounded for any $m$ and $k$.

(II) Since  $p=C exp\{-(r_0-r)^{-\beta}\}$ for $r$ near $r_0$, we get that for $r$ near $r_0$, $$p'(r)\sim -\frac{\beta}{(r_0-r)^{\beta+1}}exp\{-(r_0-r)^{-\beta}\}.$$ 
By $p=A\rho ^\gamma$,  we have $\rho \sim exp\{-\frac{1}{\gamma}(r_0-r)^{-\beta}\}$. Hence, if $\gamma> 1$, then $|p'|\leq C\rho$ for $r$ near $r_0$. By Proposition \ref{bound-growing-mode}, similarly we get that the growing mode is bounded for any $m$ and $k$.
\end{proof}
Finally, we prove that the growing mode has no lower bound under suitable condition of the pressure.
\begin{prop}
Assume $p(r)=C (r_0-r)^\beta$ for $r$ near $r_0$ and $ \beta\geq 1$. If $\gamma< \frac{\beta}{\beta-1} $, then $\frac{p'}{\rho }\rightarrow -\infty$ as $r\rightarrow r_0$ and  the growing mode has no lower bound.
\end{prop}
\begin{proof}
 Since $p(r)=C |r-r_0|^{\beta}$ for $r$ near $r_0$, from (I) in Example \ref{example}, we know that $p'\sim -|r-r_0|^{\beta-1}$ and $\rho \sim |r-r_0|^{\frac{\beta}{\gamma}}$ for $r$ near $r_0$. Hence, if $\gamma< \frac{\beta}{\beta-1} $, then $\frac{p'}{\rho}\rightarrow -\infty$ as $r \rightarrow r_0$.
 
Now we prove the ill-posedness by the above facts.	
We can choose $w$ as any smooth function with compact support near $0$ and define a sequence of test functions  $\xi_k=w(k^\alpha[r-r_0])$, $\eta_k=\frac{1}{kr}\Big((r\xi_k)'-\frac{2B_{\theta}^2}{\gamma p +B_{\theta}^2}\xi_k\Big)$, such that 
\begin{equation*}
\begin{split}
2\pi^2\int_0^{r_0}k^2(\gamma p+B_{\theta}^2)\Big[\eta_k-\frac{1}{kr}\Big((r\xi_k)'-\frac{2B_{\theta}^2}{\gamma p +B_{\theta}^2}\xi_k\Big)\Big]^2rdr=0.
\end{split}
\end{equation*}
It follows that
$$\partial_r\xi_k=k^\alpha w'(k^\alpha[r-r_0])$$ and $$\eta_k\sim k^{\alpha-1} w'(k^\alpha[r-r_0])+ k^{-1}w(k^\alpha[r-r_0]).$$
Therefore, we get for $0<\alpha<1$ that
\begin{equation}\begin{split}
\mathcal{J}_{0,k}&=	\int_0^{r_0} \rho (\xi_k^2+\eta_k^2)rdr\\&\sim \int_0^{r_0} |r-r_0|^{\frac{\beta}{\gamma}}(\xi_k^2+\eta_k^2)rdr\\
&\sim\int_0^{r_0} |r-r_0|^{\frac{\beta}{\gamma}}|w(k^\alpha[r-r_0])|^2dr+\int_0^{r_0} |r-r_0|^{\frac{\beta}{\gamma}}k^{2\alpha-2}|w'(k^\alpha[r-r_0])|^2dr\\
&\quad+\int_0^{r_0}|r-r_0|^{\frac{\beta}{\gamma}}k^{-2}|w'(k^\alpha[r-r_0])|^2dr \qquad(\mbox{let}\,\, z=k^\alpha[r-r_0])\\
&\sim k^{-\alpha-\frac{\alpha \beta}{\gamma}}\int_0^{k^\alpha r_0} z^{\frac{\beta}{\gamma}}|w(z)|^2dz+ k^{\alpha-2-\frac{\alpha \beta}{\gamma}}\int_0^{k^\alpha r_0} z^{\frac{\beta}{\gamma}}|w'(z)|^2dz\\
&\quad+ k^{-2-\alpha-\frac{\alpha \beta}{\gamma}}\int_0^{k^\alpha r_0} z^{\frac{\beta}{\gamma}}|w'(z)|^2dz\\
&\sim k^{-\alpha-\frac{\alpha \beta}{\gamma}}.
\end{split}
\end{equation}
Since  $p(r)\sim |r-r_0|^{\beta}$ and $p'\sim -|r-r_0|^{\beta-1}$ for $r$ near $r_0$,  we obtain
\begin{equation}
\begin{split}
E_{0,k}&=E_{0,k}(\xi_k,\eta_k)=2\pi^2\int_0^{r_0}\bigg\{\Big[\frac{2p'}{r}+\frac{4\gamma p B_{\theta}^2}{r^2(\gamma p+B_{\theta}^2)}\Big]\xi_k^2\\
&\quad+k^2(\gamma p+B_{\theta}^2)\Big[\eta_k-\frac{1}{kr}\Big((r\xi_k)'-\frac{2B_{\theta}^2}{\gamma p +B_{\theta}^2}\xi_k\Big)\Big]^2\bigg\}rdr\\
&=2\pi^2\int_0^{r_0}\Big[\frac{2p'}{r}+\frac{4\gamma p B_{\theta}^2}{r^2(\gamma p+B_{\theta}^2)}\Big]\xi_k^2rdr\\
&\sim -\int_0^{r_0}|r-r_0|^{\beta-1} w^2(k^\alpha[r-r_0])dr\\
&\quad+
\int_0^{r_0}|r-r_0|^{\beta}w^2(k^\alpha[r-r_0])dr\quad(\mbox{let}\,\, z=k^\alpha[r-r_0])\\
&\sim -k^{-\alpha-\alpha (\beta-1)}\int_0^{k^{\alpha}r_0}z^{\beta-1}w^2(z)dz\\
&\quad+k^{-\alpha-\alpha \beta}\int_0^{k^{\alpha}r_0}z^{\beta}w^2(z)dz\sim -k^{-\alpha-\alpha (\beta-1)}.
\end{split}
\end{equation}
Choosing $0<\alpha<1$,  if $\gamma< \frac{\beta}{\beta-1} $, then we get as $k\rightarrow \infty $ that
$$\lambda_{k}=\min\frac{E_{0,k}}{\mathcal{J}_{0,k}}\sim \frac{-k^{-\alpha-\alpha(\beta-1) }}{k^{-\alpha-\frac{\alpha \beta}{\gamma}}}=-k^{\frac{\alpha \beta}{\gamma}-\alpha (\beta-1)}
\rightarrow-\infty.$$
\end{proof}

 \section{Appendix: Perturbed MHD system in Lagrangian coordinates}\label{appendix}

 \subsection{Harmonic extension of the free surface}
 
 According to Subsection \ref{susect-free-surface-1}, we  know that the equation of the free surface $S_{t, pv} $ may be read as follows
 \begin{equation*}
 h(t, \mathcal{X})=\mathcal{X}+g(t, \mathcal{X}),\quad (\mathcal{X} \in \Sigma_{0, pv}=\{x=r_0 \}),
 \end{equation*}
 that is,
 \begin{equation*}
 \begin{cases}
 &h^r(t, r_0, \theta, z)=r_0+g^{r}(t, r_0,\theta, z),\\
 &h^{\theta}(t, r_0, \theta, z)=\theta+g^{\theta}(t, r_0, \theta, z),\\
 &h^z(t, r_0, \theta, z)=z+g^z(t, r_0, \theta, z).
 \end{cases}
 \end{equation*}
 
 We consider the fixed equilibrium vacuum domain
 \begin{equation}\label{vacuum-domain-1}
 \Omega_0^v\eqdefa \{(r, \theta, z) \in C(0; r_0, r_w) \times [0, 2\pi] \times 2\pi \mathbb{T}| r_0< r <r_w, \, \theta \in [0,2\pi], \, z\in 2\pi  \mathbb{T}\}
 \end{equation}
 for which we will write the coordinates as $\widehat{\mathcal{X}} \in \Omega_0^v$. We will think of $\Sigma_{0, pv}\eqdefa \{(r, \theta, z)|r=r_0, \, \theta \in [0,2\pi], \, z\in  2\pi \mathbb{T}\}$ as the plasma-vacuum
 interface of $\Omega_0^v$, and we will write $\Sigma_w \eqdefa \{(r, \theta, z)|r=r_w, \, \theta \in [0,2\pi], \, z\in  2\pi\mathbb{T}\}$ for the outer perfectly
 conducting wall.

 We continue to view $g(t, r_0, \theta, z)=g^{r}(t, r_0, \theta, z)\, e_r+g^\theta(t, r_0, \theta, z)\, e_\theta+ g^z(t, r_0, \theta, z)\,e_z$ as a vector field on $\mathbb{R}^+ \times \Sigma_{0, pv}$. We then define a vector field in cylindrical coordinates ${\Psi}(t, r, \theta, z)={\Psi}^r(t, r, \theta, z)\, e_r+{\Psi}^\theta(t, r, \theta, z)\, e_\theta+{\Psi}^z(t, r, \theta, z)\, e_z$ as the displacement in vacuum,
 \begin{equation}\label{harmonic-extension-0}
 \begin{split}
 &{\Psi}(t, r, \theta, z)\eqdefa \mathcal{H}_v g=\,\,\mbox{generalized harmonic extension of} \,\, g \,\, \mbox{into} \,\, \Omega_0^v,
 \end{split}
 \end{equation}
 where $\mathcal{H}_v g$ solves the following Laplacian equations:
 \begin{equation}\label{harmonic-equation-1}
 \begin{cases}
 &\Delta \,\Psi=0, \quad \mbox{in} \quad\Omega_0^v,\\
 &\Psi|_{\mathbb{R}^+ \times \Sigma_{0, pv}}=g,\quad \Psi|_{\mathbb{R}^+ \times \Sigma_{w}}=0,
 \end{cases}
 \end{equation}
 that is,
 \begin{equation*}
 \begin{cases}
 &(\widetilde{\Delta}-\frac{1}{r^2}) \,\Psi^r-\frac{2}{r^2}\partial_{\theta}\Psi^{\theta}=0, \quad \mbox{in} \quad \Omega_0^v,\\
 &\Psi^r|_{\mathbb{R}^+ \times \Sigma_{0, pv}}=g^{r},\quad \Psi^r|_{\mathbb{R}^+ \times \Sigma_{w}}=0,
 \end{cases}
 \end{equation*}
 \begin{equation*}
 \begin{cases}
 &(\widetilde{\Delta}-\frac{1}{r^2}) \,\Psi^\theta+\frac{2}{r^2}\partial_{\theta}\Psi^{r}=0, \quad \mbox{in} \quad \Omega_0^v,\\
 &\Psi^\theta|_{\mathbb{R}^+ \times \Sigma_{0, pv}}=g^\theta,\quad \Psi^\theta|_{\mathbb{R}^+ \times \Sigma_{w}}=0,
 \end{cases}
 \end{equation*}
 \begin{equation*}
 \begin{cases}
 &\widetilde{\Delta} \,\Psi^z=0, \quad \mbox{in} \quad \Omega_0^v,\\
 &\Psi^z|_{\mathbb{R}^+ \times \Sigma_{0, pv}}=g^z,\quad \Psi^z|_{\mathbb{R}^+ \times \Sigma_{w}}=0,
 \end{cases}
 \end{equation*}
 with $\widetilde{\Delta}=\partial_r^2+\frac{\partial_r}{r}+\partial_z^2+\frac{1}{r^2}\partial_{\theta}^2$.

 The generalized harmonic extension $\Phi:=Id+\Psi$ in vacuum of the flow map $h$ allows us to flatten the
 coordinate domain via the mapping
 \begin{equation*}\label{harmonic-extension-2}
 \Omega_0^v \ni (r_s, \theta_s, z_{s})\mapsto  \Phi(t, r_s, \theta_s, z_s) =(x, y, z) \in \Omega^v(t).
 \end{equation*}
 \begin{rmk}
 	Note that
 	\begin{equation*}\label{harmonic-extension-3}
 	\Phi(t, \Sigma_{0, pv})=\Sigma_{t, pv},\quad \Phi(t, \cdot)|_{\Sigma_w}=Id|_{\Sigma_w},
 	\end{equation*}
 	that is, $\Phi$ maps $\Sigma_{0, pv}$ to the free surface and keeps the outer perfectly conducting wall fixed.
 \end{rmk}

 \subsection{Vacuum equations in Lagrangian coordinates}
 
 According to the extended co-moving frame $\widehat{\mathcal{X}}(t)\eqdefa\Phi(t, r, \theta, z)$, we may introduce the``virtual velocity"  field $\widehat{u}(t, \Phi) =\frac{d}{dt}\Phi(t, r, \theta, z)$ reduced by the virtual particle in vacuum (which satisfies $|\frac{\widehat{u}}{c}| =o(1)$ when we consider the non-relativistic MHD, here $c$ is the light speed).
 
 We define Lagrangian quantities in vacuum as follows:
 \begin{equation*}
 \begin{split}
 &\widehat{b}(t, \widehat{\mathcal{X}})=\widehat{B}(t, \Phi(t, \widehat{\mathcal{X}})), \quad \widehat{v}(t, \widehat{\mathcal{X}})=\widehat{u}(t, \Phi(t, \widehat{\mathcal{X}})),  \quad \widehat{\mathcal{A}}\eqdefa (D\Phi)^{-1}, \quad \widehat{J} \eqdefa \mbox{det}(D\Phi).
 \end{split}
 \end{equation*}
 Similar to \eqref{flow-map-identity-1} and \eqref{identity-Lagrangian-1}, thanks to definitions of the mapping $\eta$ and the displacement $\Psi$ in vacuum, we may also get the following identities:
 \begin{equation}\label{vacuum-map-identity-1}
 \begin{split}
 &\widehat{\mathcal{A}}_{i}^k \partial_{k} \Phi^j=\widehat{\mathcal{A}}_{k}^j \partial_{i} \Phi^k=\delta_i^j, \quad \partial_k(\widehat{J}\widehat{\mathcal{A}}_{i}^k)=0,\quad\partial_{i} \Phi^j=\delta_i^j+\partial_{i} \Psi^j, \quad \widehat{\mathcal{A}}_{i}^j=\delta_i^j-\widehat{\mathcal{A}}_{i}^k \partial_k\Psi^j,\\
 &\partial_{\ell} \widehat{\mathcal{A}}_{i}^j=-\widehat{\mathcal{A}}_{k}^j\widehat{\mathcal{A}}_{i}^{h} \partial_{h}\partial_{\ell}\Psi^k, \quad \partial_{i} \widehat{v}^j=\partial_{i}\Phi^k \widehat{\mathcal{A}}_{k}^h \partial_{h}\widehat{v}^j= \widehat{\mathcal{A}}_{i}^h \partial_{h}\widehat{v}^j+\partial_{i}\Psi^k \widehat{\mathcal{A}}_{k}^h \partial_{h}\widehat{v}^j,\\
 &\partial_t \widehat{J} =\widehat{J} \widehat{\mathcal{A}}_{i}^j \partial_j \widehat{v}^i, \quad \partial_{\ell} \widehat{J} =\widehat{J} \widehat{\mathcal{A}}_{i}^j \partial_j\partial_{\ell} \Psi^i, \quad\partial_t \widehat{\mathcal{A}}_{i}^j=-\widehat{\mathcal{A}}_{k}^j\widehat{\mathcal{A}}_{i}^{\ell} \partial_{\ell}\widehat{v}^k.
 \end{split}
 \end{equation}
 If the displacement $\Psi$ is sufficiently small in an appropriate Sobolev space, then the flow mapping $\Phi$ is a
 diffeomorphism from $\Omega_0^v$ to $\Omega^v(t)$, which allows us to switch back and forth from Lagrangian
 to Eulerian coordinates.
 
 Denote $(\nabla_{\widehat{\mathcal{A}}})_i=\widehat{\mathcal{A}}_i^j \partial_j$, then we may write the vacuum equations in Lagrangian coordinates as follows
 \begin{equation}\label{mhd-vacuum-Lag-2}
 \begin{split}
 &\nabla_{\widehat{\mathcal{A}}} \cdot \widehat{b}=0,\quad \nabla_{\widehat{\mathcal{A}}} \times \widehat{b}=0\quad \mbox{in}\quad \Omega_0^v.
 \end{split}
 \end{equation}
 \subsection{Decompositions of Lagrangian quantities around the equilibrium}
 
 \subsubsection{Decompositions of $J$ and $b$}
 
 We may compute the Jacobian of the Lagrangian transformation as follows
 \begin{equation*}\label{J-expression-1}
 \begin{split}
 & J=\det(D(h))=\Big(1+\frac{\partial_{\theta}g^{\theta}}{r}+ \frac{1}{r}g^{r}\Big)\Big((1+\partial_zg^z)(1+\partial_rg^{r})-\partial_zg^r \partial_{r}g^z\Big)\\
 &\quad+\Big(\frac{\partial_{\theta}g^{r}}{r}-\frac{g^{\theta}}{r}\Big)\partial_zg^{\theta}\partial_rg^z+\frac{\partial_{\theta}g^z}{r}\partial_{r}g^{\theta}\partial_{z}g^{r}-\Big(\frac{\partial_{\theta}g^{r}}{r}-\frac{g^{\theta}}{r}\Big)\partial_{r}g^{\theta}(1+\partial_zg^z)\\
 &\quad-\frac{\partial_{\theta}g^z}{r}\partial_{z}g^{\theta}(1+\partial_{r}g^{r})
 =1+J_1=1+\nabla \cdot g+J_2,
 \end{split}
 \end{equation*}
 where
 \begin{equation*}\label{J-expression-2a}
 \begin{split}
 &J_1:= \nabla \cdot g+J_2,\quad \nabla \cdot g=\partial_r g^{r}+\frac{1}{r} g^{r}+\partial_zg^z+\frac{\partial_{\theta}g^{\theta}}{r},
 \\&J_2:=\partial_{r}g^{r} \partial_zg^z+\Big(\frac{\partial_{\theta}g^{\theta}}{r}+ \frac{1}{r}g^{r}\Big)(\partial_zg^z+\partial_{r}g^{r}+\partial_{r}g^{r} \partial_zg^z-\partial_rg^z \partial_zg^{r})- \partial_zg^{r} \partial_rg^z\\
 &\qquad+\Big(\frac{\partial_{\theta}g^{r}}{r}-\frac{g^{\theta}}{r}\Big)\partial_zg^{\theta}\partial_rg^z+\frac{\partial_{\theta}g^z}{r}\partial_{r}g^{\theta}\partial_{z}g^{r}-\Big(\frac{\partial_{\theta}g^{r}}{r}-\frac{g^{\theta}}{r}\Big)\partial_{r}g^{\theta}(1+\partial_zg^z)\\
 &\qquad-\frac{\partial_{\theta}g^z}{r}\partial_{z}g^{\theta}(1+\partial_{r}g^{r}).
 \end{split}
 \end{equation*}
 Denote
 \begin{equation}\label{b2}
 b_1 \eqdefa  b_0 \cdot \nabla g-b_0\nabla \cdot g,
 \end{equation}
 $$b_2\eqdefa -J^{-1}((J_2+\nabla\cdot g)b_1+J_2b_0),$$
 then we can split $b$ into three parts
 \begin{equation}\label{b-decomp-3}
 \begin{split}
 b = b_0+b_1+b_2.
 \end{split}
 \end{equation}
 
 \subsubsection{Decompositions of the pressure $q$}
 We first write
 \begin{equation*}
 \begin{split}
 &q=A\,\rho_0^{\gamma}\, J^{-\gamma}=p_0\,J^{-\gamma} =p_0+ p_0\,J^{-\gamma}(1-J^{\gamma})\\
 &=p_0+ p_0\,J^{-\gamma}\Big(1-(1+\nabla \cdot g+J_2)^{\gamma}\Big),
 \end{split}
 \end{equation*}
 which implies
 $q=p_0\,\Big(1-\gamma\, \nabla \cdot g+Q_2\Big)$,
 with $Q_2\eqdefa (1+\nabla \cdot g+J_2)^{-\gamma} - (1-\gamma\nabla \cdot g)$.
 
 Since we expect that $J-1=\nabla\cdot g+J_2$ is small, we obtain from the Taylor expansion that
 \begin{equation*}\label{q-integration-2}
 \begin{split}
 &Q_2=-\gamma\, J_2+\frac{1}{2}\gamma\, (\gamma+1)(\nabla \cdot g+J_2)^2\\
 &\qquad\quad-\frac{1}{2}\gamma\, (\gamma+1)(\gamma+2)(\nabla \cdot g+J_2)^3\, \int_0^1(1-\tau)^2\Big(1+\tau(\nabla \cdot g+J_2)\Big)^{-\gamma-3}\,d\tau.
 \end{split}
 \end{equation*}
 Therefore, we split $q$ into three parts
 $q \eqdefa p_0+q_1+q_2$,
 with $q_1=-\gamma p_0 \nabla \cdot g$ and $q_2=p_0\, Q_2$.
 
 \subsubsection{Decompositions of the normal vector $\mathcal{N}$ on the free surface }
 Let $n_0=e_r$, $\mathcal{N}_i=J\mathcal{A}_i^jn_{0,j}$,
 then we have
 \begin{equation}\label{decom-n}
 \mathcal{N}=n_0+n_1+n_2,
 \end{equation}with 
 $$n_1=\Big(\partial_zg^z+\frac{\partial_{\theta}g^{\theta}}{r}+\frac{g^{r}}{r}\Big)e_r+\Big(\frac{g^{\theta}}{r}-\frac{\partial_{\theta}g^{r}}{r}\Big)e_\theta-\partial_zg^{r}e_z,$$
 \begin{equation*}
 \begin{split}
 &n_2=\Big[\Big(\frac{\partial_{\theta}g^{\theta}}{r}+\frac{g^{r}}{r}\Big)\partial_zg^z-\frac{\partial_zg^\theta\partial_{\theta}g^z}{r}\Big]e_r
 +\Big[\Big(\frac{g^{\theta}}{r}-\frac{\partial_{\theta}g^{r}}{r}\Big)\partial_zg^z
 +\frac{\partial_{\theta}g^z\partial_zg^{r}}{r}\Big]e_\theta\\
 &\qquad+\Big[\Big(\frac{\partial_{\theta}g^{r}}{r}-\frac{g^\theta}{r}\Big)\partial_zg^\theta
 -\Big(\frac{g^{r}}{r}+\frac{\partial_{\theta}g^{\theta}}{r}\Big)\partial_zg^{r}\Big]e_z.
 \end{split}
 \end{equation*}
 \subsubsection{Decompositions of Lagrangian quantities around the equilibrium in vacuum }

 From Lemma \ref{steady-lem}, we know that the equilibrium vacuum magnetic field $\widehat{   B}=\widehat{  B}_\theta(r)e_\theta= B_\theta(r_0)\frac{r_0}{r}e_\theta$. So in vacuum, we will use \eqref{harmonic-extension-0} and \eqref{harmonic-equation-1} to split $\widehat{b}$ into three parts in Lagrangian coordinates as
 \begin{equation}\label{vacuum-b-decomp-1}
 \begin{split}
 &\widehat{b}\eqdefa \widehat{b}_0+\widehat{b}_1+\widehat{b}_2,
 \end{split}
 \end{equation}
 where
 \begin{equation*}
 \begin{split}
 &\widehat{  b}_0=\overline{ B}_\theta(r_0)\frac{r_0}{r}e_\theta,\\
 &\widehat{b}_1 \, (\, \text{fisrt order about}\, \Psi),\\
 &\widehat{  b}_2 =O\, (\, \text{nonlinear term about}\, \Psi).
 \end{split}
 \end{equation*}
 
 From the  vacuum equations in Lagrangian coordinates \eqref{mhd-vacuum-Lag-2}, we can get the linearized vacuum equations in a perturbation around steady solution 
 as follows 
 \begin{equation*}
 \begin{cases}
 \nabla \cdot \widehat{b}_1+\nabla_{\widehat{\mathcal{A}_1}}\cdot \widehat{  b}_0=0,\\
 \nabla\times \widehat{b}_1+\nabla_{\widehat{\mathcal{A}_1}}\times \widehat{  b}_0=0,\\
 n\cdot \widehat{b}_1=0,\, \,  \mbox{on}\, \, r=r_w,\\
 n_0\cdot (b_1-\widehat{b}_1)=n_1\cdot (\widehat{  b}_0-b_0)\, \,  \mbox{on}\, \, r=r_0,
 \end{cases}
 \end{equation*}
 with
 $$n_0=e_r, \quad b_1= b_0\cdot \nabla g-b_0\nabla\cdot g,$$
 $$n_1=\Big(\partial_z\Psi^z+\frac{\partial_{\theta}\Psi^{\theta}}{r}+\frac{\Psi^r}{r}\Big)e_r
 +\Big(\frac{\Psi^{\theta}}{r}-\frac{\partial_{\theta}\Psi^r}{r}\Big)e_\theta-\partial_z\Psi^re_z,$$
 \begin{equation*}\begin{split}
 \widehat{\mathcal{A}}_1=
 \left(\begin{array}{ccc}
 -\partial_r\Psi^r&-\partial_{r}\Psi^{\theta}&-\partial_r\Psi^z\\
 -\frac{\partial_{\theta}\Psi^r}{r}+ \frac{\Psi^\theta}{r}&-\frac{\partial_{\theta}\Psi^\theta}{r}-\frac{\Psi^r}{r}&-\frac{\partial_{\theta}\Psi^z}{r}\\
 -\partial_z\Psi^r&-\partial_z\Psi^\theta&-\partial_z\Psi^z
 \end{array}
 \right).
 \end{split}
 \end{equation*}
 From the steady solution $\widehat{  b}_0=\widehat{  B}_\theta(r) e_\theta= B_\theta(r_0)\frac{r_0}{r}e_\theta$ in vacuum domain, it follows that
 \begin{equation*}\begin{split}
 \nabla \widehat{  b}_0=
 \left(\begin{array}{ccc}
 0&-\frac{\widehat{  B}_\theta}{r}&0\\
 \partial_r\widehat{  B}_\theta&0&0\\
 0&0&0
 \end{array}
 \right),
 \end{split}
 \end{equation*}
 which gives that
 $$\nabla_{\widehat{\mathcal{A}}_1}\cdot \widehat{  b}_0=Tr(\widehat{\mathcal{A}}_1^T\nabla\widehat{  b}_0)=\frac{\widehat{  B}_\theta}{r}\partial_r\Psi^\theta-\partial_r{\widehat{  B}}_\theta\Big(\frac{1}{r}\partial_\theta\Psi^r-\frac{\Psi^\theta}{r}\Big) =-\nabla \cdot (\Psi \cdot \nabla \widehat{  b}_0).$$
 On the other hand, we have 
 \begin{equation*}
 \begin{split}
 \nabla_{\widehat{\mathcal{A}}_1}\times \widehat{  b}_0&=\epsilon_{ijk}\widehat{\mathcal{A}}_{1jl}\partial_l\widehat{  b}_0^k=\epsilon_{ijk}(\widehat{\mathcal{A}}_{1jl}(\nabla \widehat{  b}_0)_{kl})=\epsilon_{ijk}(\widehat{\mathcal{A}}_{1}(\nabla \widehat{  b}_0)^T)_{jk}
 \\&=\epsilon_{ijk}\left(\begin{array}{ccc}
 \partial_r\Psi^\theta\frac{\widehat{  B}_\theta}{r}&-\partial_{r}\Psi^{r}\partial_r\widehat{  B}_\theta&0\\
 \Big(\frac{\partial_{\theta}\Psi^\theta}{r}+ \frac{\Psi^r}{r}\Big)\frac{\widehat{  B}_\theta}{r}&\Big(-\frac{\partial_{\theta}\Psi^r}{r}+\frac{\Psi^\theta}{r}\Big)\partial_r\widehat{  B}_\theta&0\\
 \partial_z\Psi^\theta\frac{\widehat{  B}_\theta}{r}&-\partial_z\Psi^r\partial_r\widehat{  B}_\theta&0
 \end{array}
 \right)_{jk}
 \\
 &=e_r\partial_z\Psi^r\partial_r\widehat{  B}_\theta+e_\theta\partial_z\Psi^\theta\frac{\widehat{  B}_\theta}{r}+e_z\Big[-\Big(\frac{\partial_{\theta}\Psi^\theta}{r}+ \frac{\Psi^r}{r}\Big)\frac{\widehat{  B}_\theta}{r}-\partial_{r}\Psi^{r}\partial_r\widehat{  B}_\theta\Big],
 \end{split}
 \end{equation*}
 \begin{equation*}
 \begin{split}
 - \nabla \times (\Psi \cdot \nabla \widehat{  b}_0)&=-\Big(e_r\partial_r+\frac{e_\theta}{r}\partial_{\theta}+e_z\partial_z\Big)\times\Big(e_\theta\Psi^r\partial_r\widehat{  B}_\theta-e_r\frac{\Psi^\theta\widehat{  B}_\theta}{r}\Big)\\
 &=e_r\partial_{z}\Psi^r\partial_r\widehat{  B}_\theta
 +e_\theta\frac{\partial_{z}\Psi^\theta\widehat{  B}_\theta}{r}-e_z\frac{\partial_{\theta}\Psi^\theta\widehat{  B}_\theta}{r}-e_z\partial_{r}\Psi^{r}\partial_r\widehat{  B}_\theta-e_z\frac{\Psi^r\widehat{  B}_\theta}{r^2}.
 \end{split}
 \end{equation*}
 So we can show that 
 $\nabla_{\widehat{\mathcal{A}_1}}\times \widehat{  b}_0=- \nabla \times (\Psi \cdot \nabla \widehat{  b}_0)$.
 Therefore, from $b_1=b_0\cdot \nabla g-b_0\nabla\cdot g$ in \eqref{b2}, it follows that
 \begin{equation*}
 \begin{cases}
 \nabla \cdot (\widehat{b}_1-\Psi \cdot \nabla \widehat{  b}_0)=0,\\
 \nabla \times (\widehat{b}_1-\Psi \cdot \nabla \widehat{  b}_0)=0,\\
 n\cdot \widehat{  b}_1=0,\, \,  \mbox{on}\, \, r=r_w,\\
 n_0\cdot(b_0\cdot \nabla g-b_0\nabla\cdot g-\widehat{  b}_1)=n_1\cdot (\widehat{  b}_0-b_0),\,  \,  \mbox{on}\, \, r=r_0,
 \end{cases}
 \end{equation*}
 with
 $n_1=\Big(\partial_z\Psi^z+\frac{\partial_{\theta}\Psi^{\theta}}{r}+\frac{\Psi^r}{r}\Big)e_r+\Big(\frac{\Psi^{\theta}}{r}-\frac{\partial_{\theta}\Psi^r}{r}\Big)e_\theta-\partial_z\Psi^re_z.$
 Denoting $\widehat{ Q}=\widehat{b}_1-\Psi \cdot \nabla \widehat{  b}_0$,
 using the fact that  $\widehat{  b}_0= b_0$  on the boundary $r=r_0$,  from \eqref{harmonic-extension-0} and \eqref{harmonic-equation-1}, 
 we can show that on the boundary $r=r_0$  
 \begin{equation*}
 \begin{split}
 &n_0\cdot (b_0\cdot \nabla g-b_0\nabla\cdot g-\widehat{  b}_1)=n_0\cdot(b_0\cdot \nabla g-b_0\nabla\cdot g-g\cdot \nabla \widehat{  b}_0-\widehat{ Q})\\
 &=n_0\cdot(b_0\cdot \nabla g-b_0\nabla\cdot g-g\cdot \nabla b_0-\widehat{ Q})=n_0\cdot\Big(\nabla \times (g\times b_0)-\widehat{ Q}\Big)=0.
 \end{split}
 \end{equation*}
 Therefore,  in vacuum domain, we obtain 
 \begin{equation}\label{div-curl}
 \begin{cases}
 \nabla \cdot \widehat{Q}=0,\\
 \nabla \times \widehat{Q}=0,\\
 n\cdot \widehat{ Q}=0,\, \,  \mbox{on}\, \, r=r_w,\\
 n_0\cdot\nabla \times (g\times \widehat{  b}_0)=n_0\cdot \widehat{ Q}\, \,  \mbox{on}\, \, r=r_0.
 \end{cases}
 \end{equation}
 
 \subsection{Perturbed MHD system in plasma}
 
 Thanks to the decomposition of $b$ and $q$ again, we have
 \begin{equation}\label{pressure-decom-1}
 \begin{split}
 &q+\frac{1}{2}|b|^2=p_0+q_1+q_2+\frac{1}{2}|b_0+b_1+b_2|^2\\
 &=p_0+\frac{1}{2}|b_0|^2-\gamma p_0 \nabla \cdot g+b_0\cdot b_1+R_{1, p}
 \end{split}
 \end{equation}
 with
 \begin{equation}\label{r_ip}
 R_{1, p}:= b_0\cdot b_2+ q_2+\frac{1}{2}|b_1+b_2|^2.
 \end{equation}
 
 While for $b \cdot \nabla_{\mathcal{A}} b$, it shows that
 $J\,b \cdot \nabla_{\mathcal{A}} b=b_0\cdot \nabla b$,
 where we have used the equality \eqref{b-identity-1} $J b^j \mathcal{A}_j^i=b_0^i$.

 Since $b=J^{-1} (b_0+b_0\cdot \nabla g)$, we decompose $J\,b \cdot \grad_{\mathcal{A}} b$ as 
 \begin{equation}\label{r_1b}
 \begin{split}
 J\,b \cdot \grad_{\mathcal{A}} b&=b_0\cdot\nabla b_0+b_0\cdot\nabla b_1+b_0\cdot\nabla b_2\\
 & =b_0\cdot\nabla b_0+b_0\cdot\nabla b_1+JR_{1,b}.
 \end{split}
 \end{equation}
 
 Using \eqref{pressure-decom-1}, we may deduce that
 \begin{equation}\label{pressure-decom-5}
 \Big(\nabla_{\mathcal{A}} \big(q+\frac{1}{2}|b|^2\big)\Big)^k=\mathcal{A}_k^j\partial_j\Big(p_0+\frac{1}{2}|b_0|^2-\gamma p_0 \nabla \cdot g+b_0\cdot b_1\Big)+\mathcal{A}_k^j\partial_j \,R_{1, p}.
 \end{equation}
 
 Combining \eqref{r_1b} with  \eqref{pressure-decom-5}, we obtain
 \begin{equation}\label{pressure-decom-9}
 \begin{split}
 &J\,\nabla_{\mathcal{A}} \Big(q+\frac{1}{2}|b|^2\Big)-J\,b \cdot \nabla_{\mathcal{A}} b\\
 &\quad=J\,\nabla_{\mathcal{A}}\Big(p_0+\frac{1}{2}|b_0|^2-\gamma p_0 \nabla \cdot g+b_0\cdot b_1\Big)-b_0\cdot\nabla b_0-b_0\cdot\nabla b_1\\
 &\qquad+J\,\nabla_{\mathcal{A}} \,R_{1, p}-J\,R_{1, b}.
 \end{split}
 \end{equation}
 Substituting \eqref{f-q-1} and \eqref{pressure-decom-9} into the momentum equations of \eqref{mhd-fluid-2} results in
 \begin{equation*}
 \begin{split}
 & \rho_0\partial_t v + J \nabla_{\mathcal{A}}\Big(p_0+\frac{1}{2}|b_0|^2-\gamma p_0 \nabla \cdot g+b_0\cdot b_1\Big)-b_0\cdot\nabla b_0-b_0\cdot\nabla b_1\\
 &\quad =J\, R_{1, b}-J\,\nabla_{\mathcal{A}} \,R_{1, p}.
 \end{split}
 \end{equation*}
 
 Let's now deal with the jump conditions on $\Sigma_{0, pv}$ in \eqref{mhd-fluid-2}. In fact, thanks to \eqref{pressure-decom-1}, we know that
 \begin{equation*}
 \begin{split}
 \bigg(q+\frac{1}{2}|b|^2-\frac{1}{2}|\widehat{b}|^2\bigg)\bigg|_{\Sigma_{0, pv}}= \bigg(p_0+\frac{1}{2}|b_0|^2-\gamma p_0 \nabla \cdot g+b_0\cdot b_1+R_{1, p}-\frac{1}{2}|\widehat{b}|^2\bigg)\bigg|_{\Sigma_{0, pv}}.
 \end{split}
 \end{equation*}
 From the decomposition of $\widehat{  b}$ in \eqref{vacuum-b-decomp-1}, it follows that
 \begin{equation}\label{hat-b}
 \frac{1}{2}|\widehat{b}|^2=\frac{1}{2}|\widehat{b}_0|^2+\widehat{  b}_0\cdot\widehat{  b}_1+\frac{1}{2}|\widehat{  b}_1|^2+\frac{1}{2}|\widehat{  b}_2|^2+\widehat{  b}_0 \cdot \widehat{  b}_2+\widehat{  b}_1 \cdot \widehat{  b}_2=\frac{1}{2}|\widehat{b}_0|^2+\widehat{  b}_0\cdot\widehat{  b}_1+\widehat{ R}_{1,p},
 \end{equation}
 which along with \eqref{z-pinch-1-1}, \eqref{harmonic-extension-0}, \eqref{harmonic-equation-1} and \eqref{b2}, yields that
 \begin{equation}\label{widehat-r_1p}
 \begin{split}
 &\bigg(q+\frac{1}{2}|b|^2-\frac{1}{2}|\widehat{b}|^2\bigg)\bigg|_{\Sigma_{0, pv}}\\
 &\quad= \bigg(-\gamma p_0 \nabla \cdot g+b_0\cdot b_1+R_{1, p}-\widehat{b}_0\cdot \widehat{  b}_1-\widehat{  R}_{1,p} \bigg)\bigg|_{\Sigma_{0, pv}}\\
 &\quad=\bigg(-\gamma p_0\nabla\cdot g+b_0\cdot Q+g\cdot\nabla \Big (\frac{1}{2}|b_0|^2\Big)-\widehat{b_0}\cdot\widehat{Q}-g\cdot\nabla\Big(\frac{1}{2}|\widehat{ b}_0|^2\Big)\\
 &\qquad+R_{1,p}-\widehat{  R}_{1,p}\bigg)\bigg|_{\Sigma_{0, pv}},
 \end{split}
 \end{equation}
 with $Q=\nabla \times (g\times b_0)$, $\widehat{ Q}=\widehat{b}_1-\Psi \cdot \nabla \widehat{  b}_0$.

In conclusion,
we rephrase the MHD system \eqref{mhd-fluid-2} in a perturbation formulation around the steady solution (see special steady solution for a $z$-pinch in \eqref{z-pinch-1-1}) as follows:
\begin{equation}\label{eqns-pert-plasma-vacuum-1-viscosity}
\begin{cases}
& \partial_t g=v\quad \mbox{in}\quad \Omega_0, \\
&  \rho_0\partial_t v + J \nabla_{\mathcal{A}}\Big(p_0+\frac{1}{2}|b_0|^2-\gamma p_0 \nabla \cdot g+b_0\cdot b_1\Big)-b_0\cdot\nabla b_0-b_0\cdot\nabla b_1\\
&\quad =J\, R_{1, b}-J\,\nabla_{\mathcal{A}} \,R_{1, p}\quad \mbox{in}\quad \Omega_0,\\
& \nabla_{\widehat{\mathcal{A}}}\cdot \widehat{b}=0, \quad  \nabla_{\widehat{\mathcal{A}}} \times \widehat{b}=0 \quad \mbox{in} \quad \Omega^v_0,\\
& n \cdot b|_{\Sigma_{0, pv}}=n \cdot \widehat{b}|_{\Sigma_{0, pv}}=0, \quad  n\cdot \widehat{b}|_{\Sigma_w}=0,\\
&-\gamma p_0\nabla\cdot g+b_0\cdot Q+g\cdot\nabla \big (\frac{1}{2}|b_0|^2\big)-\widehat{b_0}\cdot(\widehat{b}_1-g \cdot \nabla \widehat{  b}_0)-g\cdot\nabla\big(\frac{1}{2}|\widehat{ b}_0|^2\big)\\
&\qquad+R_{1,p}-\widehat{  R}_{1,p}=0\quad \mbox{on} \quad \Sigma_{0, pv},\\
&g|_{t=0}=g_0, \,   v|_{t=0}=v_0,
\end{cases}
\end{equation}
with $Q=\nabla\times (g\times b_0)$, $b_1$ defined in \eqref{b2}, $R_{1,p}$ defined in \eqref{r_ip}, $JR_{1,b}$ defined in \eqref{r_1b} and $\widehat{ R}_{1,p}$ defined in \eqref{hat-b}.
Let the initial data  as the steady solution, from the force $\nabla ( p+\frac{1}{2}|B|^2)=B\cdot \nabla B$ of the $z$-pinch $(p, B, \widehat{B})$, then the linearized MHD system in a perturbation formulation around the steady solution takes the following form
\begin{equation}
\begin{cases}
&\partial_tg=v\quad \mbox{in}\quad \overline{\Omega},\\
& \rho\partial_{tt} g+\nabla ( -\gamma p \nabla \cdot g+B\cdot b_1)-B\cdot\nabla b_1+\nabla_{\mathcal{A}_1}\big(p+\frac{1}{2}|B|^2\big)\\
&\quad+(\nabla\cdot g) \nabla\big(p+\frac{1}{2}|B|^2\big) =0\quad \mbox{in}\quad \overline{\Omega},\\
&  \nabla \cdot \widehat{b}_1+\nabla_{\widehat{\mathcal{A}_1}}\cdot \widehat{ B}=0,\quad \mbox{in}\quad \overline{\Omega}^v,\\
&\nabla\times \widehat{b}_1+\nabla_{\widehat{\mathcal{A}_1}}\times \widehat{ B}=0,\quad \mbox{in}\quad \overline{\Omega}^v,\\
& n_0\cdot (b_1-\widehat{b}_1)=n_1\cdot (\widehat{ B}-B)\quad \mbox{on} \quad \Sigma_{0, pv},\\
&-\gamma p\nabla\cdot g+B\cdot Q+g\cdot\nabla \big (\frac{1}{2}|B|^2\big)\\
&\quad=\widehat{ B}\cdot(\widehat{b}_1-g \cdot \nabla \widehat{  B})+g\cdot\nabla\big(\frac{1}{2}|\widehat{ B}|^2\big), \quad \mbox{on} \quad \Sigma_{0, pv},\\
&  n\cdot {\widehat{b}_1}|_{\Sigma_w}=0,\, g|_{t=0}=g_0,
\end{cases}
\end{equation}
with $Q=\nabla\times (g\times B)$, $b_1 = B \cdot \nabla g-B\nabla \cdot g$ and  $\mathcal{A}_1$ is the first order of $\mathcal{A}$, that is, in cylindrical coordinates, 
\begin{equation}\label{first-order-A}
\begin{split}
\mathcal{A}_1=
\left(\begin{array}{ccc}
-\partial_rg^{r}&-\partial_{r}g^{\theta}&-\partial_rg^z\\
-\frac{\partial_{\theta}g^{r}}{r}+ \frac{g^\theta}{r}&-\frac{\partial_{\theta}g^\theta}{r}-\frac{g^{r}}{r}&-\frac{\partial_{\theta}g^z}{r}\\
-\partial_zg^{r}&-\partial_zg^\theta&-\partial_zg^z
\end{array}
\right),
\end{split}
\end{equation} $$n_1=(\partial_z\Psi^z+\frac{\partial_{\theta}\Psi^{\theta}}{r}+\frac{\Psi^r}{r})e_r+\bigg(\frac{\Psi^{\theta}}{r}-\frac{\partial_{\theta}\Psi^r}{r}\bigg)e_\theta-\partial_z\Psi^re_z,$$
\begin{equation}\begin{split}
\widehat{\mathcal{A}}_1=
\left(\begin{array}{ccc}
-\partial_r\Psi^r&-\partial_{r}\Psi^{\theta}&-\partial_r\Psi^z\\
-\frac{\partial_{\theta}\Psi^r}{r}+ \frac{\Psi^\theta}{r}&-\frac{\partial_{\theta}\Psi^\theta}{r}-\frac{\Psi^r}{r}&-\frac{\partial_{\theta}\Psi^z}{r}\\
-\partial_z\Psi^r&-\partial_z\Psi^\theta&-\partial_z\Psi^z
\end{array}
\right).
\end{split}
\end{equation}
Denote the new function $\widehat{ Q}=\widehat{b}_1-\Psi \cdot \nabla \widehat{  b}_0=\widehat{b}_1-\Psi \cdot \nabla \widehat{  B}$, 
after a computation,  we can get the above system is equivalent to the following equations 
\begin{equation}\label{linear-perturbation-and-boundary-viscosity}
\begin{cases}
&\partial_tg=v\quad \mbox{in}\quad \overline{\Omega},\\
& \rho\partial_{tt} g =\nabla(g\cdot\nabla p+\gamma p\nabla\cdot g)+(\nabla \times B)\times [\nabla \times (g \times B)]\\
&\quad+\{\nabla \times [\nabla \times (g \times B)]\}\times B,\quad
 \mbox{in}\quad \overline{\Omega},\\
&  \nabla \cdot \widehat{Q}=0,\quad \mbox{in}\quad \overline{\Omega}^v,\\
& \nabla \times \widehat{Q}=0,\quad \mbox{in}\quad \overline{\Omega}^v,\\
& n \cdot \nabla \times (g\times \widehat{B})=n\cdot \widehat{Q}, \quad \mbox{on} \quad \Sigma_{0, pv},\\
&-\gamma p\nabla\cdot g+B\cdot Q+g\cdot\nabla \big (\frac{1}{2}|B|^2\big)=\widehat{ B}\cdot\widehat{Q}+g\cdot\nabla\big(\frac{1}{2}|\widehat{ B}|^2\big), \quad \mbox{on} \quad \Sigma_{0, pv},\\
&  n\cdot\widehat{Q}|_{\Sigma_w}=0,\, g|_{t=0}=g_0,
\end{cases}
\end{equation}
with $Q=\nabla\times (g\times B)$.

From the divergence free condition about the magnetic field  in \eqref{mhd-fluid-2}, we know that $\nabla_{\mathcal{A}}\cdot b=0$ holds in lagrangian coordinates. We now prove that the linear perturbation of $\nabla_{\mathcal{A}}\cdot b=0$ holds automatically.
\begin{rmk}\label{div-free-condition}
	Assume  the steady solution $b_0=B_\theta(r)e_\theta$, then $\nabla_{\mathcal{A}_1}\cdot b_0+\nabla\cdot b_1=0$ is the linear perturbation of $\nabla_{\mathcal{A}}\cdot b=0$ and this linear perturbation holds for any function $g$, where $b_1 $ is defined in  \eqref{b2} and  $\mathcal{A}_1$ is defined in \eqref{first-order-A}.
\end{rmk}
\begin{proof}
 From the decomposition of $\mathcal{A}$ and $b$ in lagrangian coordinates, that is, $\mathcal{A}=I+\mathcal{A}_1+O(\mbox{nonlinear matrix about} \quad g)$ and \eqref{b-decomp-3}, it follows that  the corresponding  linear perturbation is $\nabla_{\mathcal{A}_1}\cdot b_0+\nabla\cdot b_1=0$.
	From the steady solution $b_0= B_\theta(r) e_\theta$, it follows that
	\begin{equation*}
	\begin{split}
	\nabla   b_0=
	\left(\begin{array}{ccc}
	0&-\frac{B_\theta}{r}&0\\
	\partial_r B_\theta&0&0\\
	0&0&0
	\end{array}
	\right),
	\end{split}
	\end{equation*}
	which gives that
	$$\nabla_{\mathcal{A}_1}\cdot   b_0=Tr(\mathcal{A}_1^T\nabla b_0)=\frac{  B_\theta}{r}\partial_rg^\theta-\partial_r{ B}_\theta\Big(\frac{1}{r}\partial_\theta g^r-\frac{g^\theta}{r}\Big) =-\nabla \cdot (g \cdot \nabla  b_0).$$
	Therefore, $\nabla_{\mathcal{A}_1}\cdot b_0+\nabla\cdot b_1=\nabla \cdot(b_0\cdot \nabla g-b_0\nabla \cdot g-g\cdot \nabla b_0)$. On the other hand,  we have $\nabla \cdot b_0=0$, which implies the identity
	$b_0\cdot \nabla g-b_0\nabla \cdot g-g\cdot \nabla b_0=\nabla \times (g\times b_0)$. 
Since from $b_0=B_\theta(r)e_\theta$,	we can show that
\begin{equation*}
\begin{split}
\nabla \cdot [\nabla \times (g\times b_0)]&=\Big(e_r\partial_r+\frac{e_\theta}{r}\partial_{\theta}+e_z\partial_z\Big)\cdot \Big[\Big(e_r\partial_r+\frac{e_\theta}{r}\partial_{\theta}+e_z\partial_z\Big)\times(g^rB_\theta e_z-g^zB_\theta e_r)\Big]\\
&=\Big(e_r\partial_r+\frac{e_\theta}{r}\partial_{\theta}+e_z\partial_z\Big)\cdot\Big[e_r\frac{B_\theta\partial_{\theta}g^r}{r}
-e_\theta\Big(\partial_{r}(g^rB_\theta)+B_\theta\partial_{z}g^z\Big)+e_z\frac{B_\theta\partial_{\theta}g^z}{r}\Big]\\
&=\partial_{r}\Big(\frac{B_\theta\partial_{\theta}g^r}{r}\Big)+\frac{B_\theta\partial_{\theta}g^r}{r^2}
-\frac{1}{r}\Big(\partial_{r}(\partial_{\theta}g^rB_\theta)+B_\theta\partial_{z}\partial_{\theta}g^z\Big)+\frac{B_\theta\partial_{\theta}\partial_{z}g^z}{r}=0.
\end{split}
	\end{equation*}
	Hence, $\nabla_{\mathcal{A}_1}\cdot b_0+\nabla\cdot b_1=0$ holds automatically, which implies the result. 
\end{proof}
In order to see the property of the force operator  $$F(g)=\nabla(g\cdot\nabla p+\gamma p\nabla\cdot g)+(\nabla \times B)\times [\nabla \times (g \times B)]+\{\nabla \times [\nabla \times (g \times B)]\}\times B,$$  we consider two displacement vector fields $g$ and $h$ defined over the plasma volume $V$, their associated magnetic field perturbations
\begin{equation*}
	Q=\nabla\times (g\times B), \quad R=\nabla\times (h\times B),
\end{equation*}
and the vacuum perturbations $\widehat{ Q}$ and  $\widehat{ R}$ defined over the vacuum volume $\widehat{ V}$ are their extensions,
that is, to `extend' the function $g$ into the vacuum by means of the magnetic field variable $\widehat{ Q}$, and likewise to `extend' $h$ by means of $\widehat{ R}$.
Then by Chapter 6 in \cite{Goedbloed-poedts}, we have the following lemma.
\begin{lem}\label{joint}
	Assume $g\in H^2$ is a solution of \eqref{linear-perturbation-and-boundary-viscosity}, then we get a meaning expression for the potential energy of interface plasma by identifying $g$, $h$,  $\widehat{ Q}$ and $\widehat{ R}$, in the  quadratic form 
	\begin{equation*}\label{sym}
		\begin{split}
			\int_{\overline{\Omega}} h \cdot F(g)dx&=-\int_{\overline{\Omega}}\Big[\gamma p \nabla \cdot g \nabla \cdot h+Q\cdot R+\frac{1}{2}\nabla p\cdot(g\nabla\cdot h+h\nabla\cdot g)\\
			&\quad+\frac{1}{2}\nabla \times B\cdot (g \times R+h\times Q)\Big]dx  -\int_{\overline{\Omega}^v}\widehat{Q}\cdot \widehat{ R}\,dx\\
			&\quad-\int_{\Sigma_{0, pv}} n\cdot g \, n \cdot h\, n\cdot \Big[\Big[\nabla (p+\frac{1}{2}|B|^2)\Big]\Big]\,dx,
		\end{split}
	\end{equation*}
	which is symmetric in the variables $g$ and $h$, and their extensions $\widehat{ Q}$ and $\widehat{ R}$. 
 \end{lem}
\begin{proof} The proof can be recalled from Chapter 6 in \cite{Goedbloed-poedts}, for completeness, we give it as follows.
	By the equilibrium equation $\nabla p=(\nabla \times B)\times B$, we have
	\begin{equation*}
		\begin{split}
			\nabla (g\cdot \nabla p)&=(\nabla g)\cdot \nabla p+g\cdot \nabla\nabla p=(\nabla p\times\nabla)\times g+\nabla p\nabla\cdot g +g\cdot \nabla\nabla p\\
			&=(((\nabla \times B)\times B)\times\nabla)\times g+\nabla p \nabla\cdot g+g\cdot \nabla\nabla p\\
			&=(B(\nabla \times B)\cdot \nabla -(\nabla \times B) B\cdot \nabla)\times g+\nabla  p\nabla\cdot g +g\cdot \nabla\nabla p\\
			&=B\times ((\nabla \times B)\cdot \nabla g)-(\nabla \times B)\times (B\cdot\nabla g)+\nabla  p\nabla\cdot g+g\cdot \nabla\nabla p,
		\end{split}
	\end{equation*}
	which together with 
	\begin{equation*}
		\begin{split}
			(\nabla \times B)\times [\nabla \times (g \times B)]&=(\nabla \times B)\times(B\cdot\nabla g-B\nabla\cdot g-g\cdot \nabla B)\\
			&=(\nabla \times B)\times(B\cdot\nabla g)-(\nabla \times B)\times B\nabla\cdot g\\
			&\quad-g\cdot\nabla((\nabla \times B)\times B)-B\times(g\cdot\nabla (\nabla\times B)),
		\end{split}
	\end{equation*}
	implies that
	\begin{equation}\label{last-two-term-bef}
		\begin{split}
			&\nabla (g\cdot \nabla p)+(\nabla \times B)\times [\nabla \times (g \times B)]\\
			&\quad=B\times((\nabla \times B)\cdot \nabla g)-B\times(g\cdot\nabla (\nabla\times B))\\
			&\quad=-B\times(\nabla\times(\nabla\times B\times g))-(\nabla\times B)\times B\nabla\cdot g \\
			&\quad=-B\times(\nabla\times(\nabla\times B\times g))-\nabla p \nabla\cdot g.
		\end{split}
	\end{equation}
	Exploiting the inner products and by the expression \eqref{last-two-term-bef},  we can rewrite $h \cdot F(g)$  as
	\begin{equation}\label{force-rew}
		\begin{split}
			h\cdot F(g)&=h\cdot\nabla (\gamma p\nabla\cdot g)-h\cdot B\times \{\nabla \times [\nabla \times (g \times B)]+\nabla\times (\nabla\times B\times g)\}\\
			&\quad-h\cdot\nabla p \nabla\cdot g.
		\end{split}
	\end{equation}
	The first term in \eqref{force-rew} gives the following expression
	\begin{equation}\label{first-term}
		h\cdot \nabla (\gamma p\nabla\cdot g)=-\gamma p\nabla\cdot g\nabla\cdot h+\nabla \cdot(h\gamma p\nabla\cdot g).
	\end{equation}
	From the definitions of $Q$ and $R$, the second term in \eqref{force-rew}  can be rewritten as 
	\begin{equation}\label{second-term}
		\begin{split}
			&-h\cdot B\times \{\nabla \times [\nabla \times (g \times B)]\}=-\{\nabla \times [\nabla \times (g \times B)]\}\cdot (h\times B)\\
			&\quad=-\nabla \times (g \times B)\cdot \nabla\times (h\times B)
			+\nabla\cdot [(h\times B)\times \nabla \times (g \times B)]\\
			&\quad=-Q\cdot R+\nabla\cdot [(h\times B)\times Q]=-Q\cdot R+\nabla\cdot [Bh\cdot Q-h B\cdot Q].
		\end{split}
	\end{equation}
	Applying the definitions of  $R$ and the equilibrium equation $\nabla  p=(\nabla \times B)\times B=\nabla \times B\times B$, we can rewrite the third and fourth terms in \eqref{force-rew}  as 
	\begin{equation*}
		\begin{split}
			&-h\cdot B\times [\nabla\times (\nabla\times B\times g)]-(h\cdot \nabla p) \nabla\cdot g\\
			&\quad =-\nabla\times B\times g\cdot R+\nabla\cdot [(h\times B)\times (\nabla\times B\times g)]-(h\cdot \nabla p) \nabla\cdot g\\&\quad
			=g\cdot (\nabla\times B)\times R
			+\nabla\cdot [ (\nabla\times B)B\cdot (g\times h)+gh\cdot (\nabla\times B\times B)]-(h\cdot \nabla p) \nabla\cdot g\\&\quad=g\cdot [\nabla (h\cdot\nabla p)+(\nabla\times B)\times R]+\nabla\cdot [ (\nabla\times B) B\cdot (g\times h)],
		\end{split}
	\end{equation*}
	which together with \eqref{last-two-term-bef} can be symmetrized as 
	\begin{equation}\label{last-two-terms}
		\begin{split}
			&h\cdot \big\{\nabla (g\cdot \nabla p)+(\nabla \times B)\times \big[\nabla \times (g \times B)\big]\big\}=h\cdot \big[\nabla (g\cdot \nabla p)+(\nabla \times B)\times Q\big]\\
			&\quad=\frac{1}{2}h\cdot\big [\nabla (g\cdot \nabla p)+(\nabla \times B)\times Q\big]+\frac{1}{2}g\cdot \big[\nabla (h\cdot \nabla p)+(\nabla \times B)\times R\big]\\
			&\qquad+\frac{1}{2}\nabla\cdot \big[ (\nabla\times B) B\cdot (g\times h)\big]\\
			&\quad=\frac{1}{2}\nabla\cdot \big[\nabla p\cdot (gh+hg)\big]
			-\frac{1}{2}\nabla p\cdot\big(g\nabla\cdot h
			+h\nabla\cdot g\big)-\frac{1}{2}\nabla\times B\cdot \big(g\times R+h\times Q\big)\\
			&\qquad+\frac{1}{2}\nabla\cdot\big [ (\nabla\times B) B\cdot (g\times h)\big]-\frac{1}{2}\nabla\cdot \big[\big(\nabla\times B\times B-\nabla p\big)\cdot\big(gh-hg\big)\big]\\
			&\quad=-\frac{1}{2}\nabla p\cdot \big(g\nabla\cdot h
			+h\nabla\cdot g\big)-\frac{1}{2}\nabla\times B\cdot \big(g\times R+h\times Q\big)\\
			&\qquad+\nabla\cdot \Big[ h(g\cdot \nabla  p)+\frac{1}{2}(\nabla\times B)B\cdot (g\times h)\Big]-\frac{1}{2}\nabla\cdot \big[\big(\nabla\times  B\times B\big)\cdot\big(gh-hg\big)\big].
		\end{split}
	\end{equation}
	Adding up \eqref{first-term}, \eqref{second-term} and \eqref{last-two-terms} shows that
	\begin{equation}\label{sym-form}
		\begin{split}
			h\cdot F(g)&=-\gamma p\nabla\cdot g\nabla\cdot h-Q\cdot R-\frac{1}{2}\nabla p\cdot \big(g\nabla\cdot h+h\nabla\cdot g\big)\\
			&\quad-\frac{1}{2}\nabla\times B\cdot \big(g\times R+h\times Q\big)\\
			&\quad+\nabla\cdot \big[ h(g\cdot \nabla p)\big]-\nabla\cdot (hB\cdot Q)+\nabla \cdot(h\gamma p\nabla\cdot g)\\
			&\quad+\frac{1}{2}\nabla\cdot \big[(\nabla\times B) B\cdot (g\times h)\big]+\nabla\cdot (Bh\cdot Q)\\
			&\quad-\frac{1}{2}\nabla\cdot \big[(\nabla\times B\times  B)\cdot(gh-hg)\big].
		\end{split}
	\end{equation}
	Integrating \eqref{sym-form} gives that
	\begin{equation}\label{sym-form-int}
		\begin{split}
			\int_{\overline{\Omega}}h\cdot F(g) dx&=-\int_{\overline{\Omega}}\Big[\gamma p\nabla\cdot g\nabla\cdot h+Q\cdot R+\frac{1}{2}\nabla p\cdot \big(g\nabla\cdot h+h\nabla\cdot g\big)\\
			&\quad+\frac{1}{2}\nabla\times B\cdot \big(g\times R+h\times Q\big)\Big]dx\\
			&\quad+\int_{\Sigma_{0, pv}} n\cdot h\Big(g\cdot \nabla p-B\cdot Q+\gamma p\nabla\cdot g\Big)dx.
		\end{split}
	\end{equation}
There are no contributions from the eighth, ninth and tenth terms of \eqref{sym-form} to the surface integral, since $n\cdot B=0$ and $n\cdot \nabla\times B=0$ on the plasma surface, whereas  $\nabla \times B\times B$ is parallel to $n$.	
From the second interface condition of \eqref{linear-perturbation-and-boundary-viscosity}, the surface integtral takes the form of
	\begin{equation}\label{tranf-sur}
		\begin{split}
			&\int_{\Sigma_{0, pv}} n\cdot h\Big(g\cdot \nabla p-B\cdot Q+\gamma p\nabla\cdot g\Big)dx\\
			&\quad=-\int_{\Sigma_{0, pv}} n\cdot hg\cdot \Big[\Big[\nabla \big(p+\frac{1}{2}|B|^2\big)\Big]\Big]dS-\int_{\Sigma_{0, pv}} n\cdot h\widehat{  B}\cdot \widehat{ Q}dx\\
			&\quad =-\int_{\Sigma_{0, pv}} n\cdot g \, n \cdot h\, n\cdot \Big[\Big[\nabla \big(p+\frac{1}{2}|B|^2\big)\Big]\Big]\,dS-\int_{\Sigma_{0, pv}} n\cdot h\widehat{  B}\cdot \widehat{ Q}dx.
		\end{split}
	\end{equation}
	Here, we have used the facts the equilibrium jump condition $\Big[\Big[ p+\frac{1}{2}|B|^2\Big]\Big]=0$, which implies that the tangential derivative of the jump vanishes as well
	$\vec{t}\cdot \Big[\Big[\nabla (p+\frac{1}{2}|B|^2)\Big]\Big]=0,$ where $\vec{t}$ is an arbitrary unit vector tangential to the surface. 
	
	Next, let's transform the last term in \eqref{tranf-sur}. 	
	
	For some of the derivations here, it is useful to exploit the alternative representation of test function $\widehat{ R}$ in vacuum  in terms of the vector potential
	$	\widehat{ R}=\nabla\times \widehat{  C},$ and using the first interface condition \eqref{linear-perturbation-and-boundary-viscosity}$_5$ in terms of the vector potential $\widehat{  C}$, that is, $n\cdot h \widehat{  B}=-n\times \widehat{  C}$, one has
	\begin{equation}\label{last-trans}
		\begin{split}
			&-\int_{\Sigma_{0, pv}} n\cdot h\widehat{  B}\cdot \widehat{ Q}dx
			=\int_{\Sigma_{0, pv}} n\times \widehat{  C}\cdot \widehat{ Q} dx =-\int_{\Sigma_{0, pv}} \widehat{ Q}\times\widehat{  C}\cdot n dx\\
			&\qquad=\int_{\overline{\Omega}^v}\nabla \cdot \big[\widehat{ Q}\times\widehat{  C}\big]dx=\int_{\overline{\Omega}^v} \big[\widehat{  C}\cdot \nabla\times\widehat{ Q}-\widehat{ Q}\cdot \nabla\times \widehat{  C}\big]dx\\
			&\qquad=-\int_{\overline{\Omega}^v}\widehat{ Q}\cdot \nabla\times \widehat{  C}dx=-\int_{\overline{\Omega}^v} \widehat{ Q}\cdot \widehat{ R}dx.
		\end{split}
	\end{equation}
	
Now  we prove $	\widehat{ R}=\nabla\times \widehat{  C}.$
	First, extend function $\widehat{ R}$ to the domain $\overline{\Omega} \cup \overline{\Omega}^v$ as follows
	\begin{equation*}
		A=\begin{cases}
			R\quad  
			\mbox{in}\quad \overline{\Omega},\\
			\widehat{  R} \quad
			\mbox{in}\quad \overline{\Omega}^v.
		\end{cases}
	\end{equation*} 
	Notice that $\mbox{div} R=0$ in $\overline{\Omega}$ and $\mbox{div} \widehat{  R}=0$ in $\overline{\Omega}^v$.
	So for test function $\psi$ in $\overline{\Omega} \cup \overline{\Omega}^v$, we have 
	\begin{equation*}
		\begin{split}
			\int_{\overline{\Omega} \cup \overline{\Omega}^v}A\cdot \nabla \psi dx&=\int_{\overline{\Omega}^v} \widehat{  R}\cdot \nabla \psi dx+\int_{\overline{\Omega}} R\cdot \nabla \psi dx\\&=n\cdot \widehat{  R}\psi|_{\partial\overline{\Omega}^v}+n\cdot R\psi|_{\partial\overline{\Omega}}-\int_{\overline{\Omega}^v} \mbox{div} \widehat{  R} \cdot\psi dx-\int_{\overline{\Omega}}\mbox{div} R \cdot\psi dx\\&=n\cdot \widehat{  R}\psi|_{r_w}+n\cdot (R-\widehat{  R})\psi|_{r_s}-\int_{\overline{\Omega}^v} \mbox{div} \widehat{  R} \cdot\psi dx-\int_{\overline{\Omega}}\mbox{div} R \cdot\psi dx\\&=-\int_{\overline{\Omega} \cup \overline{\Omega}^v}\mbox{div} A \cdot\psi dx=0,
		\end{split}
	\end{equation*}
	where we have used  the boundary conditions $n\cdot \widehat{  R}\psi|_{r_w}=0$ and $n\cdot (R-\widehat{  R})\psi|_{r_s}=0$, with $r_w$  the solid boundary and $r_s$ the interface. 
	Hence, we get $\mbox{div} A=0$ in the domain $\overline{\Omega} \cup \overline{\Omega}^v$ in the sense of distributions,
	which together with that the domain $\overline{\Omega} \cup \overline{\Omega}^v$  is simply connected and the weak Poinc\'are lemma, see Theorem IV 4.11 in \cite{tools-incompress-ns}, gives that $A=\nabla\times C.$ When restricted to the vacuum domain $\Omega^v$, we obtain $A=\widehat{  R}=\nabla \times \widehat{  C}$, and denote $\widehat{  C}=C|_{\overline{\Omega}^v}.$
	Combining \eqref{sym-form-int} with \eqref{tranf-sur} and \eqref{last-trans} yields \eqref{sym}, which concludes the proof.
\end{proof}
\begin{rmk}
	Even though it is natural to expect the existence of such $H^2$ solutions, their construction is beyond the focus of the current paper, which will be left for the future.
\end{rmk}

From Lemma \ref{joint}, we can get the following energy identity.
	\begin{lem}\label{vec-xi-grow-lem-first}
	Assume $g$ is a $H^2$ solution to the system \eqref{linear-perturbation-and-boundary-viscosity} with the corresponding jump and boundary conditions, then we can get
	\begin{equation}\label{vec-xi-t-est}
	\begin{split}
	&\int_{\overline{\Omega}} \Big[|Q|^2+\gamma p|\nabla\cdot g |^2\Big]dV
	+\int_{\overline{\Omega}} \Big[(\nabla\times B)\cdot (g^*\times Q)+\nabla\cdot g(g^*\cdot \nabla p)\Big]dx\\
	&\quad+\int_{\overline{\Omega}^v}|\widehat{Q}|^2dx+\|\sqrt{\rho}g_{t}\|^2_{L^2}\\
	&=\int_{\overline{\Omega}} \Big[|Q_0|^2+\gamma p|\nabla\cdot g_0 |^2\Big]dx+\int_{\overline{\Omega}} \Big[(\nabla\times B)\cdot (g_0^*\times Q_0)+\nabla\cdot g_0(g^*_0\cdot \nabla p)\Big]dx\\
	&\quad+\int_{\overline{\Omega}^v}|\widehat{Q}_0|^2dx+\|\sqrt{\rho}g_{0t}\|^2_{L^2},
	\end{split}
	\end{equation}	
	with $Q=\nabla\times (g\times B)$.
\end{lem}
\begin{proof}
Multiplying \eqref{linear-perturbation-and-boundary-viscosity}$_2$ by $g^*_{t}$,  similarly as the proof of Lemma \ref{joint},  we can show	
	\begin{equation}\label{vec-xi-grow}
\begin{split}
& \frac{1}{2}\frac{d}{dt}\|\sqrt{\rho}g_{t}\|^2_{L^2}
=-\frac{1}{2}\frac{d}{dt}\int_{\overline{\Omega}^v}|\widehat{Q}|^2dx
-\frac{1}{2}\frac{d}{dt}\int_{\overline{\Omega}} \Big[|Q|^2+\gamma p |\nabla\cdot g |^2\Big]dx\\
&\quad-\frac{1}{2}\frac{d}{dt}\int_{\overline{\Omega}} \Big[(\nabla\times B)\cdot (g^*\times Q)+\nabla\cdot g(g^*\cdot \nabla p)\Big]dx,
\end{split}
\end{equation}
with $Q=\nabla\times (g\times B)$.	Integrating \eqref{vec-xi-grow} about time, we have \eqref{vec-xi-t-est}.
\end{proof}
 \section*{Acknowledgments}
 {D. Bian is supported by an NSFC Grant (11871005).  Y. Guo is supported  by an NSF Grant (DMS \#1810868). 
 		I. Tice is supported by an NSF CAREER Grant (DMS \#1653161).
%

%

%
\end{document}